\documentclass{amsart}
\usepackage[foot]{amsaddr}

\usepackage{mathrsfs}
\usepackage{amsmath}
\usepackage{amssymb}
\usepackage{mathtools}
\usepackage{float}
\usepackage{xcolor}
\usepackage{enumerate}
\usepackage{subcaption}
\usepackage{tikz}
\usepackage{tikz-cd}

\usepackage{amsmath}
\usepackage{tikz}
\usepackage{color}
\usepackage{array}
\usepackage{multirow}
\usetikzlibrary{fadings}
\usetikzlibrary{patterns}
\usetikzlibrary{shadows.blur}
\usetikzlibrary{shapes}

\usetikzlibrary{arrows, shapes.misc, positioning, decorations.pathreplacing, hobby}

\usepackage[final]{showlabels}
\usetikzlibrary{arrows,automata,positioning}

\newtheorem{theorem}{Theorem}[section]
\newtheorem{lemma}[theorem]{Lemma}
\newtheorem{prop}[theorem]{Proposition}
\newtheorem{corollary}[theorem]{Corollary}

\theoremstyle{definition}
\newtheorem{definition}[theorem]{Definition}

\newtheorem{example}[theorem]{Example}

\theoremstyle{remark}
\newtheorem{remark}[theorem]{Remark}

\numberwithin{equation}{section}

\newcommand{\pres}[3]{\textnormal{#1} \langle #2 \mid #3 \rangle}

\newcommand{\RU}{\mathscr{R}_1}

\newcommand{\MCG}[2]{\Gamma_M(#1, #2)}
\newcommand{\GCG}[2]{\Gamma_G(#1, #2)}

\newcommand{\lr}[1]{\xleftrightarrow{}_{#1}}
\newcommand{\lra}[1]{\xleftrightarrow{\ast}_{#1}}
\newcommand{\xra}[1]{\xrightarrow{\ast}_{#1}}
\newcommand{\xr}[1]{\xrightarrow{}_{#1}}

\newcommand{\fB}{\mathfrak B}

\newcommand{\fP}{\mathfrak P}
\newcommand{\fPe}{\mathfrak P_\varepsilon}
\newcommand{\fR}{\mathfrak R}

\newcommand{\fU}{\mathfrak U}

\newcommand{\T}{\operatorname{Tree}}

\newcommand{\TU}{\textsc{T}\fU}
\newcommand{\TUE}{\widetilde{\textsc{T}}}
\newcommand{\TG}{\textsc{T}\Gamma}
\newcommand{\TGE}{\widetilde{\textsc{T}\Gamma}}
\newcommand{\TUH}{\textsc{T}\fU^{(h)}}
\newcommand{\TUHE}{\widetilde{\textsc{T}\fU^{(h)}}}
\newcommand{\ro}{\mathfrak{1}}

\begin{document}

\title{The Geometry of Special Monoids}

\author{Carl-Fredrik \textsc{Nyberg-Brodda}}
\address{Department of Mathematics, University of East Anglia, Norwich, England, UK}
%\curraddr{Department of Mathematics and Statistics,
%Case Western Reserve University, Cleveland, Ohio 43403}
\email{c.nyberg-brodda@uea.ac.uk}

%    General info
%\subjclass[2010]{20F10, 20F05, 20M05, 20M18, 20F36}

\date{\today}

%\dedicatory{}

%\keywords{}

\begin{abstract}
A monoid is said to be special if it admits a presentation in which all defining relations are of the form $w = 1$. Groups are familiar examples of special monoids. This article studies the geometric and structural properties of the Cayley graphs of finitely presented special monoids, building on work by Zhang and Gray-Steinberg. It is shown that the right Cayley graph $\Gamma$ of a special monoid $M$ is a context-free graph, in the sense of Muller \& Schupp, if and only if the group of units of $M$ is virtually free. This generalises the geometric aspect of the well-known Muller-Schupp Theorem from groups to special monoids. Furthermore, we completely characterise when the monadic second order theory of $\Gamma$ is decidable: this is precisely when the group of units is virtually free. This completely answers for the class of special monoids a question of Kuske \& Lohrey from 2006. As a corollary, we obtain that the rational subset membership problem for $M$ is decidable when the group of units of $M$ is virtually free, extending results of Kambites \& Render. We also show that the class of special monoids with virtually free group of units is the same as the class of special monoids with right Cayley graph quasi-isometric to a tree as undirected graphs. The above results are proven by developing two general constructions for graphs which preserve context-freeness, of independent interest. The first takes a context-free graph and constructs a tree of copies of this graph. The second is a bounded determinisation of the resulting tree of copies. 
\end{abstract}

\maketitle

If $M$ is a monoid admitting a presentation $\pres{Mon}{A}{R_i = 1 \: (i \in I)}$, then we say that $M$ is \textit{special}. Any group is a special monoid, but there are countless examples of special monoids that are not groups. An example of a non-group special monoid appearing in the literature with staggering frequency is the \textit{bicyclic monoid}, defined by $\pres{Mon}{b,c}{bc=1}$. The study of special monoids goes back quite far. The work by Thue \cite{Thue1914} in part deals with what today can be recognised as special monoids with a single defining relation; this definition, however, predates even that of a semigroup, and so only forms part of the proto-history of special monoids.

In a more modern setting, special monoids were first introduced in 1958 by Tse{\u{\i}}tin \cite{Tseitin1958}, who gave them their name as \textit{special associative systems}. This was swiftly followed by the first systematic study of such monoids in 1960, when Adian \cite{Adian1960} proved that the word problem for a finitely presented \textit{homogeneous} (i.e. all relations have the same length) special monoid is decidable if and only if it is decidable for its group of units. Here, the group of units of a monoid $M$ is the subgroup of all two-sided invertible elements of $M$. As a corollary of Adian's result, the word problem for any one-relator special monoid $\pres{Mon}{A}{w=1}$ is decidable, as the word problem for one-relator groups is decidable by the classical result of Magnus \cite{Magnus1932}. The study did not end there; Adian \cite{Adian1966} published a major monograph on the subject in 1966, to which the above one-relation result is often attributed, and Makanin \cite{Makanin1966, MakaninThesis1966}, a student of Adian's, extended the results to all finitely presented special monoids in his thesis, going beyond the homogeneous case. The interested reader is referred to the author's English translation of Makanin's thesis, see \cite{MakaninThesisTranslated}.

Squier \cite{Squier1987} and Otto \& Zhang \cite{Otto1991} made a thorough investigation of special monoids in the context of \textit{string rewriting systems}, and were able to show a number of results, as well as reprove the older results in a more modern framework, such as in \cite{Zhang1992b}. Perhaps most notably, Zhang \cite{Zhang1991} showed that the \textit{conjugacy problem}, that is, the problem of deciding whether two words represent conjugate elements in $M$, once appropriately defined, of a special monoid reduces to the same problem for its group of units. Furthermore, the submonoid of all right invertible elements of a special monoid was shown in \cite{Squier1987, Zhang1992} to be isomorphic to the free product of a finitely generated free monoid by the group of units. Recently, topological methods have found their place in the study of special monoids, too. Gray \& Steinberg \cite{Gray2018} showed that any special monoid $M$ whose group of units satisfies the homological finiteness condition $\operatorname{FP}_n$ will in turn satisfy this condition. The same authors focus primarily on the one-relator case, and show that the cohomological dimension of a one-relator special monoid is either $2$ (if the group of units is a torsion-free one-relator group) or $\infty$ (if it has torsion), and that any one-relator special monoid satisfies the cohomological finiteness condition $\operatorname{FP}_\infty$. As an aside, this has recently been extended by the same authors to show that all one-relation monoids $\pres{Mon}{A}{u=v}$ -- not necessarily special -- satisfy $\operatorname{FP}_\infty$ \cite{Gray2019}. These results indicate that there is prospect for geometric methods in the study of special monoids.

On the other hand, the language-theoretic methods introduced by An\={\i}s\={\i}mov \cite{Anisimov1971} in 1970 to the study of groups have found innumerable applications therein. In his article, An\={\i}s\={\i}mov introduced the idea of measuring the complexity of the word problem of a group by representing this as a set. Given a group $G$, this set -- slightly abusively called the \textit{word problem} for $G$ -- is defined as the set of all words over some fixed generating set of $G$ which represent the identity element. As deciding membership in this set is equivalent to the word problem for $G$, we know already that this set is recursive if and only if $G$ has decidable word problem. It is natural to ask what can be said about groups for which the word problem is a set with strong language-theoretic properties. An\={\i}s\={\i}mov himself proved in his original article that the word problem is a regular language if and only if $G$ is a finite group. Herbst \cite{Herbst1991} showed that the word problem is a \textit{one-counter} language if and only if $G$ is virtually cyclic. 

However, probably the most famous of all results along this line is the \textit{Muller-Schupp theorem} from 1983 \& 1985, due to Muller \& Schupp \cite{MullerSchupp1983} and Dunwoody \cite{Dunwoody1985}, which states that a group is context-free if and only if it is finitely presented virtually free. This result strongly depends on the geometry of the Cayley graph of the groups involved, and these ideas were developed by the same authors further into \textit{context-free graphs} in 1985 in \cite{Muller1985}. These graphs enjoy strong properties; most notably, the monadic second-order theory of any such graph is decidable, by reducing this problem to Rabin's tree theorem \cite{Rabin1969}. Groups whose Cayley graphs are context-free graphs are precisely the context-free groups, and hence precisely the finitely presented virtually free groups. 

Thus, two natural questions appear. The first is how one might characterise special monoids with a context-free word problem; this is beyond the scope of this paper, and is completely answered in further work by the author \cite{NybergBrodda2020b}. The second is to what extent the geometric methods employed in the proof of the Muller-Schupp theorem can be extended to special monoids. This is the main question which has guided the writing of this paper. This question specialises to the question of describing precisely which special monoids have a context-free Cayley graph; this question is entirely resolved by Theorem~\ref{Thm: MULLER-SCHUPP!} -- they are precisely those special monoids whose group of units is virtually free.

The outline of the paper is as follows. We first collect some notation and terminology for the graphs which will appear throughout; in particular, the notion of context-free graphs is defined. This is followed in Section~\ref{Sec: Trees and Determinisations} by \textit{tree constructions} $\T(\Gamma, S)$, a general construction which makes formal the notion of a ``tree of copies'' of  a graph $\Gamma$, and which preserves context-freeness (Proposition~\ref{Prop: Gamma CF => Tree(Gamma) CF}). We also prove one of the main technical results of the paper (Theorem~\ref{Thm: Bounded folding gives CF from T}), which states that under some strong combinatorial assumptions on the graph $\Gamma$ involved, including almost-transitivity, the \textit{determinisation} of $\T(\Gamma, S)$ can be ensured to be context-free. In Section~\ref{Sec: Special Monoids}, we summarise some of the key results regarding special monoids $M$. We also, building on Zhang's work, simplify the study of the group of units of $M$ by considering the set of invertible pieces of the presentation, and prove some minor new results regarding these pieces (Propositions~\ref{Prop: Invertible contains a piece} \& \ref{Prop: Xi generates the right invertibles}). In Section~\ref{Sec: Constructing fU} we introduce a graph $\fU$, which we call the \textit{Sch\"utzenberger graph of the units} of $M$. This graph is quasi-isometric to the group of units of $M$ and enjoys many significant properties. This graph serves as the first step towards a construction of the Cayley graph of $M$. In particular, $\fU$ is a context-free graph if and only if $U(M)$ is virtually free (Theorem~\ref{Thm: Big fat equivalence list}). In Section~\ref{Sec: Constructing fR1}, we introduce the \textit{Sch\"{u}tzenberger graph of $1$}, denoted $\fR_1$, which is the connected component of the identity element in the Cayley graph of $M$. We show (Theorem~\ref{Thm: R_1 is TU determinised}) that $\fR_1$ is isomorphic to the determinisation of a tree of copies of $\fU$. Using the results from the earlier sections, we are then able to conclude that $\fR_1$ is context-free if and only if the group of units $U(M)$ is virtually free, independently of finite generating set chosen for $M$. In Section~\ref{Sec: Building MCG from fR1}, we use the structural results of Gray \& Steinberg to conclude (Corollary~\ref{Cor: fR1 CF iff U(M) VF}) that the Cayley graph of $M$ is context-free if and only if $\fR_1$ is context-free. This, when combined with the results from previous sections, yields the main theorem (Theorem~\ref{Thm: MULLER-SCHUPP!}) of the paper: the Cayley graph of $M$ is context-free if and only if $U(M)$ is virtually free. The final Section~\ref{Sec: Applications and Open Problems} deals with applications of this theorem. The two main applications of independent interest are: the monadic second-order theory of the Cayley graph of $M$ is decidable if and only if $U(M)$ is virtually free (Theorem~\ref{Thm: Dec MSO iff U(M) VF}); and if $U(M)$ is virtually free, then the rational subset membership problem for $M$ is decidable (Corollary~\ref{Cor: VF U(M) => RSMP dec}). Finally, we combine the results of the author on the word problem \cite{NybergBrodda2020b}, to obtain a list (Theorem~\ref{Thm: Full equivalence of context-free M}) of five equivalent properties for characterising special monoids with virtually free group of units. We conclude with some possible directions for further study into the geometry of special monoids. 

We remark that while the results stated in this paper generally deal with \textit{right} Cayley graphs and \textit{right} invertible elements, one can at all places, if one is so inclined, without any loss of validity of the main results replace \textit{right} by \textit{left}.

\section{Background on Graphs}\label{Sec: Background on Graphs}

\subsection{Graphs} 

Before any results can be shown, we will need to fix some terminology regarding the graphs that will appear in the paper. The following definition of a graph follows \cite{Muller1985} and \cite{Kuske2005}, but it is not too dissimilar from any standard treatment of labelled graphs. An \textit{alphabet} is a finite set of symbols. A \textit{labelled graph} $\Gamma$ consists of a set $V = V(\Gamma)$ of vertices, a \textit{label alphabet} $\Sigma$, and a set $E$ of (labelled) \textit{edges}, where $E \subseteq V \times \Sigma \times V$. For every edge $e \in E$, the projection to the first coordinate is called the \textit{origin} $o(e) \in V$ of $e$, the projection to the second coordinate is called the \textit{label} $\ell(e) \in \Sigma$ of $e$, and the projection to the third coordinate is called the \textit{terminus} $t(e) \in V$ of $e$. For $\sigma \in \Sigma$, let $E_\sigma = E \cap (V \times \{ \sigma \} \times V)$ be the set of edges labelled by $\sigma$. This definition of edge does not allow for multiple edges sharing all of origin, terminus, and label. Some constructions, particularly graph quotients, in later sections will often at first appearance seem to introduce pairs of such edges. To enhance readability, we will throughout keep the convention that any such edges are identified as a single edge, and make no further comment on them.

We will often find it useful to associate to $\Gamma$ the \textit{undirected} and \textit{unlabelled} graph
\[
\text{ud}(\Gamma) := \left( V, \bigcup_{\sigma \in \Sigma} \{ (u, \sigma,  v) \mid u \neq v, (u, \sigma, v) \in E_\sigma \: \text{or} \: (v, \sigma, u) \in E_\sigma \}\right)
\]

If $\Sigma$ is the label alphabet of $\Gamma$, then we associate an alphabet $\overline{\Sigma}$ in bijective correspondence with $\Sigma$, denoting this bijection by $\sigma \mapsto \overline{\sigma}$ for all $\sigma \in \Sigma$, with $\Sigma \cap \overline{\Sigma}= \varnothing$. There is a natural labelled graph, adjoint to $\textnormal{ud}(\Gamma)$, obtained by labeling $(u,v)$ by $\sigma$ if $(u,\sigma, v) \in E$, and by $\overline{\sigma}$ if $(v, \sigma, u) \in E$. We denote the resulting \textit{undirected} and \textit{labelled} graph by $\textnormal{lud}(\Gamma)$, and note that $\textnormal{lud}(\textnormal{lud}(\Gamma)) = \textnormal{lud}(\Gamma)$.

The connected components of a graph $\Gamma$ is defined as the connected components of ud$\Gamma$. The \textit{tree-width} of a graph $\Gamma$ is the minimum width among all possible \textit{tree decompositions} of $\Gamma$; this is the same notion of tree decompositions as appears originally in \cite{Robertson1986}. The reader need know nothing about tree-decompositions than what appears in the below paragraph. A tree has tree-width $1$. The \textit{strong tree-width} of a graph $\Gamma$ is the minimum width among all possible \textit{strong tree decompositions} of $\Gamma$. Both these notions were introduced in \cite{Seese1985}, and are explained lucidly in \cite{Kuske2005}. The only property of strong tree decompositions needed is the following, cf. \cite{Kuske2005}: if a graph $\Gamma$ has strong tree-width $\leq k$, then any induced subgraph of $\Gamma$ has strong tree-width $\leq k$. 

We say that a graph $\Gamma$ is \textit{locally finite} if for every vertex $v$ there are only finitely many edges with $v$ as origin or terminus. We say that a graph $\Gamma$ has \textit{bounded in-degree} if there exists $K \geq 0$ such that for every vertex $v$ of $\Gamma$, there are at most $K$ edges with $v$ as terminus; analogously, $\Gamma$ has \textit{bounded out-degree} if there exists $K \geq 0$ such that for every vertex $v$ of $\Gamma$, there are at most $K$ edges with $v$ as origin. We say that $\Gamma$ has \textit{bounded degree} if it has bounded in-degree and bounded out-degree. 

We will frequently reference walks in graphs; if $p$ is a walk
\[
u_0 \xrightarrow{\sigma_1} u_1 \xrightarrow{\sigma_2} \cdots \xrightarrow{\sigma_n} u_n
\]
where $u_i \xrightarrow{\sigma_{i+1}} u_{i+1}$ is meant to indicate that $(u_i, u_{i+1}) \in E_{\sigma_{i+1}}$, then we say that $p$ has \textit{walk label} $\ell(p) = \sigma_1 \cdots \sigma_n$, and we say that $p$ is the walk $u_0 \xrightarrow{\ell(p)} u_n$. At times, this may be abbreviated to simply read $p : u_0 \xrightarrow{\ell(p)} u_n$. If all vertices $u_i$ are pairwise distinct, except possibly $u_0$ and $u_n$, then we say that $p$ is a \textit{path}, and we will accordingly refer to its walk label as its \textit{path label}.

\subsection{Ends of graphs}

An incredibly useful notion in the study of the coarse geometry of a graph is that of \textit{ends}. For this, we follow \cite{Muller1985}. Let $\Gamma$ be a connected labelled graph of bounded degree. We will distinguish a vertex $\ro \in V(\Gamma)$ and say that $\ro$ is the \textit{root} of $\Gamma$. If $v$ is any vertex of $\Gamma$, then we use $|v|_\Gamma$ to denote the length of a shortest (undirected) walk from $\ro$ to $v$ in $\text{ud}(\Gamma)$. By $\Gamma^{(n)}$ we mean the subgraph of $\Gamma$ consisting of all the vertices and edges which are connected to $\ro$ by an undirected walk of length less than $n$; in particular $\Gamma^{(0)}$ is empty, $\Gamma^{(1)}$ consists of $\ro$, and $\Gamma^{(2)}$ consists of $\ro$, its neighbours, and all edges connecting these vertices. As in the theory of ends in e.g.\ \cite{Cohen1972} or \cite{Rabin1969}, the connected components of $\Gamma \setminus \Gamma^{(n)}$ will be the central objects of study. If $C$ is a connected component of $\Gamma \setminus \Gamma^{(n)}$, then we say that a \textit{frontier point} of $C$ is a vertex $u$ of $C$ such that $|u|_{\text{ud} \Gamma} = n$. If $v$ is a vertex of $\Gamma$ with $|v|_{\text{ud} \Gamma} = n$, then we use $\Gamma(v)$ to denote the component of $\Gamma \setminus \Gamma^{(n)}$ which contains $v$. The set of frontier points of $\Gamma(v)$ will be denoted by $\Delta(v)$; this set is always finite (but possibly unbounded in $v$) as $\Gamma$ has bounded degree. Let $u, v \in V(\Gamma)$. An \textit{end-isomorphism} between the two subgraphs $\Gamma(u)$ and $\Gamma(v)$ is a mapping $\psi$ between $\Gamma(u)$ and $\Gamma(v)$ such that 
\begin{enumerate}
\item $\psi$ is a label-preserving graph isomorphism, and
\item $\psi$ maps $\Delta(u)$ onto $\Delta(v)$.
\end{enumerate}

We will write $\Gamma(u) \sim \Gamma(v)$ if there exists some end-isomorphism $\psi : \Gamma(u) \to \Gamma(v)$. 

\begin{definition}
A connected labelled graph $\Gamma$ of bounded degree is said to be \textit{context-free} if: 
\begin{enumerate}
\item $\Gamma$ has bounded degree;
\item $\Gamma(v) \mid v \in V(\Gamma) \} / \sim$ is finite.
\end{enumerate} 
Furthermore, a labelled graph $\Gamma$ of bounded degree that is the union of finitely many connected graphs is said to be context-free if all its connected components are context-free.
\end{definition}

We now give one example together with a non-example of a context-free graph, to illustrate some of the considerations of importance, as well as consolidate the definitions. The interested reader may construct many more examples of both kinds without much difficulty. 

\begin{example}[A context-free graph]
Let $\Gamma$ be the graph obtained from the following procedure: take a triangle graph, and attach a single edge to each of its vertices. To each of these edges, attach an isomorphic copy of the original triangle, and repeat; the graph $\Gamma$ is the colimit of the sequence of graphs obtained. See Figure~\ref{Fig: C2C3 is virtually free}, in which the resulting graph is drawn out without any edge labels. We note that this graph is very closely related to the Cayley graph of the virtually free group $\pres{Gp}{a, b}{a^2 = b^3 = 1} \cong C_2 \ast C_3$.
\end{example}

Thus context-freeness captures the idea that if one traverses the graph from the identity, travelling outwards, one eventually encounters graphs which one already has seen.

\begin{figure}
\begin{center}
\hspace*{-1cm}\includegraphics[scale=0.8]{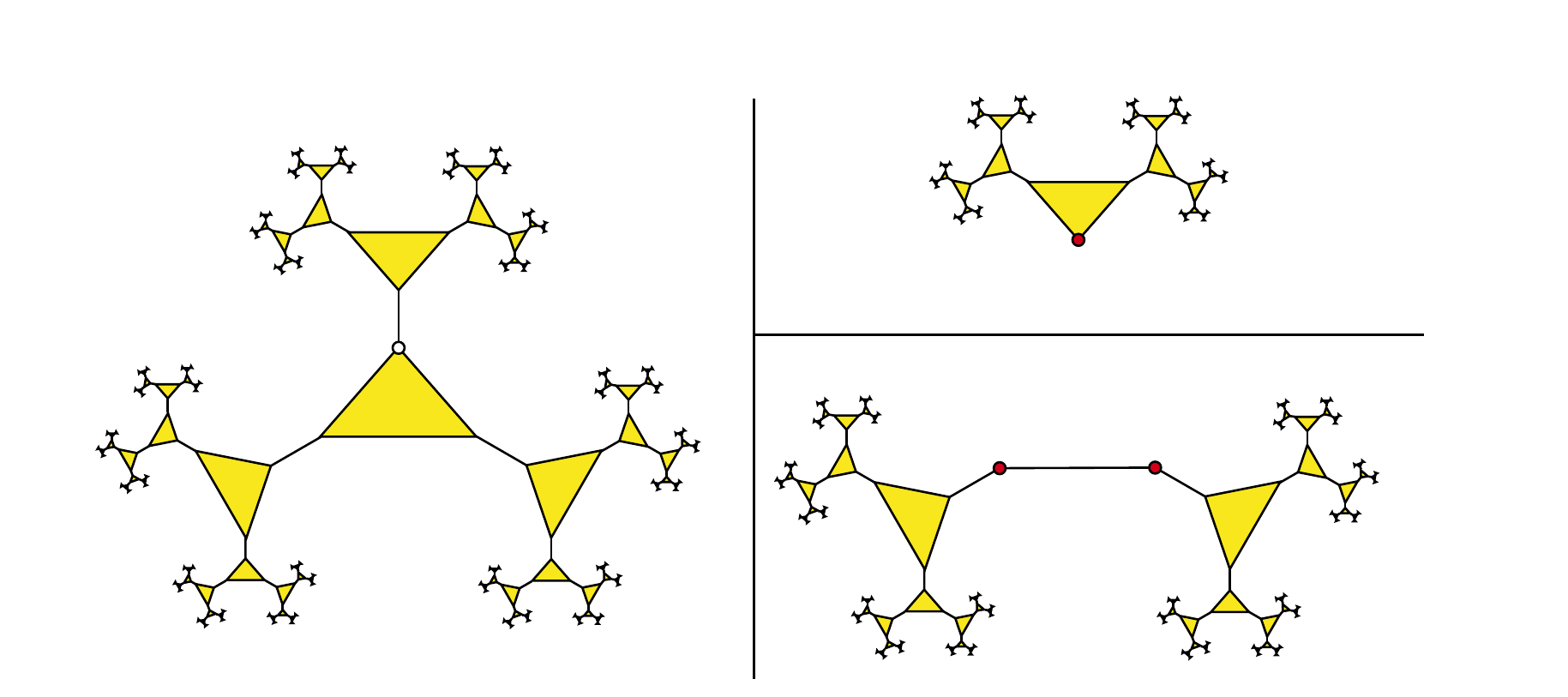}
\end{center}
\caption{Left: A context-free graph $\Gamma$, with root $\ro$ central and enlarged. Right: Two of the three end-isomorphism classes of $\Gamma$, with frontier points marked in red. The third end-isomorphism class is represented by $\Gamma(\ro)$, and is isomorphic to $\Gamma$. Note that all edge labels are intentionally suppressed; it is assumed that all triangles have the same labels, as do all single edges.}
\label{Fig: C2C3 is virtually free}
\end{figure}

\begin{example}[A graph which is not context-free]
Let $\Gamma$ be the infinite two-dimensional grid, with two different types of edge labels; see Figure~\ref{Fig: ZxZ grid is not CF}. The root $\ro$ is the central enlarged vertex. Some graphs $\Gamma^{(n)}$ are drawn out. In all of these graphs, the red vertices indicate vertices at distance $n$  from the root $\ro$. Then $\Gamma$ is not a context-free graph, as the number of frontier points of $\Gamma \setminus \Gamma^{(n)}$ equals the number of vertices at distance $n$ from $\ro$. But this number grows quadratically, and hence unboundedly in $n$, so there cannot be only finitely many end-isomorphism classes of ends of $\Gamma$. We note that this graph is isomorphic to the Cayley graph of the group $\pres{Gp}{a,b}{[a, b] = 1} \cong \mathbb{Z} \times \mathbb{Z}$, which is not virtually free.
\end{example}

\begin{figure}
\begin{center}
\hspace*{-1cm}\includegraphics[scale=0.7]{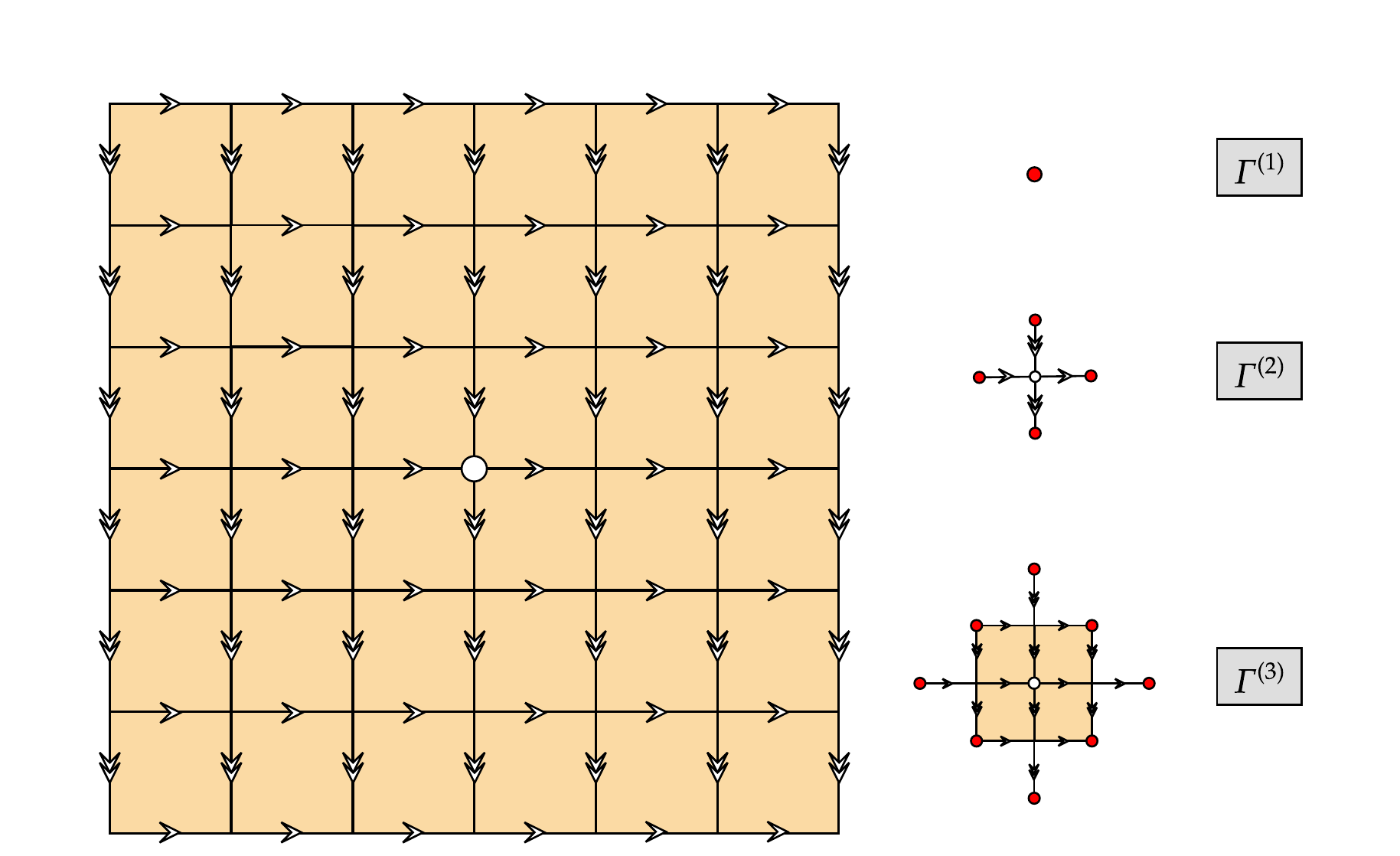}
\end{center}
\caption{An illustration of why the infinite two-dimensional grid is not a context-free graph. We remark that the filled-in squares are purely to make the picture clear, and have no significance to the graph itself. Left: the graph $\Gamma$, with root $\ro$ in the center. Right: three complements of ends of $\Gamma$, showing that the number of frontier points of $\Gamma \setminus \Gamma^{(n)}$ grows unboundedly in $n$.}
\label{Fig: ZxZ grid is not CF}
\end{figure}

Note that as the two graph-theoretic tools used above for defining context-free graphs -- namely connected components and distances -- are used in the undirected sense, lud$(\Gamma)$ is context-free if and only if $\Gamma$ is context-free. We note that by \cite[Corollary~2.7]{Muller1985} a context-free graph remains context-free independent of the choice of root; this result will often be implicitly assumed throughout. In particular, some statements will be easier to state without explicitly rooting the graphs involved; the reader can to such statements always add the sentence ``independently of root chosen''. Furthermore, if we speak of a context-free graph, then to save ink we implicitly assume this graph is labelled and has bounded degree.

A final definition which occasionally is useful is that of a \textit{second-level subgraph} of a context-free graph. Let $\Gamma$ be a context-free graph rooted at $\ro$, and let $\{ \Gamma_0, \Gamma_1, \dots, \Gamma_k\}$ be a complete list of representatives of the end-isomorphism classes of $\Gamma$. Assume without loss of generality that $\Gamma_0 = \Gamma(\ro)$. Let $\Gamma_i$ be one of these representative graphs, with $\Delta_i$ as its set of frontier points. If the edges of $\Gamma_i$ that are incident to some vertex of $\Delta_i$ are deleted, then there remains a finite union of connected subgraphs of $\Gamma_i$. These subgraphs are called the \textit{second-level subgraphs} of $\Gamma_i$. It may be the case that $\Gamma_0$ is not the second-level subgraph of any $\Gamma_i$. If this is the case, then we declare $\Gamma_0$ to nevertheless be a second-level subgraph.

We note that for every $\Gamma_i$ there exists some $\Gamma_j$ such that $\Gamma_i$ is a second-level subgraph of $\Gamma_j$, and every second-level subgraph is isomorphic to some $\Gamma_i$. Thus the notion of second-level subgraphs formalises the notion of ``seeing'' elements of $\{ \Gamma_0, \Gamma_1, \dots, \Gamma_k \}$ as one travels in $\Gamma$ outward from the root. 
\subsection{Cayley graphs}

There remain a handful of standard definitions to make. The first is that of a \textit{quasi-isometry of undirected graphs} -- a quasi-isometry is a quasi-isometry in the usual sense, but the graphs considered are rather the geometric realisations of the graphs above. This distinction is important, but to prevent cumbersome notation, we omit writing this explicitly in all places. Hence we will occasionally speak of quasi-isometries of graphs. The second definition is that of a Cayley graph of a monoid, which was mentioned in passing in both of the above examples. The (right) \textit{monoid Cayley graph} $\MCG{M}{A}$ of a monoid $M$ together with a generating set $A$ for $M$ is the labelled graph with vertex set $M$ and edges $u \xrightarrow{a} v$, i.e.\ $(u, v) \in E_a$, such that $u \cdot \pi(a) =v$ in $M$. Here $\pi : A^\ast \to M$ denotes the natural homomorphism from the free monoid on $A$ to $M$. The \textit{group Cayley graph} $\GCG{G}{A}$ of a group $G$ with a generating set $A$ for $M$ is defined analogously, but with two additional assumptions: that $A$ be placed in involutive correspondence with an alphabet $A^{-1}$ such that $A \cap A^{-1} = \varnothing$, and such that for every edge $(u, v) \in E_a$, there exists an edge $(v, u) \in E_{a^{-1}}$. In particular, the label alphabet of $\GCG{G}{A}$ is $A \sqcup A^{-1}$, and $\pi(a^{-1}) = \pi(a)^{-1}$. 

\section{Tree Constructions}\label{Sec: Trees and Determinisations}

\subsection{Trees of copies}

For ease of notation, throughout this section all graphs will be labelled with alphabet $A$. Let $\Gamma$ be a connected labelled graph rooted at $\ro$ with bounded degree. Let $\{ \Gamma_0, \Gamma_1, \dots \}$ be a (at most countable) set of representatives of the end-isomorphism classes of $\Gamma$, with $\Gamma_0 := \Gamma(\ro)$ for notational convenience. For every $i$ appearing in this list, let $\Delta_i$ be the set of frontier points of $\Gamma_i$, and let $F(\Gamma) = \cup_i \Delta_i$, where the union is taken over all $i$ appearing in the earlier list. We assume $\ro$ is the representative for $\Gamma_0$, so $\ro \in F(\Gamma)$. For example, in Figure~\ref{Fig: C2C3 is virtually free}, we can take $F(\Gamma)$ as a set of four vertices: the root $\ro$, together with the three vertices around any triangle other than the central one. Thus $F(\Gamma)$ is a set of representatives of frontier points of the ends of $\Gamma$. 

For any subset $S \subseteq F(\Gamma)$, we will now define a graph $\T(\Gamma, S)$, and will show that if $\Gamma$ is a context-free graph then so too is $\T(\Gamma, S)$, for any (non-trivial) choice of $S$. We first give some motivation. Intuitively, this graph will capture the idea of ``a tree of copies of $\Gamma$'', where the ``branching'' of the tree takes place only on vertices which are end-equivalent to a vertex of $S$. If one were to ``branch'' at completely arbitrary places in $\Gamma$ and thus construct a tree of copies, then it is clear that the resulting graph could be highly non-context-free, even if $\Gamma$ itself was chosen to be context-free. However, the condition that we only ``branch'' at vertices which are end-equivalent to a vertex from $S$, which is necessarily a finite set if $\Gamma$ is context-free, means that the resulting tree of copies will only have finitely many different types of ``branching behaviours'', which gives some initial justification for the claim that this tree of copies should be a context-free graph.

We now make this formal. Fix some $S \subseteq F(\Gamma)$ with $\ro \not\in S$. By definition of $F(\Gamma)$, every vertex $v \in V(\Gamma)$ is such that there exists a unique $i$ with $f \in \Delta_i \subseteq F$ and an end-isomorphism $\psi : \Gamma(v) \to \Gamma_i$ with $\psi(v) = f$. In particular, there exists a set $V_S \subseteq V(\Gamma)$ which is the collection of all $v \in V(\Gamma)$ with an end-isomorphism such that $\psi(v) \in S$. We remark that if $\Gamma$ is context-free and infinite, then whereas $S$ is necessarily finite, the set $V_S$ will in general be infinite, unless $S = \varnothing$. Furthermore, $\ro \not\in V_S$.

We inductively define graphs $\T_n(\Gamma, S)$ and associated distinguished sets of vertices $V_n \subseteq V(\T_n(\Gamma, S))$. Define $\T_0(\Gamma, S) := \Gamma$, and $V_0 := V_S$. Assume that for some $n \geq 0$ the graphs $\T_k(\Gamma, S)$ and $V_k$ have been defined for some $0 \leq k \leq n$. Then $\T_{n+1}(\Gamma, S)$ is defined as the graph obtained from attaching a copy $\Gamma_v$ of $\Gamma$ to every vertex $v \in V_n$ inside $\T_n(\Gamma, S)$, identifying the root of $\Gamma_v$ with $v$. Denote by $V_{v,S}$ the copy of the subset $V_S$ inside $V(\Gamma_v)$. Then we define 
\[
V_{n+1} := \bigcup_{v \in V_n} V_{v,S} \subseteq V(\T_{n+1}(\Gamma, S)).
\]
In other words, $V_{n+1}$ is defined by taking the union of all the newly added copies of $V_0$. Note that the graphs 
\[
\T_0(\Gamma, S) \subseteq \T_1(\Gamma, S) \subseteq \cdots \subseteq \T_k(\Gamma, S) \subseteq \cdots
\] form a directed system, and in particular their directed colimit exists. If $S$ is non-empty, then the inclusions are strict.

\begin{definition}
For a graph $\Gamma$ and a set $S \subseteq F(\Gamma)$ of vertices of $\Gamma$, we define the \textit{tree of copies of $\Gamma$ with respect to $S$} as \[\T(\Gamma, S) := \varinjlim_i \T_i(\Gamma, S) = \bigcup_i \T_i(\Gamma, S).\]
The set $\bigcup_i V_i$ is called the set of \textit{branch points} of $\T(\Gamma, S)$. 
\end{definition}

We will almost exclusively root $\T(\Gamma, S)$ at $\ro \in \Gamma = \T_0(\Gamma, S)$. An example of $\T(\Gamma, S)$ for when $\Gamma$ is a triangle graph and $S$ consists of the two non-root vertices is shown in Figure~\ref{Fig: A triangle graph tree of copies}.

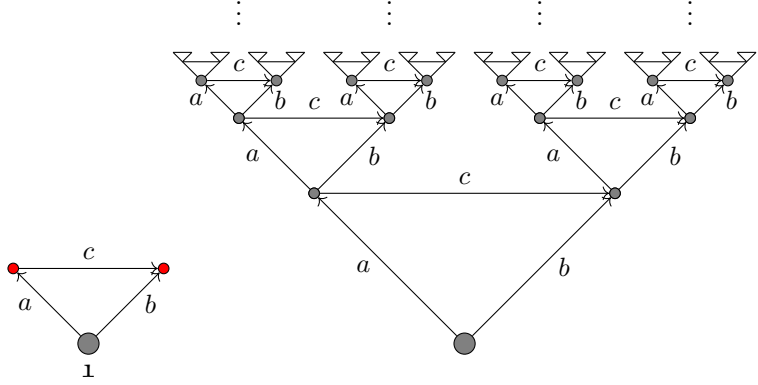
\begin{figure}
\begin{tikzpicture}
\node (v0)[label=below:$\ro$][circle, draw, fill=black!50, inner sep=0pt, minimum width=8pt] at (0,0) {};
\node (a1)[circle, draw, fill=red!100, inner sep=0pt, minimum width=4pt] at (-1, 1) {};
\node (a2)[circle, draw, fill=red!100, inner sep=0pt, minimum width=4pt] at (1, 1) {};

\path[->]
(v0)  edge node[left = 0.1]{$a$}         (a1)
(v0)  edge node[right=0.1]{$b$}         (a2)
(a1)  edge node[above]{$c$}         (a2);

%%%%%%%%%%%%%% TREE OF COPIES %%%%%%%%%%%%%%%%%

\node (v0)[circle, draw, fill=black!50, inner sep=0pt, minimum width=8pt] at (5,0) {};
\node (a1)[circle, draw, fill=black!50, inner sep=0pt, minimum width=4pt] at (3, 2) {};
\node (a2)[circle, draw, fill=black!50, inner sep=0pt, minimum width=4pt] at (7, 2) {};
\path[->]
(v0)  edge node[left = 0.1]{$a$}         (a1)
(v0)  edge node[right=0.1]{$b$}         (a2)
(a1)  edge node[above]{$c$}         (a2);

%%%%%%%%%%% SECOND LAYER%%%%%%%%%%%%%%%%%%%%%%%%%
\node (v0)[circle, draw, fill=black!50, inner sep=0pt, minimum width=4pt] at (3,2) {};
\node (a1)[circle, draw, fill=black!50, inner sep=0pt, minimum width=4pt] at (2,3) {};
\node (a2)[circle, draw, fill=black!50, inner sep=0pt, minimum width=4pt] at (4,3) {};
\path[->]
(v0)  edge node[left = 0.1]{$a$}         (a1)
(v0)  edge node[right=0.1]{$b$}         (a2)
(a1)  edge node[above]{$c$}         (a2);

\node (v0)[circle, draw, fill=black!50, inner sep=0pt, minimum width=4pt] at (7,2) {};
\node (a1)[circle, draw, fill=black!50, inner sep=0pt, minimum width=4pt] at (6,3) {};
\node (a2)[circle, draw, fill=black!50, inner sep=0pt, minimum width=4pt] at (8,3) {};
\path[->]
(v0)  edge node[left = 0.1]{$a$}         (a1)
(v0)  edge node[right=0.1]{$b$}         (a2)
(a1)  edge node[above]{$c$}         (a2);

%%%%%%%%%%%%%%%%%%%%%%% FOURTH LAYER %%%%%%%%%%%%%%%%%%%%%%%%%%%

\foreach \x in {0,...,7}{
    \draw (1.5+\x, 3.5) edge (1.25+\x, 3.75);
    \draw (1.5+\x, 3.5) edge (1.75+\x, 3.75);
	\draw (1.25+\x, 3.75) edge (1.75+\x, 3.75);}
	
\foreach \x in {0,...,15}{
    \draw (1.25+\x*0.5, 3.75) edge (1.125+\x*0.5, 3.875);
    \draw (1.25+\x*0.5, 3.75) edge (1.375+\x*0.5, 3.875);
	\draw (1.125+\x*0.5, 3.875) edge (1.375+\x*0.5, 3.875);}

%%%%%%%%%%% THIRD LAYER%%%%%%%%%%%%%%%%%%%%%%%%%
\node (v0)[circle, draw, fill=black!50, inner sep=0pt, minimum width=4pt] at (2,3) {};
\node (a1)[circle, draw, fill=black!50, inner sep=0pt, minimum width=4pt] at (1.5,3.5) {};
\node (a2)[circle, draw, fill=black!50, inner sep=0pt, minimum width=4pt] at (2.5,3.5) {};
\path[->]
(v0)  edge node[left = 0.1]{$a$}         (a1)
(v0)  edge node[right=0.1]{$b$}         (a2)
(a1)  edge node[above]{$c$}         (a2);

\node (v0)[circle, draw, fill=black!50, inner sep=0pt, minimum width=4pt] at (4,3) {};
\node (a1)[circle, draw, fill=black!50, inner sep=0pt, minimum width=4pt] at (3.5,3.5) {};
\node (a2)[circle, draw, fill=black!50, inner sep=0pt, minimum width=4pt] at (4.5,3.5) {};
\path[->]
(v0)  edge node[left = 0.1]{$a$}         (a1)
(v0)  edge node[right=0.1]{$b$}         (a2)
(a1)  edge node[above]{$c$}         (a2);

\node (v0)[circle, draw, fill=black!50, inner sep=0pt, minimum width=4pt] at (6,3) {};
\node (a1)[circle, draw, fill=black!50, inner sep=0pt, minimum width=4pt] at (5.5,3.5) {};
\node (a2)[circle, draw, fill=black!50, inner sep=0pt, minimum width=4pt] at (6.5,3.5) {};
\path[->]
(v0)  edge node[left = 0.1]{$a$}         (a1)
(v0)  edge node[right=0.1]{$b$}         (a2)
(a1)  edge node[above]{$c$}         (a2);

\node (v0)[circle, draw, fill=black!50, inner sep=0pt, minimum width=4pt] at (8,3) {};
\node (a1)[circle, draw, fill=black!50, inner sep=0pt, minimum width=4pt] at (7.5,3.5) {};
\node (a2)[circle, draw, fill=black!50, inner sep=0pt, minimum width=4pt] at (8.5,3.5) {};
\path[->]
(v0)  edge node[left = 0.1]{$a$}         (a1)
(v0)  edge node[right=0.1]{$b$}         (a2)
(a1)  edge node[above]{$c$}         (a2);

\draw (4,4.5) node{$\vdots$};
\draw (2,4.5) node{$\vdots$};
\draw (6,4.5) node{$\vdots$};
\draw (8,4.5) node{$\vdots$};

\end{tikzpicture}
\caption{A triangle graph $\Gamma$ (left), and a tree of copies of $\Gamma$ (right). The set $S$ of attachment points consists of two frontier points (onto which copies are attached), and are marked as red vertices. It is clear that the tree of copies is a context-free, having only two distinct end-spaces; that of the root vertex, and the union of two such spaces, connected by an edge. For all $i \geq 0$, the set $V_i$ (considered as a subset of $\T(\Gamma, S)$) consists of the $2^{i-1}$ vertices at undirected distance $i-1$ from the root.}
\label{Fig: A triangle graph tree of copies}
\end{figure}

Note that if we were to permit $\ro \in S$, then $\T(\Gamma, S)$ would in general not have bounded degree, as we would then in the above definitions have that $\ro \in V_n$ for all $n \geq 0$, and infinitely many copies of $\Gamma$ will be attached to $\ro$. However, without this restriction, it is easy to see that $\T(\Gamma, S)$ has bounded degree. Thus, to simplify notation, we will call a subset $S$ of $F(\Gamma)$ a set of \textit{attachment points} (of $\Gamma$) if we have $\ro \not\in S$. We note again that if $\Gamma$ is context-free, then $S$ is necessarily a finite set, as $F(\Gamma)$ is. Hence if $\Gamma$ is a context-free graph, then $\T(\Gamma, S)$ can be encoded by a finite amount of information, viz. the ends of $\Gamma$ together with $S$. In fact, we will show that in this setting, $\T(\Gamma, S)$ is actually a context-free graph. Before this, we make a simple observation about $\T(\Gamma, S)$. 

\begin{prop}\label{Prop: Tree(Gamma, S) is a tree, tuples}
Let $\Gamma$ be a graph, and $S \subseteq F(\Gamma)$ a set of attachment points. Then the set of vertices of $\T(\Gamma, S)$ is in bijective correspondence with all finite tuples
\[
(u_0, u_1, \dots, u_k) 
\]
where $u_i \in V_0$ for $0 \leq i < k$; and where $u_k \in V(\Gamma) \setminus \{ \ro \}$ if $k > 0$, or else $u_k \in V(\Gamma)$. Furthermore, for two vertices $u, w \in \T(\Gamma, S)$ with corresponding tuples $(u_0, u_1, \dots, u_k)$ and $(w_0, w_1, \dots, w_n)$, respectively, there is an edge $u \xrightarrow{a} w$ if and only if one of the following holds: 
\begin{enumerate}
\item $n = k$, $u_i = w_i$ for all $0 \leq i \leq k-1$, and $(u_k \xrightarrow{a} w_n) \in E(\Gamma)$; or
\item $n = k+1$, $u_i = w_i$ for all $0 \leq i \leq k-1$, $u_k \in V_0$, and $(\ro \xrightarrow{a} w_n) \in E(\Gamma)$. 
\end{enumerate}
\end{prop}
\begin{proof}
Let $v \in \T(\Gamma, S)$. Then there exists a minimal $k \geq 0$ such that $v \in V(\T_k(\Gamma, S))$. Hence, we will for all $k \geq 0$ define a bijection $\phi_k$ from the set of vertices of $\T_k(\Gamma, S)$ to the set of tuples of the above form of length at most $k$. This bijection will then extend to a bijection $\phi$ from the set of vertices of $\T(\Gamma, S)$ of the desired form when taking the directed colimit.

If $k=0$, then $v \in V(\T_0(\Gamma))$; this set is in bijective correspondence with $V(\Gamma)$ by definition, and we will define $\phi_0$ to be this bijection.

Assume, for strong induction, that $\phi_k$ is defined for all $0 \leq k \leq n$ for some $n \geq 0$. Let $v \in \T(\Gamma, S)$ be such that $v \in V(\T_{n+1}(\Gamma, S))$ with $n+1$ minimal. Since $n+1$ is minimal and $\ro \not\in S$, we have that there exists a unique $u_n \in V_n$ such that $v \in \Gamma_{u_n}$, and furthermore since $\Gamma_{u_n}$ is a copy of $\Gamma$, there exists a vertex $v^\prime \in V(\Gamma)$ uniquely determined by $v \in V(\Gamma_{u_n})$. By strong induction, we can find a unique $n$-tuple $\phi_n(u_n)$ representing $u_n$. We then define $\phi_{n+1}(v) := (\phi_n(u_n), v^\prime)$. In this resulting tuple, the first $n$ entries will be chosen as elements of $V_0$, and the final will be any element of $V(\Gamma)$, as needed. This completes the claim concerning the vertices by induction.

Now, for the edges, it is clear that the only possible adjacent vertices are either such that they belong to the same added copy of $\Gamma$; or one is an element of $V_n$ for some $n \geq 0$, and the other is connected to the root of the copy of $\Gamma$ attached to the former. But clearly, the first case happens if and only if criterion $(1)$ of the statement of the proposition holds; the second case happens if and only if criterion $(2)$ holds. This completes the proof.
\end{proof}

Because of this proposition, we will often consider this above prescribed bijection as invisible, and say that vertices of $\T(\Gamma, S)$ with $S \neq \varnothing$ are equal to a tuple of elements of the above prescribed form. We shall phase this notation in throughout this section. Now, a natural definition in light of the proposition above is the following.

\begin{definition}
The \textit{depth function} $d \colon V(\T(\Gamma, S)) \to \mathbb{N}$ is defined as follows: if $v \in V(\T(\Gamma, S))$ is uniquely associated to the tuple $(u_0, u_1, \dots, u_k)$ by Proposition~\ref{Prop: Tree(Gamma, S) is a tree, tuples}, then $d(v) := k$. 
\end{definition}

\begin{figure}
\includegraphics[scale=0.7]{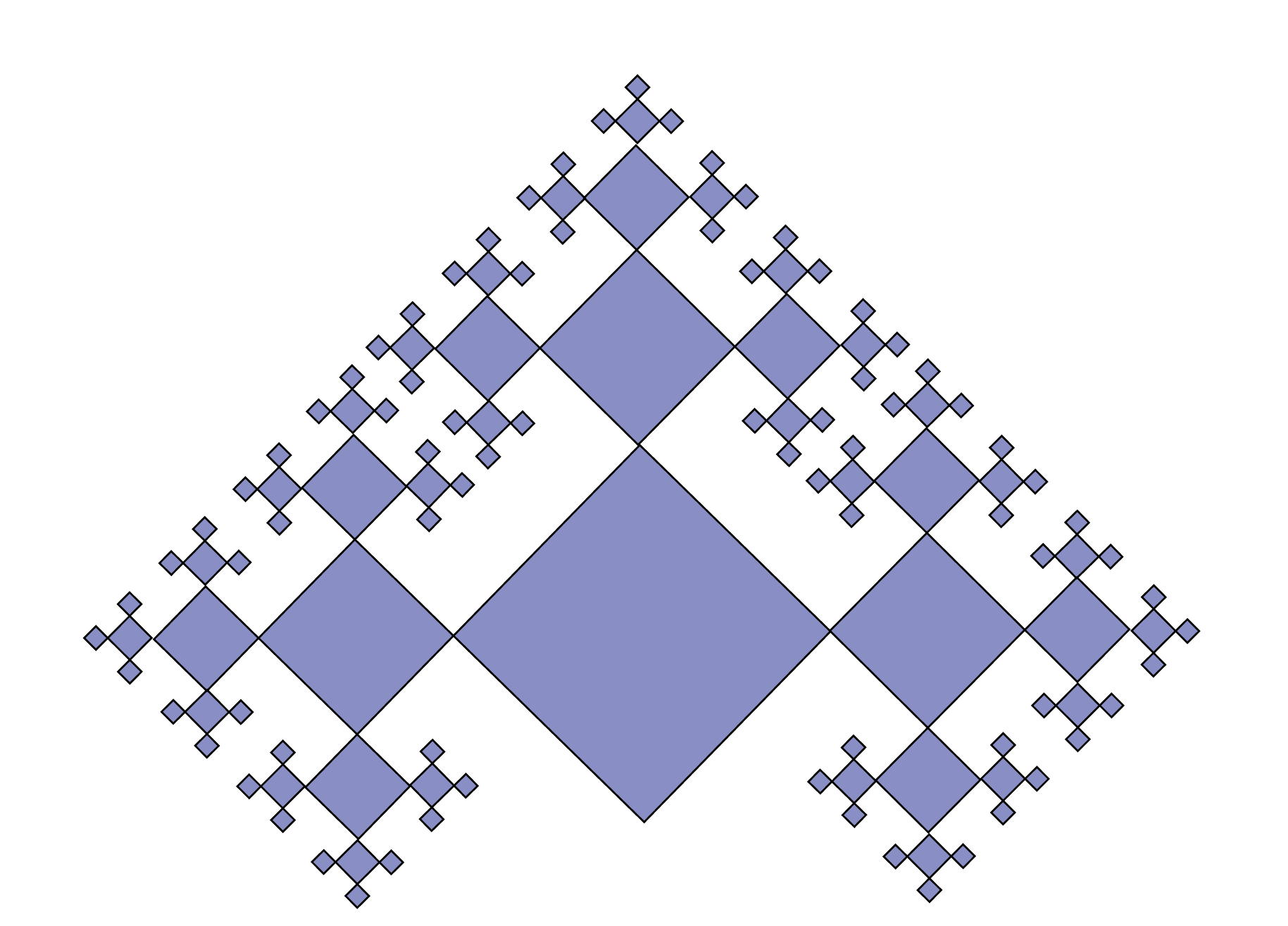}
\caption{A tree of copies of a square graph, with branch points everywhere except the root vertex (the only vertex of degree $2$). This graph is context-free (regardless of edge labels) as it has two distinct end-spaces; that of the root, and that of either of the two vertices at distance $1$ from the root. Both of these repeat infinitely often.}
\end{figure}

We note that the edge case, and indeed the base cases for many inductive arguments on the depth, that $d(u) = 0$ for all $u \in V(\Gamma)$, where we canonically identify $u$ with the $1$-tuple $(u)$. In particular, $d(\ro) = 0$. 

By the tree-like nature of the construction, we additionally have the following, which is one of the main reasons for introducing constructions of this form. We remind the reader that this proposition captures the main reason of choosing $S \subseteq F(\Gamma)$ to define the tree-construction.

\begin{prop}\label{Prop: Gamma CF => Tree(Gamma) CF}
If $\Gamma$ is context-free and $S$ a set of attachment points, then $\T(\Gamma, S)$ is context-free.
\end{prop}
\begin{proof}
Assume that $\Gamma$ is context-free and that $S \subseteq F(\Gamma)$ with $\ro \not\in S$. Then in particular $S \subseteq F(\Gamma)$ is finite. If $S$ is empty, then $\T(\Gamma, S) = \T(\Gamma, \varnothing) = \Gamma$, and there is nothing to show. Hence, assume that $S$ is non-empty. Since $\ro \not\in S$, it is clear that $\T(\Gamma, S)$ has bounded degree, as every vertex has a copy of $\Gamma$ attached at most once. 

We will now show that for any $v \in V(\T(\Gamma, S))$ there exists $v^\prime \in V(\T_0(\Gamma, S))$ such that $\T(\Gamma, S)(v) \sim \T(\Gamma, S)(v^\prime)$. To establish this claim, let $v \in V(\T(\Gamma, S))$ be any vertex. Then there exists a minimal $n \geq 0$ such that $v \in V(\T_n(\Gamma, S))$. If $n = 0$, then we may take $v^\prime = v$, and there is nothing to show. Hence, assume that $n > 0$. By Proposition~\ref{Prop: Tree(Gamma, S) is a tree, tuples}, there exists a unique tuple $(u_0, u_1, \dots, u_n)$ with $u_i \in V_0$ for $0 \leq i < n$, and $u_n \in V(\Gamma)$. We claim that we can take $v^\prime = u_n$.

First, note that $v$ belongs to a copy $\Gamma_{u_{n-1}}$ of $\Gamma$ with a labelled graph isomorphism $\phi : \Gamma_{u_{n-1}} \to \Gamma$ taking $v$ to $u_n$, where $\phi$ is simply the identity mapping. This isomorphism hence extends to an end-isomorphism 
\[\phi^\prime : \Gamma_{u_{n-1}}(v) \to \Gamma(u_n) \]
Since all walks in $\T(\Gamma, S)$ from $\ro$ to vertices $w \in V(\Gamma_{u_{n-1}})$ must pass through $u_{n-1}$, applying Proposition~\ref{Prop: Tree(Gamma, S) is a tree, tuples} to such $w$, we see that the tuple corresponding to $w$ must be $(u_0, u_1, \dots, u_{n-1}, w_n)$ for some $w_n \in V(\Gamma)$. In particular, if $u^\prime$ denotes the vertex corresponding to the tuple $(u_0, u_1, \dots, u_{n-1})$, then 
\[
|w|_{\T(\Gamma, S)} = |u^\prime|_{\T(\Gamma, S)} + |w_n|_{\Gamma}.
\]
In particular, our end-isomorphism $\phi^\prime$ extends to an end-isomorphism $\phi^{\prime \prime}$ between the embedded $\Gamma_{u_{n-1}}(v)$ and $\Gamma(u_n)$ inside the graph $\T(\Gamma, S)$, where $\Gamma$ is now identified with the canonically embedded copy of $\Gamma$ inside $\T(\Gamma, S)$, i.e.\ $\Gamma_{\ro}$. 

We now extend this end-isomorphism one final time. If $w_s$ is a vertex of the embedded $\Gamma_{u_{n-1}}(v)$ end-isomorphic to an element of $S$ in $\Gamma_{u_{n-1}}$, then this will be carried to an element $\phi^{\prime \prime}(w_s)$ end-isomorphic to an element of $S$ by $\phi^{\prime\prime}$, as $S$ is a set of representatives of frontier points. This is a key step, and the reason for choosing $S$ as it is. In particular, the second-level subgraphs of $w_s$ and $\phi^{\prime \prime}(w_s)$ are the same, as a copy of $\Gamma$ is attached to both vertices, and $\phi^{\prime \prime}$ is an end-isomorphism. If $w_s$ is instead not end-isomorphic to an element of $S$, then its second-level subgraphs have frontier points entirely in $\Gamma_{u_{n-1}}(v)$, and so $\phi^{\prime \prime}$ being an end-isomorphism guarantees that $w_s$ and $\phi^{\prime \prime}(w_s)$ have the same second-level subgraphs. Thus $\phi^{\prime \prime}$ extends to an end-isomorphism $\phi^\ast : \T(\Gamma, S)(v) \to \T(\Gamma, S)(u_n)$, and since $u_n \in \T_0(\Gamma, S)$, we have our claim.

To show the proposition it now suffices, as $\T_0(\Gamma, S)$ is context-free, to show that if $u, v \in V(\T_0(\Gamma, S))$, then $\T(\Gamma, S)(u) \sim \T(\Gamma, S)(v)$. But this follows from a near-identical argument to the final extension of the end-isomorphism above, as we only need to show that the end-isomorphism sends elements of $S$ to $S$, and elements not in $S$ to elements not in $S$. Thus every vertex in $\T(\Gamma, S)$ has an end-isomorphism representative in $\T_0(\Gamma, S)$, and there are only finitely many end-isomorphism classes of such vertices; hence $\T(\Gamma, S)$ is context-free.
\end{proof} 

\subsection{Determinisations} 

We will now show that under certain conditions, the process of \textit{determinising} a tree of copies of a context-free graph can preserve context-freeness. Let $\Gamma$ be a labelled graph, with label alphabet $A$. We say that $\Gamma$ is \textit{deterministic} if for every $a \in A$, the presence of a subgraph of $\Gamma$ of the form 
\[
\begin{tikzpicture}
\node (ro)[circle, draw, fill=black!50, inner sep=0pt, minimum width=8pt] at (0,0) {};
\node (v0)[label=right:$v_0$][circle, draw, fill=black!50, inner sep=0pt, minimum width=8pt] at (2, 0.5) {};
\node (v1)[label=right:$v_1$][circle, draw, fill=black!50, inner sep=0pt, minimum width=8pt] at (2, -0.5) {};

\path[->]
(ro)  edge node[above]{$a$}         (v0)
(ro)  edge node[below]{$a$}         (v1);
\end{tikzpicture}
\]
implies that $v_0 = v_1$, and hence that the edges are the same element of $E(\Gamma)$. Let a \textit{congruence} on a graph be an equivalence relation on $\Gamma$ identifying vertices with vertices and edges with edges, such that if two edges are identified, then the origins of the edges are identified, and the endpoints of the edges are identified. A congruence on a graph is \textit{determinising} if the quotient of the graph by this congruence is deterministic. The \textit{least determinising congruence} $\eta$ on $\Gamma$ is the least congruence on $\Gamma$ such that $\Gamma / \eta$ is deterministic. That a least determinising congruence always exists is clear, as the intersection of determinising congruences is also determinising.

We then say that $\Gamma / \eta$ is the \textit{determinised form} of $\Gamma$. It is not immediately clear that this congruence is well-defined. However, the following appears in the work of Stephen as \cite[Theorem~4.4]{Stephen1987}, but the terminology used therein, involving category theory, is omitted from here, as it is not necessary for the purposes of this article. 

\begin{lemma}
The determinised form of a labelled graph is unique.
\end{lemma}

Throughout, if not explicitly stated otherwise, let $\Gamma$ be a labelled graph with bounded degree, with label alphabet $A$, and $S \subseteq F(\Gamma)$ a set of attachment points. In general, it is clear that $\T(\Gamma, S)$ is not deterministic. We will henceforth let $\eta$ denote the least determinising congruence on $\T(\Gamma, S)$. If $u \in \T(\Gamma, S)$, we will simply write $[u]$ for $[u]_\eta$. Much as we were able to define the notion of depth in $\T(\Gamma, S)$, we can do the same in the determinised case.

\begin{definition}
Let $[u] \in V(\T(\Gamma, S) / \eta)$. Then we define $d([u]) \in \mathbb{N}$ as 
\[
d([u]) = \min_{w \in [u]} d(w).
\] 
\end{definition}
We note that if $u = (u_0, u_1, \dots, u_k)$, then $d([u]) \leq k$. In other words, determinisation does not increase the depth of any vertices. As the following example will show, determinisation can drastically decrease the depth of vertices, and will not preserve context-freeness in general. 

\begin{example}
Let $A$ be a finite alphabet, and $X \subseteq A^\ast$ a subset such that $\pres{Mon}{A}{x = 1 \: (x \in X)}$ is isomorphic to $\mathbb{Z} \times \mathbb{Z}$. Let $\Gamma$ be the \textit{flower graph} associated to $X$, i.e.\ the graph labelled by elements of $A$, consisting of a distinguished root $\ro$, and for each $w \in X$ a petal (loop) of length $|w|$ labelled by $w$. Let $S = V(\Gamma) \setminus \{ \ro \}$. Then $\Gamma$ is a finite graph and hence context-free, whence $\T(\Gamma, S)$ is also context-free by Proposition~\ref{Prop: Gamma CF => Tree(Gamma) CF}. However, we will see in Section~\ref{Sec: Special Monoids} that the monoid Cayley graph of a group whose word problem is not context-free is not context-free. It follows from \cite[Theorem~4.12]{Stephen1987} that $\T(\Gamma, S) / \eta$ is isomorphic to the right monoid Cayley graph of $\mathbb{Z} \times \mathbb{Z}$ with respect to the generating set $A$, and hence is not context-free. Furthermore, since the r\^ole played by $\mathbb{Z} \times \mathbb{Z}$ is simply as a group without context-free word problem, we may substitute e.g.\ a finitely presented group with undecidable word problem, or indeed any other horrible property, for $\mathbb{Z} \times \mathbb{Z}$ above, and conclude that determinisation in full generality yields very little in terms of preserving structural properties.
\end{example}

Thus the notion of depth is in general not very useful in the fully general case. The goal of the remaining parts of the section will be to introduce rather strong conditions on $\Gamma$ and $S$ such that the determinised form of $\T(\Gamma, S)$ can be guaranteed to remain context-free. Recall that for $S \subseteq F(\Gamma)$, the set $V_S$ was defined as the set of all vertices of $\Gamma$ to which copies of $\Gamma$ are attached when defining $\T_1(\Gamma, S)$.

\begin{definition}\label{Def: S-overlap-free}
We say that $\Gamma$ with $S \subseteq F(\Gamma)$ is $S$-\textit{overlap-free} if there do not exist any two (non-empty) paths $p_1$ and $p_2$ in $\Gamma$ such that:
\begin{enumerate}
\item $p_1$ begins in $\ro$;
\item $p_2$ begins in an element of $V_S$ and ends in an element not in $V_S$; and
\item $\ell(p_1) \equiv \ell(p_2)$.
\end{enumerate}
\end{definition}

As an important example to illustrate a pair satisfying this property, we have the following. A set of words is said to be \textit{cross-bifix-free} if no non-empty prefix of any element is a suffix of any other; see e.g.\ \cite{Bernini2017} for a more detailed study of such sets. We will not need any properties of such sets other than their definition.  

\begin{lemma}\label{Lem: Bifix-free set is S-overlap free}
Let $A$ be a finite alphabet, and let $X \subseteq A^\ast$ be a cross-bifix-free set and finite subset. Let $\Gamma$ be the flower graph associated to $X$. Let $S = V(\Gamma) \setminus \{ \ro\}$. Then $\Gamma$ is $S$-overlap-free. 
\end{lemma}
\begin{proof}
In this case, we evidently have that $V_S = S$, which remains finite. Assume for contradiction that $p_1$ and $p_2$ are two paths contradicting the $S$-overlap-free property. Then $p_2$ necessarily ends in $\ro$. Since $p_2$ does not begin in $\ro$, it must label a proper suffix of some $x_2 \in X$, as it is a path, and can therefore not pass through $\ro$. But $\ell(p_1) \equiv \ell(p_2)$ is also a non-empty prefix of some $x_1 \in X$, as $p_1$ begins in $\ro$, and hence $X$ is not cross-bifix-free.
\end{proof}

Hence one may reasonably claim that if $\Gamma$ is an $S$-overlap free graph, then this captures the idea of paths ``not overlapping with respect to $S$'', much the same as a subset $L$ of $A^\ast$ being a biprefix code captures the idea of words over $A^\ast$ ``not overlapping with respect to $L$''. 

One of the most important applications of being $S$-overlap-free is that it permits a great deal of control over the depth of vertices in the determinised form of a tree of copies of a graph. 

\begin{lemma}\label{Lem: Overlap-free means vertices are same, or one different.}
Assume that $\Gamma$ is a graph, with $S \subseteq F(\Gamma)$ a set of attachment points. Suppose $\Gamma$ is $S$-overlap-free. Let $(u_0, \dots, u_{k-1}, u_k)$ be an arbitrary vertex of $\T(\Gamma, S)$. Denote by $\Gamma'$ the subgraph of $\T(\Gamma, S) / \eta$ induced on the set of vertices
\[ 
\{ [(u_0, \dots, u_{k-1}, u')] \mid u' \in V(\Gamma) \}.
\]
If $m := \min_{v \in V(\Gamma')} d([v])$ denotes the minimal depth of a vertex of $\Gamma'$, then for every $w \in V(\Gamma')$, we have that 
\[
d([w]) \in \{ m, m+1 \}.
\]
\end{lemma}
\begin{proof}
Recall that if $u \in \T(\Gamma, S)$, then $d([u]) := \min_{w \in [u]} d(w)$. It follows that $d(u) \geq d([u])$, and so determinising $\T(\Gamma, S)$ cannot increase the depth of a vertex. Consequently the claim directly follows by induction on the depth of vertices if we can show that if $(\Gamma, S)$ is $S$-overlap-free, then a vertex can only be identified with a vertex of depth at most one less than itself as a result of determinisation. First, note that $(u_0, \dots, u_{k-1})$ is an branch point. Thus $u_{k-1} \neq \ro$, as $\ro \not\in V_S$. If $w = (u_0, \dots, u_{k-1}, w')$, then let $p_1 \colon \ro \xrightarrow{\alpha} w'$ be a path, where $\alpha \in A^\ast$. Assume for contradiction that there is some $\hat{w} \in [w]$ with $d(\hat{w}) = d(w) - 2$. Then necessarily there is some path $p_2 \colon u_{k-1} \xrightarrow{\alpha_1} \ro$, where $\alpha_1 \in A^+$ and $\alpha \equiv \alpha_1 \alpha_2$ for some $\alpha_2 \in A^\ast$; this follows from the fact that all paths  from $(u_0, \dots, u_{k-1})$ to a lower depth vertex must pass through $(u_0, \dots, u_{k-2})$. As we are determinising, there must hence also be a path $p_1' \colon \ro \xrightarrow{\alpha_1} w''$ where $w'' \in V(\Gamma)$ is the last entry of $\hat{w}$. Thus $p_1'$ and $p_2$ contradict the $S$-overlap-free condition. 
\end{proof}

\subsection{Bounded folding}

Let $\Gamma$ be a labelled graph. Let Aut$(\Gamma)$ denote the group of all label-preserving graph automorphisms of $\Gamma$. We say that an action $\Phi_{H, \Gamma} \colon H \to \textnormal{Aut}(\Gamma)$ of a group $H$ on $\Gamma$ is \textit{almost-transitive} if the action is faithful and induces an almost-transitive action on the vertices of $\Gamma$, i.e. if $| V(\Gamma) / H| < \infty$. As $H$ acts on $\Gamma$, this means $\Gamma / H$ is a finite graph. We note that if $H \leq \textnormal{Aut}(\Gamma)$, then assuming the action is faithful is no restriction.

Let $H$ be a group acting almost-transitively on a graph $\Gamma$. Then we can express the vertices of $\Gamma$ as a disjoint union $V(\Gamma) = \bigsqcup_{h \in H} V_h(\Gamma / H)$, where $V_h(\Gamma / H) \subseteq V(\Gamma)$ is a copy of $V(\Gamma / H)$. Indeed, as the action is faithful, for any two $h_1, h_2 \in H$ with $h_1 \neq h_2$ we have $V_{h_1}(\Gamma / H) \cap V_{h_2}(\Gamma / H) = \varnothing$. 

Let $u, v \in V(\Gamma)$ be any two vertices, and let $h_1, h_2 \in H$ be the unique elements such that $u \in V_{h_1}(\Gamma / H)$ and $v \in V_{h_2}(\Gamma / H)$. If $h_1 = h_2$, then we will write $u \sim_{L, H} v$. Here $L$ is used to abbreviate the idea of a \textit{local neighbourhood} with respect to the action of $H$. Clearly $\sim_{L, H}$ is an equivalence relation. 

Assuming $S \neq \varnothing$ is a set of attachment points of $\Gamma$, there is a natural way to extend $\sim_{L, H}$ to $\T(\Gamma, S)$, by having $(u_0, u_1, \dots, u_k) \sim_{L, H} (u'_0, u'_1, \dots, u'_n)$ if and only if $u_k \sim_{L,H} u'_n$. We note that in general there are two distinct types of equivalence classes under this relation; one of which is isomorphic to the graph $\Gamma / H$, and the other of which is $(\Gamma / H) \setminus \{ \ro\}$. The first occurs as the equivalence class of any vertex of depth $0$, and the latter as the equivalence class of any vertex of depth greater than $0$. 

\begin{definition}
Let $\Gamma$ be a labelled graph with bounded degree, $H \leq \textnormal{Aut}(\Gamma)$ act almost-transitively on $\Gamma$, and $S \subseteq F(\Gamma)$ a set of attachment points. Then we say that the triple $(\Gamma, H, S)$  is $S$-\textit{full} if: for all edges $(u, a, v) \in E(\Gamma)$ such that $u$ and $v$ are in distinct $\sim_{L,H}$-classes, we have that $v \not\in V_S$.
\end{definition}

\begin{example}\label{Example: S-fullness ex}
Let $\Gamma$ be the finite graph with root $\ro$ and a single loop involving a total of four vertices. See Figure~\ref{Fig: S-fullness loop}, in which the root vertex is all black. Take $H \leq \textnormal{Aut}(\Gamma) = 1$ as the trivial action, and $S = \{ v_1, v_2, v_3 \}$, which gives $V_S = S$. For this graph, the relation $\sim_{L, H}$ is the identity relation on $\Gamma$, i.e. $u \sim_{L,H} v$ if and only if $u = v$. Then this triple $(\Gamma, H, S)$ is not $S$-full, as witnessed by e.g. the edge $v_1 \xrightarrow{b} v_2$, for $v_1$ and $v_2$ are in distinct $\sim_{L,H}$-classes, but $v \in V_S$. 
\end{example}

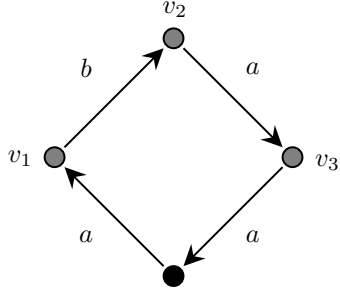
\begin{figure}
\tikzset{every picture/.style={line width=0.75pt}} %set default line width to 0.75pt        

\begin{tikzpicture}[x=0.75pt,y=0.75pt,yscale=-1,xscale=1]
%uncomment if require: \path (0,149); %set diagram left start at 0, and has height of 149

%Shape: Circle [id:dp8250065430702584] 
\draw  [fill={rgb, 255:red, 0; green, 0; blue, 0 }  ,fill opacity=1 ] (70,135) .. controls (70,132.24) and (72.24,130) .. (75,130) .. controls (77.76,130) and (80,132.24) .. (80,135) .. controls (80,137.76) and (77.76,140) .. (75,140) .. controls (72.24,140) and (70,137.76) .. (70,135) -- cycle ;
%Shape: Circle [id:dp4858786434259328] 
\draw  [fill={rgb, 255:red, 128; green, 128; blue, 128 }  ,fill opacity=1 ] (10,75) .. controls (10,72.24) and (12.24,70) .. (15,70) .. controls (17.76,70) and (20,72.24) .. (20,75) .. controls (20,77.76) and (17.76,80) .. (15,80) .. controls (12.24,80) and (10,77.76) .. (10,75) -- cycle ;
%Shape: Circle [id:dp2791746494603886] 
\draw  [fill={rgb, 255:red, 128; green, 128; blue, 128 }  ,fill opacity=1 ] (70,15) .. controls (70,12.24) and (72.24,10) .. (75,10) .. controls (77.76,10) and (80,12.24) .. (80,15) .. controls (80,17.76) and (77.76,20) .. (75,20) .. controls (72.24,20) and (70,17.76) .. (70,15) -- cycle ;
%Shape: Circle [id:dp9686076437843945] 
\draw  [fill={rgb, 255:red, 128; green, 128; blue, 128 }  ,fill opacity=1 ] (130,75) .. controls (130,72.24) and (132.24,70) .. (135,70) .. controls (137.76,70) and (140,72.24) .. (140,75) .. controls (140,77.76) and (137.76,80) .. (135,80) .. controls (132.24,80) and (130,77.76) .. (130,75) -- cycle ;
%Straight Lines [id:da05228453261885835] 
\draw [color={rgb, 255:red, 0; green, 0; blue, 0 }  ,draw opacity=1 ]   (70,130) -- (22.12,82.12) ;
\draw [shift={(20,80)}, rotate = 405] [fill={rgb, 255:red, 0; green, 0; blue, 0 }  ,fill opacity=1 ][line width=0.08]  [draw opacity=0] (10.72,-5.15) -- (0,0) -- (10.72,5.15) -- (7.12,0) -- cycle    ;
%Straight Lines [id:da23795514157966102] 
\draw [color={rgb, 255:red, 0; green, 0; blue, 0 }  ,draw opacity=1 ]   (20,70) -- (67.88,22.12) ;
\draw [shift={(70,20)}, rotate = 495] [fill={rgb, 255:red, 0; green, 0; blue, 0 }  ,fill opacity=1 ][line width=0.08]  [draw opacity=0] (10.72,-5.15) -- (0,0) -- (10.72,5.15) -- (7.12,0) -- cycle    ;
%Straight Lines [id:da1625305641456749] 
\draw [color={rgb, 255:red, 0; green, 0; blue, 0 }  ,draw opacity=1 ]   (80,20) -- (127.88,67.88) ;
\draw [shift={(130,70)}, rotate = 225] [fill={rgb, 255:red, 0; green, 0; blue, 0 }  ,fill opacity=1 ][line width=0.08]  [draw opacity=0] (10.72,-5.15) -- (0,0) -- (10.72,5.15) -- (7.12,0) -- cycle    ;
%Straight Lines [id:da36088415604157653] 
\draw [color={rgb, 255:red, 0; green, 0; blue, 0 }  ,draw opacity=1 ]   (130,80) -- (82.12,127.88) ;
\draw [shift={(80,130)}, rotate = 315] [fill={rgb, 255:red, 0; green, 0; blue, 0 }  ,fill opacity=1 ][line width=0.08]  [draw opacity=0] (10.72,-5.15) -- (0,0) -- (10.72,5.15) -- (7.12,0) -- cycle    ;

% Text Node
\draw (31,115) node [inner sep=0.75pt]    {$\displaystyle a$};
% Text Node
\draw (31,30) node [inner sep=0.75pt]    {$\displaystyle b$};
% Text Node
\draw (115,30) node [inner sep=0.75pt]  {$\displaystyle a$};
% Text Node
\draw (115,115) node [inner sep=0.75pt]    {$\displaystyle a$};
\draw (-10,70) node [anchor=north west][inner sep=0.75pt]   [align=left] {$\displaystyle v_{1}$};
\draw (145,72) node [anchor=north west][inner sep=0.75pt]   [align=left] {$\displaystyle v_{3}$};
\draw (68, -5) node [anchor=north west][inner sep=0.75pt]   [align=left] {$\displaystyle v_{2}$};

\end{tikzpicture}

\caption{An example of a graph $\Gamma$, with $H$ the action of the trivial group and $S$ all three non-root vertices, which is not $S$-full See Example~\ref{Example: S-fullness ex}.}
\label{Fig: S-fullness loop}
\end{figure}

Hence $S$-fullness, with $\Gamma, H$, and $S$ as in the definition, ensures that if one has a path $p \colon u \to v$ in which $u$ and $v$ are in distinct $\sim_{L,H}$-classes, then one can factor the path as $p \colon u \to u' \to v$, where $u' \not\in V_S$. Thus this factoring gives a path $p_2 \colon u \to u'$, which can be used as the $p_2$ in Definition~\ref{Def: S-overlap-free} of $S$-overlap-freeness. In particular, if $u \in V_S$, then in $\T(\Gamma, S)$ one attaches a copy of $\Gamma$ to $u$. This gives a number of paths originating in $u$ and otherwise only involving vertices of greater depth; these paths are the only candidates for determinisation if one assumes $\Gamma$ is deterministic. On the other hand, these paths are also candidates for $p_1$ in Definition~\ref{Def: S-overlap-free}. This means all candidates for determinisation are controlled, if one assumes $S$-overlap-freeness and $S$-fullness.

This links together these two important concepts, and together we shall see that they indeed offer a great deal of control over the determinising of $\T(\Gamma, S)$. We first combine the two into a single definition.

\begin{definition}
We say that a triple $(\Gamma, H, S)$ with $H \leq \textnormal{Aut}(\Gamma)$ and $S \subseteq F(\Gamma)$ satisfies the \textit{bounded folding condition} if all of the following hold:
\begin{enumerate}
\item The action of $H$ on $\Gamma$ is almost-transitive;
\item $(\Gamma, H, S)$ is $S$-full;
\item $\Gamma$ is $S$-overlap-free.
\end{enumerate}
Furthermore, if $(\Gamma, H, S)$ satisfies the bounded folding condition, then we say that $\Omega := |V(\Gamma/H)|$ is the \textit{folding constant} of the triple.
\end{definition}

We will introduce a couple of pieces of notation, which will come in handy in a few circumstances below. Fix, for the remainder of the section, an arbitrary triple $(\Gamma, H, S)$ satisfying the bounded folding condition. We will denote by $\TG := \T(\Gamma, S)$, and by $\TGE := \TG / \eta$ the determinised form of $\TG$. 

Recall that as $(\Gamma, H, S)$ satisfies the bounded folding condition, then there are two distinct types of equivalence classes of graphs under the relation $\sim_{L, H}$ on $\TG$. The following technical lemma will demonstrate that in this setting, almost the same is true for $\TGE$, where we replace ``two'' by ``finitely many''. 

\begin{lemma}\label{Lem: Decompose TGE into finitely many types}
Suppose that $\Gamma$ is deterministic and context-free. Suppose further that if Aut$_S(\Gamma)$ denotes the subgroup of Aut$(\Gamma)$ consisting of automorphisms $\phi$ such that $(\phi(s) \in V_S \iff s \in V_S)$, then we have that $(\Gamma, \textnormal{Aut}_S(\Gamma), S)$ satisfies the bounded folding condition. 

Then there exists a finite set $U := \{ \Gamma_{V_0}, \dots, \Gamma_{V_m} \}$ of context-free, deterministic, pairwise non-isomorphic graphs with $V_i \subseteq V(\Gamma)$ such that:
\begin{enumerate}
\item For every $\Gamma_{V_i} \in U$, there exist sets $G_{V,j} \subseteq V(\TGE)$ such that for each $j$, the graph $\Gamma_{V_i}$ is isomorphic to the subgraph of $\TGE$ induced on $G_{V,j}$, and $d(g)$ is constant for all $g \in G_{V,j}$.
\item There exists a surjective function $f \colon V(\TGE) \to U$ such that $u$ appears as a vertex of exactly one of the above copies of $f(u)$ inside $\TGE$. This unique copy will be denoted $f^\ell(u)$.
\end{enumerate}
\end{lemma}
\begin{proof}
Before stating the claim directly, we summarise its philosophy as follows: just as it is possible to decompose $\TG$ into a disjoint collection of copies of $\Gamma$ and $\Gamma \setminus \{ \ro\}$, the bounded folding condition makes it possible to decompose $\TGE$ into a disjoint collection of copies of $\Gamma$ and boundedly modified copies of $\Gamma$, all of which are connected components of end-spaces of context-free graphs. We now make this formal.

Let $\Omega \in \mathbb{N}$ denote the folding constant of $(\Gamma, \textnormal{Aut}_S(\Gamma), S)$. If $S = \varnothing$, then $\TGE = \TG = \Gamma$, in which case there is nothing to show; we may take $U = \{ \Gamma \}$. Hence, assume $S \neq \varnothing$.

We will first construct a set $U'$ of graphs as the set of all graphs $\Gamma_V$ which can be obtained from $\Gamma$ by the following ``re-rooting'' procedure: 
\begin{enumerate}[(i)]
\item Select a set of vertices $V \subseteq V(\Gamma)$ with $1 \in V$ and such that all vertices $v \in V$ are at distance at most $\Omega$ from $1$. Such a set is clearly always finite. 
\item Add a new vertex $\hat{1}$ to the graph $\Gamma$ and an edge (arbitrarily labelled) from $\hat{1}$ to every vertex $v \in V$. Call this graph $\widehat{\Gamma}$.
\item Take as $\Gamma_V$ the graph $\widehat{\Gamma} \setminus \widehat{\Gamma}^{(2)}$. 
\end{enumerate}

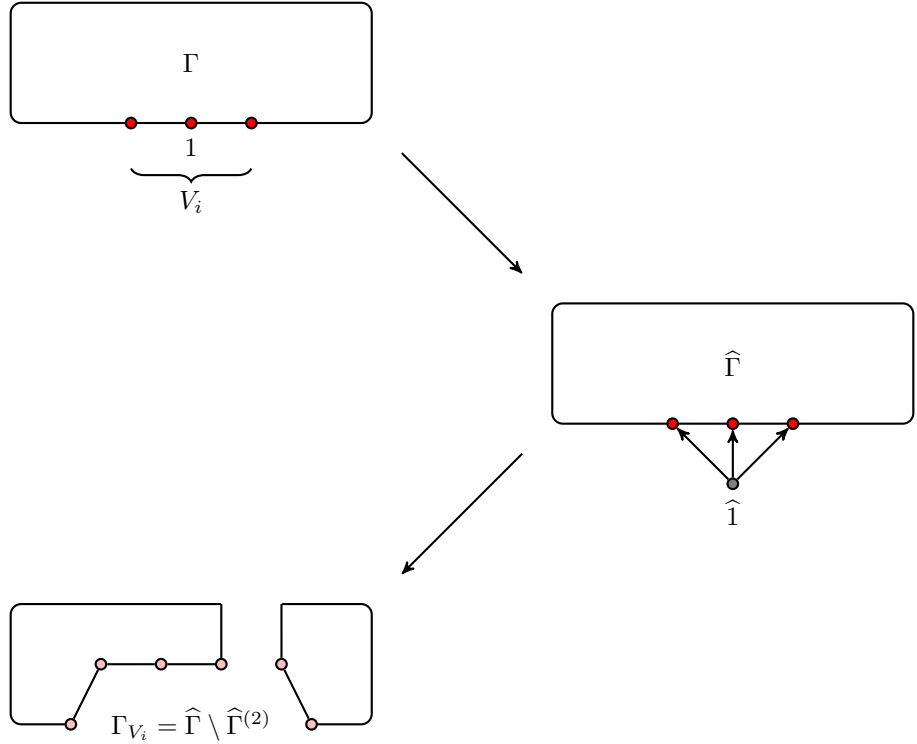
\begin{figure}
\begin{tikzpicture}[>=stealth',thick,scale=0.8,el/.style = {inner sep=2pt, align=left, sloped}]

%%%%%%%%%%%%%%%%%%%%%% First rectangle %%%%%%%%%%%%%%%%%%%%%%%
\draw[rounded corners] (0, 0) rectangle (6, 2) {};
\draw (3, 1) node {$\Gamma$};

\node (v0)[][circle, draw, fill=red!100, inner sep=0pt, minimum width=4pt] at (2,0) {};
\node (v1)[label = below:$1$][circle, draw, fill=red!100, inner sep=0pt, minimum width=4pt] at (3,0) {};
\node (v2)[][circle, draw, fill=red!100, inner sep=0pt, minimum width=4pt] at (4,0) {};

\draw [decorate,decoration={brace,amplitude=5pt,mirror,raise=4ex}]
  (2,0) -- (4,0) node[midway,yshift=-3em]{$V_i$};
                        
\path[->]
(6.5,-0.5) edge (8.5,-2.5);

%%%%%%%%%%%%%%%% Second rectangle below %%%%%%%%%%%%%%%%%%%%%
\draw[rounded corners] (9, -3) rectangle (15, -5) {};

\node (vv0)[][circle, draw, fill=red!100, inner sep=0pt, minimum width=4pt] at (11,-5) {};
\node (vv1)[circle, draw, fill=red!100, inner sep=0pt, minimum width=4pt] at (12,-5) {};
\node (vv2)[][circle, draw, fill=red!100, inner sep=0pt, minimum width=4pt] at (13,-5) {};
\node (1hat)[label = below:$\widehat{1}$][circle, draw, fill=black!50, inner sep=0pt, minimum width=4pt] at (12,-6) {};

\draw (12,-4) node {$\widehat{\Gamma}$};
                                             
\path[->] 
 (1hat)  edge (vv0)
 (1hat)  edge (vv1)
 (1hat)  edge (vv2);

%%%%%%%%%%%%%%%% Third rectangle below %%%%%%%%%%%%%%%%%%%%%
\draw[rounded corners] (0, -8) rectangle (6, -10) {};
\fill [white] (1,-10.1) rectangle (5,-9.9);
\fill [white] (3.5,-8.1) rectangle (4.5,-7.9);

\node (vvv0)[][circle, draw, fill=pink!100, inner sep=0pt, minimum width=4pt] at (1,-10) {};
\node (vvv1)[][circle, draw, fill=pink!100, inner sep=0pt, minimum width=4pt] at (1.5,-9) {};
\node (vvv2)[][circle, draw, fill=pink!100, inner sep=0pt, minimum width=4pt] at (2.5,-9) {};
\node (vvv3)[][circle, draw, fill=pink!100, inner sep=0pt, minimum width=4pt] at (3.5,-9) {};
\node (vvv4)[][circle, draw, fill=pink!100, inner sep=0pt, minimum width=4pt] at (4.5,-9) {};
\node (vvv5)[][circle, draw, fill=pink!100, inner sep=0pt, minimum width=4pt] at (5,-10) {};

\draw (3,-10) node {$\Gamma_{V_i} = \widehat{\Gamma} \setminus \widehat{\Gamma}^{(2)}$};
                                              
\path[->]
(8.5,-5.5) edge (6.5,-7.5);

\draw
 (vvv0)  edge (vvv1)
 (vvv1)  edge (vvv2)
 (vvv2)  edge (vvv3)
 (vvv3)  edge (3.5,-8)
 (vvv4)  edge (4.5,-8)
 (vvv4)  edge (vvv5);    
    
\end{tikzpicture}
\caption{An example of the construction of a graph in $U'$ for a given set $V_i$. The light red vertices in the final step are the neighbours of the set $V_i$ inside $\Gamma$.}
\label{Fig: Constructing fUVi}
\end{figure}
An example of this procedure is given in Figure~\ref{Fig: Constructing fUVi}. As $\Gamma$ is context-free, it follows either by direct geometric considerations, or by considering $\Gamma$ as the \textit{transition graph of a pushdown automata} (see \cite{Muller1985} for precise definitions), that $\widehat{\Gamma}$ is also a context-free graph; for the latter transition graphs, we can add some initial transitions to a pushdown automata which $\Gamma$ is the transition graph of. Furthermore, by a similar argument, any end-space of a context-free graph is a context-free graph. Hence any end-space of $\widehat{\Gamma}$ is context-free, and hence in particular $\Gamma_{V_i}$ is context-free for all $i$. 

We will now prove by induction on the depth of vertices that there is a $U \subseteq U'$ of graphs with the claimed properties. First, note that $\Gamma_{\varnothing} \cong \Gamma$. The subgraph $\TGE(0)$ of vertices with depth $0$ is by assumption isomorphic to $\Gamma$. Hence $\Gamma_{\varnothing}$ appears as a subgraph of $\TGE$ induced on a set of vertices of fixed depth; namely $0$. As a consequence, if $u \in V(\TGE)$ with $d(u) = 0$, then we set $f(u) = \Gamma_\varnothing$.

Now, assume for induction that for some $n \in \mathbb{N}$ we have that there exists a set $U_n$ such that (1) is satisfied, and such that (2) is satisfied for all vertices $u \in \TGE$ with depth $\leq n$. That is, such that there exists a function $f_n$ mapping any vertex $u$ with $d(u) \leq n$ to some element of $U_n$, and any such $u$ appears as a vertex of exactly one of the copies of $f(u)$ in $\TGE$, which we denote by $f^\ell(u)$. Let $u$ be a vertex of $\TGE$ with $d(u) = n+1$. Let $u = [\big( u_0, \dots, u_{n+1} \big)]$. Now $( u_0, \dots, u_n ) \in V_n$, so there exists a copy of $\Gamma$ attached to $( u_0, \dots, u_n)$ in $\TG$. Hence, since $\TGE$ is obtained from $\TG$ by determinising, from $[(u_0, \dots, u_n)]$ in $\TGE$ there still, just as in $\TG$, exists a unique copy of $\Gamma$ attached, which we denote by $\Gamma'$. 

Now by Lemma~\ref{Lem: Overlap-free means vertices are same, or one different.}, as $\Gamma$ is $S$-overlap-free, all vertices of $\Gamma'$ have depth $n$ or depth $n+1$. The root $\ro'$ of $\Gamma'$ will certainly have depth $n$, but unlike the case in $\TG$, certain non-root vertices of $\Gamma'$ will also have depth $n$ as a result of determinisation. 

We will show that all vertices of $\Gamma'$ with depth $n$ are at $\Gamma'$-distance $\leq \Omega$ from $\ro'$. In particular, there are only boundedly many such vertices for any choice of $u$. 

Assume for contradiction that $v$ is a vertex of $\Gamma'$ at $\Gamma'$-distance $> \Omega$ from $\ro'$, such that $d(v) = n$. Then, there is a path $p \colon \ro' \to v_1' \to \cdots \to v_k' = v$ in $\Gamma'$, where $k > \Omega$, and all $\to$ correspond to passing over single edges. As $\Gamma$ is deterministic, we must also find a corresponding path $p \colon \ro \to v_1 \to \cdots \to v_k = v$ in $\Gamma$, for determinising $\TG$ does not affect the individual copies of $\Gamma$ within. But by the definition of $\Omega$, such a path must necessarily pass through some element of $V_S$, and we hence have a contradiction to $S$-overlap-freeness. Hence any vertex $v$ in $\Gamma'$ with $d(v) \leq n$ will be at $\Gamma'$-distance $\leq \Omega$ from $\ro'$.

Let $V' \subseteq V(\Gamma')$ be the set of vertices $v'$ of $\Gamma'$ with $d(v') = n$. Then (independently of $u$) the above implies that $|V'| \leq |B(\ro, \Omega)|$, where $B(\ro, \Omega)$ denotes the ball of radius $\Omega$ around $\ro$ inside $\Gamma$. Hence $V'$ is one of the sets $V$ chosen above, and so it follows that $\Gamma_{V'} \in U'$. We set $f^\ell(u)$ to be the the subgraph of $\TGE$ induced on $V(\Gamma') \setminus V'$, and $f(u) = \Gamma_{V'}$ as its isomorphism class. Thus by induction we may take $U \subseteq U'$ as the image of this function $f$. 
\end{proof}
\begin{remark}
We note that the graphs $\Gamma_{V_i}$ constructed above not be connected, and no roots have explicitly been specified for them yet. Furthermore, we recall that a graph is context-free if and only if all of its connected components are. 
\end{remark}

The set of vertices of depth $0$ in $\TG$ clearly induces a subgraph of $\TG$ isomorphic to $\Gamma$. However, the same statement need not be true when substituting $\TGE$ for $\TG$. For although the map sending $\Gamma$ to this subgraph of $\TGE$ is always a bijective homomorphism, it need not be faithful, as some edges may be missing from $\Gamma$. To enable adding conditions to avoid cases of this form, we let $\TGE(0)$ denote the subgraph of $\TGE$ induced on the set of vertices of depth $0$. We remark that the condition below that $\TGE(0) \cong \Gamma$ is by no means critical, and only helps simplify parts of the proof; one could without too much additional effort prove the below theorem without this condition, but this is not done in interest of some brevity.

\begin{theorem}\label{Thm: Bounded folding gives CF from T}
Let $\Gamma$ be a graph, and $S \subseteq F(\Gamma)$ a set of frontier points with $\ro \not\in S$. Let $\TG := \T(\Gamma, S)$, and let $\TGE$ denote the determinised form of $\TG$. Assume that the following hold: 
\begin{enumerate}
\item $\Gamma$ is deterministic and context-free;
\item $\TGE(0) \cong \Gamma$; 
\item If Aut$_S(\Gamma)$ denotes the subgroup of Aut$(\Gamma)$ consisting of automorphisms $\phi$ such that $(\phi(s) \in V_S \iff s \in V_S)$, then $(\Gamma, \textnormal{Aut}_S(\Gamma), S)$ satisfies the bounded folding condition.
\end{enumerate}
Then the determinised form $\TGE$ of $\TG$ is context-free.
\end{theorem}
\begin{proof}
Let $\Omega \in \mathbb{N}$ denote the folding constant of $(\Gamma, \textnormal{Aut}_S(\Gamma), S)$. If $S = \varnothing$, then $\TGE = \TG = \Gamma$, in which case there is nothing to show. Hence, assume $S \neq \varnothing$. We note that the conditions of Lemma~\ref{Lem: Decompose TGE into finitely many types} are satisfied, and hence we will speak of the set $U = \{ \Gamma_{V_0}, \dots, \Gamma_{V_m} \}$ constructed in the conclusion of the same, as well as the function $f$ and the associated $f^\ell$. 

Concretely, for $u \in V(\TGE)$, we let $f^\ell(u)$ denote the unique copy of $f(u)$ in which $u$ appears. As it will be useful to consider $f^\ell(u)$ as a graph with its own ``intrinsic geometry'', as opposed from the ``extrinsic'' distances resulting from its embedding in $\TGE$, we will introduce the notation $u^\ell \in V(f^\ell(u))$ for the image of $u$ in $f^\ell(u)$ under the bijection between $f^\ell(u)$ and the induced subgraph of $\TGE$. If $f(u) = V_i$, then we will denote by $V_i^\ell$ the copy of the vertices $V_i$ appearing in $f^\ell(u)$. The main idea of the proof is now the following. We will introduce on the vertices of $\TGE$ a series of equivalence relations, which together induce an equivalence relation $\approx$ on $V(\TGE)$ with finitely many classes. This equivalence relation will have the powerful property that for any $u, v \in V(\TGE)$, we have 
\[
u \approx v \implies \TGE(u) \sim \TGE(v).
\] Since $\approx$ has only finitely many classes, proving the above will demonstrate that $\TGE$ only has finitely many end-isomorphism classes, and hence is context-free, completing the proof. We make the relation concrete as follows, with letters next to each new relation introduced for ease of reading. ´

\vspace{0.5cm}

($\mathcal{A}$) First, since the range of $f$ is finite, we will introduce an equivalence relation $\approx_A$ with classes  
\[
\mathcal{A}_i = \{ u \in V(\TGE) \mid f(u) = \Gamma_{V_i} \}.
\] for $1 \leq i \leq |U|$. We will fix a class $\mathcal{A} := \mathcal{A}_\alpha$. For ease of notation, we will denote $\Gamma_{V_\alpha}$ by $\Gamma_\alpha$. 

\vspace{0.5cm}

($\mathcal{B}$) The next equivalence relation is based on the ``extrinsic'' geometry of $f^\ell(u)$ with respect to its embedding in $\TGE$, where $u$ is some vertex. We have not yet explicitly rooted $f^\ell(u)$, so the notation $f^\ell(u)(u^\ell)$ is not well defined; furthermore, intuitively, this ``intrinsic'' end of $f^\ell(u)$ may not necessarily be compatible to the ``extrinsic'' end of $u$ inside $\TGE$ arising from the embedding of $f^\ell(u)$ inside $\TGE$.

We will now remedy this formally, by introducing \textit{labellings} on some of the vertices of each $f^\ell(u)$, where $u$ ranges over the vertices of $\TGE$, which will capture the behaviour of how the intrinsic geometry of $f^\ell(u)$ may differ from the extrinsic geometry resulting from its embedding in $\TG$. The labelling map, which we will denote as $g \colon \{ f^\ell(u) \mid u \in V(\TGE) \} \to \{ 1, \dots, K\}$ is defined as follows, where $K \in \mathbb{N}$ is a constant determined by the below definition. First, let $L$ be the set containing the set of all partiaully vertex-labelled graphs $G$ satisfying the following conditions: $G$ is isomorphic to some $f^\ell(u)$, and some subset of the vertices of $G$ within a $2 \Omega$-radius from the root of $G$ has been labelled by a natural number in $\{ 0, \dots, 2 \Omega \}$. Clearly, $|L| < \infty$, as $\{ f^\ell(u) \mid u \in V(\TGE) \}$ is a finite set, and all graphs considered have bounded degree.

Let $u \in \mathcal{A}$. We will show that the ``extrinsic geometry'' of the embedding of $f^\ell(u)$ in $\TGE$ is completely captured by some element of $L$. Consider $f^\ell(u)$. As in the proof of Lemma~\ref{Lem: Decompose TGE into finitely many types}, since $f(u) = \Gamma_{\alpha}$, the vertices $V_A^\ell \subseteq V(f^\ell(u))$ all have depth $d(u) -1$ in $\TGE$. Furthermore, every vertex in $V_\alpha^\ell$ is within a $2 \Omega$ radius, in the intrinsic metric on $f^\ell(u)$, from the root $\ro^\ell$ of $f^\ell(u)$, and includes $\ro^\ell$. Enumerate the finite set $V_\alpha^\ell = \{ \ro^\ell, v_1^\ell, \dots, v_\nu^\ell\}$, and for all $0 \leq j \leq \nu$ we set $v_j$ to be the (unique) corresponding vertex of $\TGE$. We then label $v_j^\ell$ by the natural number $|v_j|_{\TGE} - \min_i |v_i|_{\TGE}$. As $\TGE$ is obtained from $\TG$ by determinising, it is clear that none of these natural numbers exceeds $2 \Omega$. Thus this labelling of $f^\ell(u)$ is contained in $L$. Fix a bijection between $L$ and $\{1, \dots, |L|\}$. We define $g(u)$ as the natural number corresponding to the labelling of $f^\ell(u)$ under this bijection. We thus introduce an equivalence relation $\approx_B$ on $\mathcal{A}$ with classes
\[
\mathcal{B}_i := \{ u \in V(\TGE) \mid g(u) = i \}.
\]
As in part ($\mathcal{A}$), we will fix an arbitrary $\beta \in \{ 1, \dots, |L| \}$, and denote $\mathcal{B} := \mathcal{B}_\beta$. Given this fixed labelling $g(u) = \beta$ it is straightforward, by adding a new root and some (finitely many) paths of finite length from this root to each of the vertices in $V^\ell_\alpha$, to construct a context-free graph $f_+^\ell(u)$, rooted at a new root $1_+$, such that the intrinsic distance from $1_+$ to any vertex $v^\ell_i \in V^\ell_\alpha$ is precisely that given by the label of $v^\ell_i$. An example of this construction is shown in Figure~\ref{Fig: f(ell)+ giving f(ell) as end}. 

Thus we fix this graph $f_+^\ell(u)$. In particular, we may now without ambiguity speak of the end-space $f_+^\ell(u)(u^\ell)$. Furthermore, since this labelling is constructed precisely based on the difference in distances caused by embedding $f^\ell(u)$ in $\TGE$, it is easy to see that we have $|w|_{\TGE} = c(u) + |w^\ell|_{f_+^\ell(u)}$ for all vertices $w^\ell \in V(f^\ell(u))$, where $c(u)$ is some constant depending only on $u$. See Figure~\ref{Fig: Embedding f(ell)(endspace) into TGE.}.

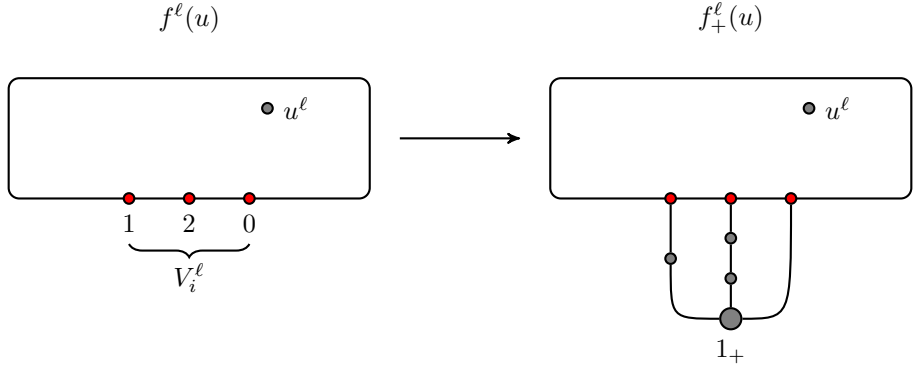
\begin{figure}
\begin{tikzpicture}[>=stealth',thick,scale=0.8,el/.style = {inner sep=2pt, align=left, sloped}]

%%%%%%%%%%%%%%%%%%%%%% First rectangle %%%%%%%%%%%%%%%%%%%%%%%
\draw[rounded corners] (0, 0) rectangle (6, 2) {};
\draw (3, 3) node {$f^\ell(u)$};

\node (v0)[label = below:$1$][circle, draw, fill=red!100, inner sep=0pt, minimum width=4pt] at (2,0) {};
\node (v1)[label = below:$2$][circle, draw, fill=red!100, inner sep=0pt, minimum width=4pt] at (3,0) {};
\node (v2)[label = below:$0$][circle, draw, fill=red!100, inner sep=0pt, minimum width=4pt] at (4,0) {};
\node (u)[label = right:$u^\ell$][circle, draw, fill=black!50, inner sep=0pt, minimum width=4pt] at (4.3,1.5) {};

\draw [decorate,decoration={brace,amplitude=5pt,mirror,raise=4ex}]
  (2,0) -- (4,0) node[midway,yshift=-3em]{$V^\ell_i$};
                        
\path[->]
(6.5,1) edge (8.5,1);

%%%%%%%%%%%%%%%% Second rectangle below %%%%%%%%%%%%%%%%%%%%%
\draw[rounded corners] (9, 0) rectangle (15, 2) {};
\draw (12, 3) node {$f_+^\ell(u)$};

\node (w0)[circle, draw, fill=red!100, inner sep=0pt, minimum width=4pt] at (11,0) {};
\node (w1)[circle, draw, fill=red!100, inner sep=0pt, minimum width=4pt] at (12,0) {};
\node (w2)[circle, draw, fill=red!100, inner sep=0pt, minimum width=4pt] at (13,0) {};
\node (uu)[label = right:$u^\ell$][circle, draw, fill=black!50, inner sep=0pt, minimum width=4pt] at (13.3,1.5) {};

\node (vv0)[][circle, draw, fill=black!50, inner sep=0pt, minimum width=4pt] at (11,-1) {};
\node (vv1)[][circle, draw, fill=black!50, inner sep=0pt, minimum width=4pt] at (12,-1.33) {};
\node (vv2)[][circle, draw, fill=black!50, inner sep=0pt, minimum width=4pt] at (12,-0.66) {};
\node (1+)[label = below:$1_+$][circle, draw, fill=black!50, inner sep=0pt, minimum width=8pt] at (12,-2) {};

\draw (1+) .. controls (11,-2) .. (vv0); 
\draw (1+) edge (vv1)  
 	  (vv1) edge (vv2)
 	  (vv2) edge (w1)
 	  (vv0) edge (w0); 
\draw (1+) .. controls (13,-2) .. (w2);

\end{tikzpicture}
\caption{The construction of the end-space $f^\ell_+(u)(u^\ell)$ from a given labelling of the vertices of $f^\ell(u)$.}
\label{Fig: f(ell)+ giving f(ell) as end}
\end{figure}

\begin{figure}
\begin{tikzpicture}[scale = 0.75]

\fill[rounded corners, green!50] (3.1,2) -- (6,0.8) -- (6,2)-- cycle;
\draw[rounded corners] (0, 0) rectangle (6, 2) {};
\draw (3, 1) node {$f_+^\ell(u)$};

\node (w0)[circle, draw, fill=red!100, inner sep=0pt, minimum width=4pt] at (2,0) {};
\node (w1)[circle, draw, fill=red!100, inner sep=0pt, minimum width=4pt] at (3,0) {};
\node (w2)[circle, draw, fill=red!100, inner sep=0pt, minimum width=4pt] at (4,0) {};
\node (uu)[label = right:$u^\ell$][circle, draw, fill=black!50, inner sep=0pt, minimum width=4pt] at (4.3,1.5) {};

\node (vv0)[][circle, draw, fill=black!50, inner sep=0pt, minimum width=4pt] at (2,-1) {};
\node (vv1)[][circle, draw, fill=black!50, inner sep=0pt, minimum width=4pt] at (3,-1.33) {};
\node (vv2)[][circle, draw, fill=black!50, inner sep=0pt, minimum width=4pt] at (3,-0.66) {};
\node (1+)[label = below:$1_+$][circle, draw, fill=black!50, inner sep=0pt, minimum width=8pt] at (3,-2) {};

\draw (1+) .. controls (2,-2) .. (vv0); 
\draw (1+) edge (vv1)  
 	  (vv1) edge (vv2)
 	  (vv2) edge (w1)
 	  (vv0) edge (w0); 
\draw (1+) .. controls (4,-2) .. (w2);  
\path[->] 
(5.3, 1.6) edge (6.2, 2.5);
\draw (6.5, 2.8) node {$f_+^\ell(u)(u^\ell)$};

%%%%%%%%%%%%%%% Second rectangle, inside the blob %%%%%%%%%%%%%%%%%%%%%%%%%%%%%%%%%%%%%%%%%%

\fill[rounded corners, green!50] (12.1,4) -- (15,2.8) -- (15,4)-- cycle;
\draw[rounded corners] (9, 2) rectangle (15, 4) {};

\node (w0)[circle, draw, fill=red!100, inner sep=0pt, minimum width=4pt] at (11,2) {};
\node (w1)[circle, draw, fill=red!100, inner sep=0pt, minimum width=4pt] at (12,2) {};
\node (w2)[circle, draw, fill=red!100, inner sep=0pt, minimum width=4pt] at (13,2) {};
\node (uu)[label = right:$u^\ell$][circle, draw, fill=black!50, inner sep=0pt, minimum width=4pt] at (13.3,3.5) {};

\draw (12,6 ) node {$\TGE$};
\path[draw,use Hobby shortcut,closed=true]
(9,-3) .. (11,-2) .. (13,-3) .. (16, 3) .. (15, 6) .. (12, 5) .. (11, 5) .. (8,4) .. (8,2);

\node (root)[label = below:$1$][circle, draw, fill=black!50, inner sep=0pt, minimum width=8pt] at (11,-2) {};
\path[dashed, ->]
(root) edge (w0)
(root) edge (w1)
(root) edge (w2);
\end{tikzpicture}
\caption{Embedding the end-space $f_+^\ell(u)(u^\ell)$, marked as a green shaded area in the graph, into $\TGE$. The distance from $1$ to the red vertices is reflected entirely in the labellings of the vertices, ensuring the intrinsic geometry of the green shaded area is compatible with its extrinsic geometry arising from the embedding. In particular, any geodesics inside $f_+^\ell(u)$ (involving none of the finitely many added vertices or $1_+$) remain geodesics when embedded inside $\TGE$.}
\label{Fig: Embedding f(ell)(endspace) into TGE.}
\end{figure}
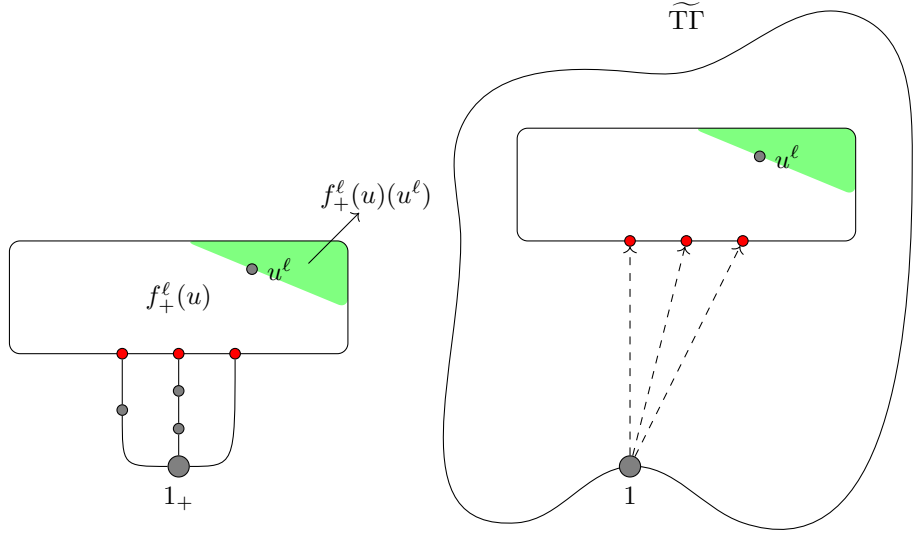

\vspace{0.5cm}

($\mathcal{C}$) Let $u \in \mathcal{B} \subseteq \mathcal{A}$. As shown in Lemma~\ref{Lem: Decompose TGE into finitely many types}, we have that $f(u) = f(v) = \Gamma_{\alpha}$ is a context-free graph. Let $\{ \Gamma_{\alpha, 0}, \Gamma_{\alpha, 1}, \dots, \Gamma_{\alpha, \mu}\}$ be a complete set of representatives of the end-isomorphism classes of $\Gamma_{\alpha}$ for some $\mu \in \mathbb{N}$, and let $f^\ell_i(u)$ denote the copy of $\Gamma_{\alpha, i}$ inside $f^\ell(u)$. We introduce an equivalence relation $\approx_C$ on $\mathcal{B}$ with classes
\[
\mathcal{C}_i = \{ u \in V(\TGE) \mid f^\ell(u)(u^\ell) \cong f^\ell_i(u)\}
\]
for $0 \leq i \leq \mu$, and fix some $\mathcal{C} := \mathcal{C}_\gamma$. In particular, for $u, v \in \mathcal{A} \cap \mathcal{B} \cap \mathcal{C}$, there is now an end-isomorphism $\Phi \colon f^\ell(u)(u^\ell) \to f^\ell(v)(v^\ell)$. 

\vspace{0.5cm}

($\mathcal{D}$) Let $u, v \in \mathcal{C} \subseteq \mathcal{B} \subseteq \mathcal{A}$. There is a natural way to extend the definition of $\sim_{L, H}$ to $\TGE$, by setting $u \sim_{L, H} v$ if and only if $u^\ell$ and $v^\ell$ are $\sim_{L,H}$-related in $f^\ell(u) = f^\ell(v)$. We can use this relation to relate vertices of $f^\ell(u)(u^\ell)$ and $f^\ell(v)(v^\ell)$ via $\sim_{L, H}$. It is clear that since $u, v \in \mathcal{A}$ we have that there is a graph isomorphism $\phi_{u^\ell, v^\ell} : [u^\ell]_{\sim_{L,H}} \to [v^\ell]_{\sim_{L, H}}$. However, this graph isomorphism may not compose with the earlier end-isomorphism $\Phi$ to form a new end-isomorphism, as the frontier points may differ. Thus, we will add the assumption that for all frontier points $f_i \in f^\ell(u)(u^\ell)$, we have that $\phi_{u^\ell, v^\ell}(f_i) = \Phi(f_i)$, where $\Phi$ is the earlier end-isomorphism. Since there are only finitely many frontier points, and since the size of each $\sim_{L,H}$-class is bounded, we note that this induces an equivalence relation $\approx_D$ on $\mathcal{C}$ with finitely many classes. For later, we fix a single equivalence class $\mathcal{D}$ of this relation. 

\vspace{0.5cm}

$(\mathcal{E})$ For any vertex $u \in V(\TG)$, it is clear that $\TG(u)$ only contains vertices of depth $j \geq d(u)$. However, it is possible that $\TGE(u)$ contains vertices of depth $j$ strictly less than $d(u)$. To deal with this, we will be actively using the bounded folding condition. In particular, we will show that there exists a fixed constant $N$ such that for all $u \in V(\TGE)$, we have $d(w) \geq d(u) - N$ for any vertex $w \in \TGE(u)$. We will use this property to construct our final equivalence relation. Thus, choose an arbitrary $u \in V(\TGE)$, let $d(u) = k$, and let $u = [(u_0, \dots, u_k)]$. Let $s_0 = [(u_0, \dots, u_{k-1})]$. Then $d(s_0) = k-1$, as it cannot be any shorter without making $d(u) < k$. Similarly, for $1 \leq i \leq k+1$, let $s_i = [(u_0, \dots, u_{k-i+1}]$, in which case $d(s_i) = k - i -1$. Note that in $\TG$, any walk from the root to $(u_0, \dots, u_k)$ must necessarily pass through $(u_0, \dots, u_i)$ for all $0 \leq i \leq k$. As $\TGE$ is obtained from $\TG$ by determinising, and since the bounded folding condition holds, it follows that any shortest walk from $\ro$ to $s_i$ must necessarily pass through an element of the $\sim_{L, H}$-class of $s_{i-1}$. Furthermore, it follows that the number of elements in the $\sim_{L,H}$-class of $s_{i}$ is monotonically decreasing as $i$ increases. Since the size of any $\sim_{L,H}$-class is uniformly bounded across $\TGE$, say by $N$, it follows that $i$ can only decrease $N$ steps from $k$. Hence the only $s_i$ that may be present in $\TGE(u)$ are those with $d(s_i) \geq d(u) - N$. Finally, if $w$ is a vertex such that $w \in \TGE(u)$ and $d(w) < u$, then all shortest walks from $u$ to $w$ in $\TGE(u)$ must pass through $s_{d(u)-d(w)-1}$, by lifting any such walks to $\TG$ and applying the bounded folding condition again. Hence $d(u)-d(w)-1 \leq N$, and hence $d(w) \geq d(u) - N - 1$, and we have our claim.

We now introduce an equivalence relation on $\mathcal{D}$, based on the above claim. We will assign to $u$ the tuple
\[
s(u) := (s_{\max(0, d(u)-N)}, s_{\max(0, d(u)-N)+1}, \dots, s_{d(u)-1}, s_{d(u)})
\]
and note that the length of $s(u)$ is bounded uniformly above by $N+1$. Assume $u, v \in \mathcal{D} \subseteq \mathcal{C} \subseteq \mathcal{B} \subseteq \mathcal{A}$. We say that $s(u) \sim_s s(v)$ if and only if:
\begin{enumerate}[(a)]
\item The sequences $s(u)$ and $s(v)$ have the same length;
\item $f(s(u)_j) = f(s(v)_j)$, and $f_+^\ell(s(u)_j)(s(u)^\ell_j) \sim f_+^\ell(s(v)_j)(s(v)^\ell_j)$, for all $j$.
\item For all $j$, we have $s(u)_j$ and $s(v)_j$ represent the same vertex of the graphs $[s(u)_j]_{\sim_{L, H}}$ and $[s(v)_j]_{\sim_{L, H}}$, respectively. That is, we have that the fixed isomorphism $\phi_{s(u)_j, s(v)_j}$, which exists by (b), maps $s(u)_j$ to $s(v)_j$. 
\end{enumerate}
Clearly, each of the above conditions induces an equivalence relation with finitely many classes. Hence, intersecting them all to a relation $\sim_s$, it follows that the total number of different equivalence classes of sequences $s(u)$ under $\sim_s$ is uniformly bounded. We define the relation $\approx_E$ on $\mathcal{D}$ by, for any $u, v \in \mathcal{D}$, setting $u \approx_E v$ if and only if $s(u) \sim_s s(v)$. This has finitely many equivalence classes, of which we fix a single one, denoted $\mathcal{E}$.

We have now defined all our equivalence relations. This induces an equivalence relation on all of $V(\TGE)$ in the following way. Let $u, v \in V(\TGE)$. If we have $u \not\approx_A v$, then we also say that $u \not\approx v$. On the other hand, if $u \approx_A v$, with both of them, say, in the equivalence class $\mathcal{A}_j$, then there is an equivalence relation $\approx_B$ defined on this equivalence class, constructed as in part $(\mathcal{B})$ above. If $u \not\approx_B v$, then we say that $u \not\approx v$, and otherwise we proceed in the same way, continuing with $\approx_C$, then $\approx_D$, and finally $\approx_E$; at this stage, we say $u \approx v$ if and only if $\approx_E$. 

Thus, by construction $\approx$ only has finitely many equivalence classes on $V(\TGE)$. Let now $u, v \in V(\TGE)$ be such that $u \approx v$. Fix an equivalence class $\mathcal{E}$ of $\approx$, such that $u, v \in \mathcal{E}$.

We will show by induction on the depth $d(u)$ (and, by symmetry, of $d(v)$) that we can extend the fixed end-isomorphism $\Phi \colon f^\ell(u)(u^\ell) \to f^\ell(v)(v^\ell)$ to an end-isomorphism $\Phi^\ast \colon \TGE(u) \to \TGE(v)$, which would complete the proof of the theorem.

\textbf{Base Case ($d(u) = 0$).}

First, note that all the vertices $w \in \TGE$ at the same depth as $u$ is by definition the vertices of the graph $f^\ell(u)$. Since $d(u) = 0$ in this base case, we must have that $f^\ell(u) \cong \TGE(0) \cong \Gamma$ by assumption. This copy clearly receives the labelling of $L$ corresponding to only labelling the root by $0$, and nothing else. However, since the below argument applies to all labellings $p \in L$ and consequently will also be useful in the inductive step, we will write it out in full generality. 

Since by part $(\mathcal{B})$ we have that $g(u) = g(v) = \beta$, we know by the argument at the end of part $(\mathcal{B})$ that for all $w^\ell$ in $f^\ell(u)$ we have $|w|_{\TGE} = c_1 + |w^\ell|_{f_+^\ell(u)}$ for some fixed constant $c_1$, depending only on $u$. Accordingly, for $w^\ell$ in $f^\ell(v)$, we have $|w|_{\TUE} = c_2 + |w^\ell|_{f_+^\ell(v)}$, for some fixed $c_2$. Thus, defining $\Phi^\ast(w) := \Phi(w^\ell)$ for all $w$ in $f^\ell(u)$, we see that for any two vertices $w_1, w_2 \in f_+^\ell(u)(u^\ell)$, we have
\begin{align*}
d_{\TGE}(w_1, w_2) &= \bigg| |\Phi^\ast(w_1)|_{\TGE} - |\Phi^\ast(w_2)|_{\TGE}\bigg| \\
&= \bigg| \left( c_2 + |\Phi^\ast(w_1)^\ell|_{f_+^\ell(v)} \right) - \left( c_2 + |\Phi^\ast(w_2)^\ell|_{f_+^\ell(v)}\right) \bigg| \\
&= \bigg| |\Phi(w_1^\ell)|_{f_+^\ell(v)} - |\Phi(w_2^\ell)|_{f_+^\ell(v)}\bigg| \\
&= d_{f^\ell(v)}(w^\ell_1, w^\ell_2).
\end{align*}
Thus $\Phi^\ast$ preserves the intrinsic geometry between the vertices in $f^\ell_+(u)(u^\ell)$, and since $f_+^\ell(u)(u^\ell) \cong f_+^\ell(u)(v^\ell)$, it follows that $\Phi^\ast$ extends to an end-isomorphism on the vertices of depth $0$. See Figure~\ref{Fig: Embedding f(ell)(endspace) into TGE.}.

We now just need to check that $\Phi$ can also be extended to the vertices of depth greater than $0$. Note that it is possible to find a copy of a graph isomorphic to $f_+^\ell(u)(u^\ell)$ in $\TG$. This is done by considering the set of all vertices which map to $f_+^\ell(u)(u^\ell)$ under the quotient by $\eta$, and taking (any) single connected component of the graph induced on the subset of vertices of smallest depth in this set. Now, for every branch point in this copy one finds a graph isomorphic to $\TG$ attached; furthermore, the choice of connected component above clearly does not affect this. Also, if $w$ is an branch point in our copy, then if we can find vertices $w_1 = (w, w_1')$ and $w_2 = (w, w_2')$, where $w_1', w_2'$ denote some (possibly empty) tuples of elements of $\Gamma$, such that $w_1 \eta w_2$, then we will also have $\hat{w}_1 \eta \hat{w}_2$, where $\hat{w}_1 = (w', w_1')$ and $\hat{w}_2 = (w', w_2')$ for any other branch $w'$ in the copy with $w \sim_{L, H} w'$. Furthermore, $(w, w_1') \eta (w', w_2')$ only if $w \sim_{L,H} w'$; this is because no folding can start inside, but then end up outside of the $\sim_{L,H}$-class of $w$, or vice versa, as any such folding would by $S$-fullness necessarily pass through one of the vertices in the orbit of $\ro \not\in S$, contradicting the fact that $\Gamma$ is $S$-overlap-free.

The facts in the above paragraph have two major consequences. The first is that the image under $\eta$ of our copy of $f_+^\ell(u)(u^\ell)$, together with its attached copies of $\TG$ at every branch point, is precisely $\TGE(u)$ restricted to vertices of depth $\geq 0$ (and hence isomorphic to all of $\TGE(u)$). The second is that $\eta$ is entirely determined by its behaviour on the individual $\sim_{L,H}$-classes and the copies of $\TG$ attached therein. 

In view of this, we now divide $f_+^\ell(u)(u^\ell)$ into $\sim_{L,H}$-classes. It is clear that of these, only finitely many are such that they are missing some of the vertices contained in the corresponding $\sim_{L,H}$-class in $f^\ell(u)$. Each such class missing some vertices must clearly all contain at least one frontier point of $f_+^\ell(u)(u^\ell)$. Analogously considering $f_+^\ell(v)(v^\ell)$, we see that for each such frontier point $f_i$ in $f_+^\ell(u)(u^\ell)$ that, by the assumption that $u, v \in \mathcal{D}$, that $\phi_{u^\ell, v^\ell}(f_i) = \Phi(f_i)$. In particular, it follows that $\Phi$ maps any $\sim_{L, H}$-class $K$ of $f_+^\ell(u)(u^\ell)$ to a $\sim_{L,H}$-class isomorphic to $K$ in $f_+^\ell(v)(v^\ell)$. 

Thus combining this information, i.e.\ that the $\sim_{L, H}$-classes of $f_+^\ell(u)(u^\ell)$ are identical to those of $f_+^\ell(v)(v^\ell)$, with the two facts above, it follows that $\Phi$ also extends to an end-isomorphism between the graph induced on the vertices of depth $> 0$ in $\TGE(u)$ to those of depth $> 0$ in $\TGE(v)$, and hence, by the earlier argument that it extends to one on the vertices of depth $=0$, extends to an end-isomorphism $\Phi^\ast \colon \TGE(u) \to \TGE(v)$. 

\vspace{0.5cm}

\textbf{Inductive step; $d(u) = n > 0$, holds for all $i < n$.}

Now, when $d(u) = n > 0$, it is possible that some vertices of depth strictly less than $n$ are included in $\TGE(u)$, unlike in the case when $d(u) = 0$. We must therefore ensure that $\TUE(v)$ has such vertices present, too, and that we may map the former to the latter by an end-isomorphism. This will make heavy use of the relation associated with the set $\mathcal{E}$. 

First, if there are no such lower depth vertices, then the same argument as for the base case directly applies, in precisely the same manner, with the only modification necessary being to replace the vertices $w, w'$ by tuples of vertices rather than single vertices; this is well-defined as $\sim_{L, H}$ can be extended to all of $\TG$. 

Now, assume instead there is some vertex $w \in \TGE(u)$ with $d(w) < n$. By the argument given when defining $\mathcal{E}$ and since $u, v \in \mathcal{E}$, we see that we can find some vertex $w_u$ with the properties that
\begin{enumerate}
\item $|w_u|_{\TGE} = |u|_{\TGE}$;
\item $w_u$ is such that $d(w_u) < d(u)$ is minimised.
\end{enumerate}
In the notation from earlier, the above follows from the fact that at least one of the vertices $s'$ in the $\sim_{L,H}$-class of $s_0$ must satisfy $|s'|_{\TGE} < |u|_{\TGE}$, and thus for any vertex $w \in \TGE(u)$ with $d(w) < u$ and $|w|_{\TGE} > |u|_{\TGE}$, there lies on the shortest path from $u$ to $w$ a vertex $w' \in \TGE(u)$ with $|w|_{\TGE} = |u|_{\TGE}$. Inductively continuing this argument for $s_0, s_1, \dots, s_N$, we find our vertex $w_u$ with the desired properties. Thus as a direct consequence, since $s(u) \sim_s s(v)$ by assumption, we must have that we can also find a vertex $w_v$ satisfying the properties above, and furthermore, by definition of $\sim_s$, we must first of all have $f(w_u) = f(w_v)$. Indeed, by condition (b) and (c) imposed when defining $\mathcal{E}$, we even have that $f_+^\ell(w_u)(w_u^\ell) \cong f_+^\ell(w_v)(w_v^\ell)$. Since $d(w_u) < d(u)$, it follows by the inductive hypothesis that we have an end-isomorphism $\Phi^\ast \colon \TGE(w_u) \to \TGE(w_v)$. But since $|w_u|_{\TGE} = |u|_{\TGE}$ and $|w_v|_{\TGE} = |v|_{\TGE}$, we have 
\[
\TGE(u) = \TGE(w_u) \sim \TGE(w_v) = \TGE(v)
\]
and hence we have an end-isomorphism $\Phi^\ast \colon \TGE(u) \to \TGE(v)$. 

Hence we have shown that $(u \approx v) \implies (\TGE(u) \sim \TGE(v))$, and since $\approx$ only has finitely many equivalence classes, it follows that $\TGE$ has only finitely many ends under end-isomorphism, and hence is a context-free graph.
\end{proof}

We will now turn towards our main focus of study, to which we will apply the above theorem.

\section{Special Monoids}\label{Sec: Special Monoids}

We will begin this section by fixing some notation pertaining to monoids, which will be used throughout without further explanation. A \textit{monoid presentation} for a monoid $M$ is a presentation $\pres{Mon}{A}{u_i = v_i \: (i \in I)}$ such that the monoid presented by this presentation is isomorphic to $M$. A good introduction to the theory of monoid presentations can be found in e.g. \cite{Book1993}. If $v_i$ is the empty word, then we will write the relation $u_i = v_i$ as $u_i = 1$, for notational convenience. Every group presented by a group presentation $\pres{Gp}{A}{R_i = 1 \: (i \in I)}$ admits a monoid presentation as $\pres{Mon}{A \sqcup \bar{A}}{(R_i' = 1 \: (i \in I)) \cup (a \bar{a} = 1, \: \bar{a}a = 1 \mid (a \in A))}$ where $\bar{A}$ is in involutive correspondence with $A$ such that $A \cap \bar{A} = \varnothing$, and $R_i'$ is the word obtained from $R_i$ by replacing any occurrence of a letter $a^{-1}$ by $\bar{a}$, for all $a \in A$. In a monoid presentation as above, each pair $(u_i, v_i)$ is called a \textit{defining relation}. If every defining relation is of the form $(u_i, \varepsilon)$, where $\varepsilon$ denotes the empty word, then we say that the monoid presentation is \textit{special}, and we say that such $u_i$ is a \textit{relator word}. If $M$ is a monoid admitting a special presentation, then we say that $M$ is a \textit{special monoid}. Clearly, every group is a special monoid. If $M$ is a special monoid, then we will often slightly abusively identify this monoid with a special presentation for it. 

For the remainder of this section, we will fix a special monoid $M = \pres{Mon}{A}{R_i = 1, \: (i \in I)}$ be a special monoid. We let $A^\ast$ denote the free monoid on $A$, with identity element the empty word $\varepsilon \in A^\ast$. Then there is a canonical homomorphism $\pi : A^\ast \to M$ associated to this presentation for $M$. If $u, v \in A^\ast$ are graphically equal, then we will write $u \equiv v$, and if $\pi(u) = \pi(v)$ in $M$ then we will write $u =_M v$. 

Let $m \in M$. If there exists $n \in M$ with $mn = 1$, then we say that $m$ is \textit{right invertible}. If there exists $n \in M$ with $nm = 1$, then we say that $m$ is \textit{left invertible}. If $m$ is left and right invertible, then we say that $m$ is \textit{invertible}. We note that $m$ being invertible is equivalent to the existence of a unique $n \in M$ such that $mn = nm = 1$. The set of all invertible elements of $M$ form a submonoid of $M$ under the same multiplication. This monoid is called \textit{the group of units} of $M$, and is denoted $U(M)$. Analogously, the (left) right invertible elements also form a submonoid, the \textit{right units} $M$, which is denoted $R(M)$, and analogously for the \textit{left units} $L(M)$. We note that in all three cases, the identity element is always $1 \in M$. 

\subsection{String rewriting}\label{Subsec: String rewriting}

The application of the techniques of string rewriting techniques to semigroup theory has been incredibly fruitful. Particularly in the work of McNaughton, Narendran, Otto, and Zhang around the 1990s, many fundamental results were shown, forming a solid basis for the connection of the two areas. Some of the main articles constituting their work on applying string rewriting to special monoids can be summarised by the following list:  \cite{Kapur1985, McNaughton1987, Narendran1991, Otto1984, Otto1991, Otto1992, Otto1995, Zhang1991, Zhang1992, Zhang1992b, Zhang1992c, Zhang1992d, Zhang1994, Zhang1996}.

We now give a few brief examples of purely semigroup-theoretical results that have been derived through string rewriting, illustrating its importance to the area. These can all be found in the above listed references. Let $M$ be a finitely presented special monoid. If the centre of $M$ is non-trivial, then $M$ is a group or the infinite cyclic monoid $\mathbb{N}$. If $M$ has only finitely many Green's $\mathscr{R}$- or $\mathscr{L}$-classes, then $M$ is a group. It is undecidable, given a finite special presentation for $M$, whether $M$ is a group. Assume that every generator of $M$ appears in some defining relation of $M$. Then $M$ is cancellative if and only if it is left cancellative, if and only if it is right cancellative, if and only if it is a group. Furthermore, if $M$ is not a group, then $M$ contains a non-trivial idempotent.

Let $A$ be an alphabet. A \textit{string-rewriting system} $R$ on $A$ is a subset of $A^\ast \times A^\ast$. The elements of $R$ are called the \textit{rules} of the system. A string-rewriting system is called \textit{special} if for all rules $(\ell, r) \in R$, we that $\ell$ is non-empty, and $r \equiv \varepsilon$. A string-rewriting on $A$ is thus named because it induces a number of binary relations on $A^\ast$. We define the \textit{single-step reduction} $\xr{R}$ for two words $u, v \in A^\ast$ as follows: 
\[
u \xr{R} v \qquad \iff \qquad \exists x, y \in A^\ast, \: \exists (\ell, r) \in R \: : u \equiv x\ell y \quad \text{and} \quad v \equiv xry.
\]
The symmetric closure of $\xr{R}$ is denoted $\lr{R}$. The \textit{reduction relation} $\xra{R}$ induced by $R$ is the reflexive and transitive closure of the relation $\xr{R}$. Analogously, we let $\lra{R}$ denote the reflexive and transitive closure of $\xr{R}$. Now, $\lra{R}$ is by definition an equivalence relation on $A^\ast$. However, it is also compatible with the concatenation of words in $A^\ast$, and hence is a congruence on $A^\ast$. This congruence is called the \textit{Thue congruence} generated by $R$. 

For the remainder of this Section~\ref{Subsec: String rewriting}, let $M$ denote the finitely presented special monoid defined by the presentation $\pres{Mon}{A}{R_i = 1 \: (i \in I)}$. Then $\{ (w_i, \varepsilon) \}$ forms a special string-rewriting system on $A$. Clearly, $M$ is isomorphic to the quotient of $A^\ast$ by the Thue congruence generated by this rewriting system. We will denote the Thue congruence associated to this string rewriting system by $\lra{M}$. Let $\pi \colon A^\ast \to M$. Then, for a word $u \in A^\ast$, we will identify the congruence class $[u]_{\lra{M}}$ with the element $\pi(u) \in M$. We say that a word $u \in A^\ast$ is \textit{invertible modulo $M$} if $[u]_{\lra{M}}$ is invertible, and analogously for left and right invertibility.

The notation $\lra{M}$ is slightly abusive, as a given special monoid $M$ can of course in general be defined by many different presentations, each of which would give rise to a different Thue congruence, making $\lra{M}$ ill-defined. However, as all the special monoids considered in this paper will be associated with a fixed special presentation, this concern can be waivered, and since the notation aids in readability, it will therefore be opted for throughout the remainder. We note that this abuse is no greater than denoting the canonical homomorphism associated to the above presentation for $M$ by $\pi_M$. 

As an example of the type of result which is not difficult to show using string-rewriting techniques, we have the following lemma, which appears in the literature as \cite[Lemma~1.6]{Otto1984}, and which will prove a useful tool for later sections. 

\begin{lemma}\label{Lem: Special monoid has ancestors}
Assume that $u, v \in A^\ast$ are such that $u \lra{M} v$. Then there exists some word $w \in A^\ast$ such that $w \xra{M} u$ and $w \xra{M} v$. 
\end{lemma}

That is, any two congruent words have a common ancestor. This is not true in general for string-rewriting systems. Furthermore, this proof is constructive, even if one cannot solve the word problem: given two words $u, v \in A^\ast$ for which one knows that $u \lra{M} v$, it is easy to explicitly construct a $w \in A^\ast$ satisfying the above properties.

\subsection{Identical subwords}

We will now mention briefly that, given a special presentation $\pres{Mon}{A}{R_i = 1 \: (i \in I)}$, we can without loss of generality assume that none of the $R_i$ contains a proper non-empty subword $u$ such that $\pi(u) = 1$. For if $R_i \equiv w u w^\prime$ such that $\pi(u) = 1$, then we may replace the relation $R_i = 1$ by the two relations $ww^\prime = 1$ and $u = 1$, generating the same congruence on $A^\ast$, and one may hence, starting with a presentation for $M$, obtain a presentation defining the same congruence on $A$ (and hence in particular also defining $M$), but which also satisfies the above condition on the defining relations. We remark, however, that this procedure is not effective; for that, we would require a solution to the word problem. 

As an aside, of slightly independent interest, we note that in the one-relator case, every presentation already satisfies this condition by the following chain of reasoning. The following theorem appears as \cite[Theorem~2]{Weinbaum1972}.

\begin{theorem}[Weinbaum]\label{Thm: One-relator, no subword is = 1}
Let $G = \pres{Gp}{A}{w = 1}$ be a one-relator group such that $w$ is a cyclically reduced word involving all the generators. Then no proper subword $w'$ of $w$ satisfies $w' =_G 1$.
\end{theorem}

The proof of this theorem is rather long, but we mention that the fundamental difference in the one-relator case, when compared to the general, which permits a proof of the above theorem is, not entirely unexpectedly, Magnus' \textit{Freiheitssatz} \cite{Magnus1930}.

\begin{corollary}
Let $M = \pres{Mon}{A}{w=1}$ be a special one-relator monoid with $w$ a word involving all the generators. Then no proper subword $w'$ of $w$ satisfies $w' =_M 1$. 
\end{corollary}
\begin{proof}
Let $G = \pres{Gp}{A}{w=1}$, and let $\pi_M, \pi_G$ be the homomorphisms associated to the presentations for $M$ and $G$, respectively. Let $\phi \colon M \to \pres{Gp}{A}{w=1}$ be defined by $\pi_M(a_i) \mapsto \pi_G(a_i)$ for all $a_i \in A$. Then $\phi$ clearly extends to a homomorphism. Let $w'$ be a subword of $w$ such that $\pi_M(w') = 1$. But then $\pi_G(w') = \phi(\pi_M(w')) = \phi(1) = 1$. Hence $w' =_G 1$. Furthermore, $w$ is cyclically reduced as it is a positive word, and all generators $a_i \in A$ appear in $w$. Thus we have a contradiction to Theorem~\ref{Thm: One-relator, no subword is = 1}, and no such subword $w'$ exists. 
\end{proof}

The analogue of the above proposition fails already for two relations, and indeed even for special one-relator monoids defined by presentations with more than one relator; an obvious counterexample is the monoid presented by the presentation $\pres{Mon}{a,b,c}{abc = 1, b = 1}$, isomorphic to the bicyclic monoid. We finally note that the analogue of Theorem~2 holds for $C'(1/6)$-groups too, by \cite[p. 558]{Weinbaum1966}. However, the tempting generalisation of the above corollary to special monoids $M$ with $U(M)$ a $C'(1/6)$-group is false by virtue of the above example, as the group of units of the bicyclic monoid is trivial and hence trivially $C'(1/6)$.

\subsection{Invertible Pieces}\label{Subsec: Invertible pieces}

The invertible pieces of a special presentation are some of the most fundamental objects of study. We first recall some terminology from \cite{Zhang1992}. Throughout the remainder of Section~\ref{Sec: Special Monoids}, we fix $M = \pres{Mon}{A}{R_i = 1 \: (i \in I)}$ be a finitely presented special monoid, which satisfies the assumption from the previous section that no $R_i$ contains a subword congruent to $1$. A non-empty invertible word is called \textit{minimal} if its length does not exceed the longest relator word of the presentation, and none of its non-empty prefixes is invertible. The set of all minimal words forms a biprefix code as a subset of $A^\ast$. As every relator word $R_i$ is invertible, we may uniquely decompose each word $R_i$ into invertible factors $R_{i,j} \in A^\ast$, all of which are minimal words. We shall denote this factorisation as
\[
R_i \equiv R_{i,1} R_{i,2} \cdots R_{i, \ell}.
\]
We will refer to these minimal invertible factors $R_{i,j}$ as the \textit{invertible pieces} of the presentation. We define $\Delta$ as the set of all minimal words that are congruent to some invertible piece. The partition of $\Delta$ induced by the equivalence relation $\lra{M}$ is denoted $\Delta_1 \cup \Delta_2 \cup \cdots \cup \Delta_\kappa$. 

In \cite{Zhang1992}, the set $\Delta$ is the main set of study for the invertible elements of $M$. This set has many nice properties; in particular, using it one may construct a complete rewriting system $S(M)$ for solving the word problem in $M$. However, for the geometric properties which we will investigate in later sections, we will primarily be interested in a smaller set. Let $\Lambda$ be the collection of all minimal invertible pieces $R_{i,j}$. For brevity, will refer to elements of $\Lambda$ as \textit{pieces}. To ensure this definition is clear, $\Lambda \subseteq \Delta$ consists of all words in $\Delta$ which furthermore also appear in the above factorisation of some $R_i$. Note that every element of $\Lambda^\ast$ is invertible. In general, not every element of $\Delta$ is a piece, and so $\Lambda$ consists of fewer elements than $\Delta$. In general $\Lambda$ also may have more than one representative for each $\Delta_i$. That is, two distinct pieces may represent the same element of $M$. Occasionally, for ease of notation, we will denote by $\Lambda_\varepsilon$ the set $\Lambda \cup \{ \varepsilon \}$.

For every $\lambda \in \Lambda$, by definition there exists exactly one $\Delta_i$ such that $\lambda \in \Delta_i$. We accordingly write $\Lambda = \{ \lambda_{1,1}, \dots, \lambda_{i,j}, \dots\}$ to indicate when $\lambda_{i,j} \in \Delta_i$. Let $\fB = \{b_{1,1}, \dots, b_{i,j}, \dots \}$ be a set in bijective correspondence with $\Lambda$ via $b_{i,j} \mapsto \lambda_{i,j}$. The set $\Lambda / \lra{M}$ contains $\kappa$ elements, since $\Delta = \Delta_1 \cup \Delta_2 \cdots \cup \Delta_\kappa$, each one a unique representative for each of the $\Delta_i$, as we by definition of $\Delta_i$ have $\lambda_{i,j} \lra{M} \lambda_{k, \ell}$ if and only if $i=k$. The following proposition is more or less obvious, but we spell out its proof to ensure the distinction between $\Delta$ and $\Lambda$ is facilitated.

\begin{prop}\label{Prop: Lambda generates the invertibles}
$\Lambda$ generates the set of invertible words. More specifically, for every invertible word $u \in A^\ast$ there exists $\lambda \in \Lambda^\ast$ such that $u \lra{M} \lambda$.  
\end{prop}
\begin{proof}
Let $u \in A^\ast$ be invertible. By Lemma~3.4 of \cite{Zhang1992}, we have $u \in \Delta^\ast$. If we write $u \equiv u_1 \cdots u_k$ for $u_i \in \Delta$, then for every $u_i$ there exists a $\Delta_j$ such that $u_i \in \Delta_j$, by definition. By definition, for every $\Delta_i$ there exists $\lambda_i \in \Lambda$ such that $\lambda_i \in \Delta_i$. In particular, since all elements of $\Delta_j$ represent the same element of $M$, for every $u_i$ there exists $\lambda_i \in \Lambda^\ast$ such that $u_i \lra{M} \lambda_i$. Hence \[
u \equiv u_1 \cdots u_k \lra{M} \lambda_1 \lambda_2 \cdots \lambda_k \in \Lambda^\ast
\]
and since $u \in A^\ast$ was arbitrary, we are done.
\end{proof}

If $u \in A^\ast$, then we say that an invertible subword of $u$ is \textit{maximal} if it is not contained in any larger invertible subword of $u$. As the invertible words of $A^\ast$ form a biprefix code (see e.g. \cite[Corollary~3.3]{Kobayashi2000}) it follows that every invertible subword of any word is contained in a unique maximal invertible subword. 

\begin{lemma}[Normal Form Lemma \cite{Otto1991}]\label{Lem: Zhang's Lemma}
Let $u, v \in A^\ast$ be any two words. Suppose that $u \lra{M} v$. Then $u$ and $v$ can be uniquely factored as
\[
u \equiv u_0 a_1 u_1 \cdots a_n u_n \quad \text{and} \quad v \equiv v_0 a_1 v_1 \cdots a_n v_n
\]
such that both the following conditions hold: 
\begin{enumerate}
\item $a_i \in A$;
\item all $u_i$ (resp. $v_i$) are maximal invertible factors of $u$ (resp. $v$) such that $u_i \lra{M} v_i$ for $0 \leq i \leq n$, and at least one of $u_j$ and $v_j$ is non-empty for $1 \leq j \leq n-1$. 
\end{enumerate}
\end{lemma}

As an illustration of the power of this Normal Form Lemma, we have the following proposition, which does not appear in the literature.

\begin{prop}\label{Prop: Invertible contains a piece}
Assume that $u \in A^\ast$ is invertible and non-empty. Then $u$ contains some piece $\lambda \in \Lambda$ as a subword.
\end{prop}
\begin{proof}
Let $u \in A^\ast$ be invertible and non-empty. Then as by Proposition~\ref{Prop: Lambda generates the invertibles} $\Lambda$ generates the set of invertible elements, there exists some $w \in \Lambda^\ast$ such that $u \lra{M} w$. In particular there exists a sequence of invertible words $u \equiv u_0, u_1, \dots, u_k \equiv w$ such that
\[
u_0 \lr{M} u_1 \lr{M} \cdots \lr{M} u_{k-1} \lr{M} u_k.
\]
Write $u_0 \equiv a_1 \cdots a_n$ for $a_i \in A$. We split into two (very similar) cases.

\begin{enumerate}
\item None of the applications of $\lr{M}$ involve any of the $a_i$. Then we may write
\[
w \equiv w_1 a_1 w_2 a_2 \cdots w_{n} a_n w_{n+1}, \quad \textnormal{where} \: w_i \lra{M} \varepsilon \: \textnormal{for $1 \leq i \leq {n+1}$}.
\]
Now $w \in \Lambda^\ast$, and so we may factor $w$ uniquely into minimal invertible pieces as $\Lambda^\ast$ is a biprefix code. By our assumption on the presentation, no piece contains a proper subword that is equal to $1$ in $M$. Hence in a factorisation of $w$ into minimal invertible pieces $\lambda_1 \cdots \lambda_m$ we must have that if a piece $\lambda_i$ contains more than one $a_j$, then these are all adjacent. Hence each $\lambda_i$ can be written in the form $s_j a_j a_{j+1} \cdots a_{j+\ell} p_{j+\ell+1}$, where $s_j$ is some suffix of $w_j$ and $p_{j+\ell+1}$ is some prefix of $w_{j+\ell+1}$. But now $s_j$ is invertible, being a prefix of $\lambda_i$ and a suffix of $w_j$. Symmetrically, $p_{j+\ell+1}$ is invertible; consequently $a_j a_{j+1} \cdots a_{j+\ell}$ is invertible. Furthermore, as $\lambda_i$ cannot be written as a product of non-empty invertible words by assumption of minimality, it thus follows that $s_j \equiv p_{j+1} \equiv \varepsilon$, and $\lambda_i \equiv a_j a_{j+1} \cdots a_{j+\ell}$. Decomposing all such $\lambda_i$ we see that each $a_j$ must occur exactly once inside some $\lambda_i$, and all will appear in the same order as in $a_1, \dots, a_n$. Thus we even have $a_1 \cdots a_n \in \Lambda^\ast$, and since $u \equiv a_1 \cdots a_n \in \Lambda^\ast$ we are done. 

\item There exists a smallest $0 \leq N \leq k-1$ such that $u_N \xr{M} u_{N+1}$ by deleting a relator word $R$ containing some $a_i$. Then we may again write 
\[
u_N \equiv w_1 a_1 w_2 a_2 \cdots w_{n} a_n w_{n+1}, \quad \textnormal{such that} \quad w_i \lra{M} \varepsilon \: \textnormal{for $1 \leq i \leq {n+1}$}
\]
where $R$ now appears as a subword of $u_N$. Since $R$ does not contain any $w_i$ as a subword by our assumption on the presentation, it follows that all occurrences of some $a_i$ inside $R$ must all be consecutive, by the same argument as in $(1)$. Hence we can again write $R \equiv s_j a_j a_{j+1} \cdots a_{j+\ell} p_{j+\ell+1}$, where $s_j$ is some suffix of $w_j$ and $p_{j+\ell+1}$ is some prefix of $w_{j+\ell+1}$. Thus we see that $s_j$ and $p_{j+\ell+1}$ are invertible, and hence so too is $a_j a_{j+1} \cdots a_{j+\ell}$. Factorising $R$ into minimal invertible pieces will hence recognise $a_j a_{j+1} \cdots a_{j+\ell}$ as a product of pieces. Thus $a_j a_{j+1} \cdots a_{j+\ell}$ contains a piece, and hence so too does $u \equiv a_1 \cdots a_n$. 
\end{enumerate}
This completes the proof.
\end{proof}

Note that if our assumption that the relator words $R_i$ do not contain subwords equal to $1$ is violated, then Proposition~\ref{Prop: Invertible contains a piece} is false as stated, as the following example shows.

\begin{example}
Let $M = \pres{Mon}{a,b,c}{b = 1, abc = 1}$. Then $\Lambda = \{ b, abc \}$, and $ac = 1$ is an invertible word containing no piece as a subword. Note that we clearly have that $M \cong \pres{Mon}{a,c}{ac = 1}$, the bicyclic monoid.
\end{example}

Now it is a classical result, due to Makanin, that the group of units of a finitely presented special monoid is again finitely presented. This appears in the context of string-rewriting as \cite[Theorem~3.7]{Zhang1992}, where a finite monoid presentation for the group of units is constructed by means of the set $\Delta$. We will now briefly mention that such a presentation can be obtained by means of the set of pieces $\Lambda$, too.

Let $B = \{ b_1, \dots, b_i, \dots \}$ be a set in bijective correspondence with $\Lambda / \lra{M}$, via the bijection induced by $[\lambda_{i,j}]_{\lra{M}} \mapsto b_i$. We extend this map to a homomorphism $\phi : \Lambda^\ast \to B^\ast$. We note that this $\phi$ is the restriction to $\Lambda^\ast \subseteq \Delta^\ast$ of the map $\phi : \Delta^\ast \to B^\ast$ defined in Zhang \cite{Zhang1992} via $\phi(\delta) = b_i$ if $\delta \in \Delta_i$, and to maintain consistency we use the same symbol for both maps, as the two sets $B$ in question are the same. We define a rewriting system with rules $T_0$ defined as 

\[ T_0 := \{ (s, \varepsilon) \mid \text{$s$ is a cyclic permutation of some $\phi(R_i)$}\}.\]

Note that here $R_j$ denotes a relator word in the presentation of $M$; in particular $T_0$ is a rewriting system over $\phi(\Lambda^\ast)$.
Note that every left-hand side of a rule in $T_0$ is the image under $\phi$ of a word over $\Lambda$. Let $\chi : \fB^\ast \to B^\ast$ be the homomorphism induced by the mapping $b_{i,j} \mapsto b_i$. Then $\chi$ has finite fibers, and 
\[
\mathfrak{T}_0 = \{ (w, \varepsilon) \mid w \in \chi^{-1}(s) \text{ where $s$ is a cyclic permutation of some $\phi(R_j)$} \}
\]
is a string-rewriting system with finitely many rules. Let $\ell(\mathfrak{T}_0)$ denote the set of left-hand sides of rules in $\mathfrak{T}_0$. Then, since by \cite[Theorem~3.7]{Zhang1992} we have $U(M) \cong B^\ast / \lra{T_0}$, it follows that
\[
\pres{Mon}{\fB}{\{ b_{i,j} = b_{i,k} \mid 1 \leq i \leq \kappa; \: 1 \leq j, k \leq |\Delta_i| \} \cup \{ \mathfrak{t} = 1 \mid \mathfrak{t} \in \ell(\mathfrak{T}_0) \} }
\]
is a finite monoid presentation for $U(M)$, via the isomorphism induced by the map $\chi : \fB^\ast \to B^\ast$ defined as earlier by $b_{i,j} \mapsto b_i$. We make a few remarks regarding this presentation. First, it is not a special monoid presentation, as the relations $b_{i,j} = b_{i,k}$ indicate. However, any monoid presentation for a group can always be transformed into a special monoid presentation with the same alphabet, and which also defines the same group, so this is not a restriction. Furthermore, many generators of this presentation are redundant, as again witnessed by the relations $b_{i,j} = b_{i,k}$ indicate. However, we will not use e.g. Tietze transformations to simplify the presentation by removing these generators; their presence is crucial for the construction of the Cayley graph of $M$.

\subsection{Right Invertible Elements}

We turn to right invertible elements. These will play a key r\^{o}le in the proof of the main theorem of this article. First, define a set $\Xi$ from the set $\Lambda$ as the set of non-empty prefixes of $\Lambda$, i.e.\ 
\[
\Xi := \{ w \in A^+ \mid \exists y \in A^\ast : xy \in \Lambda \}
\]

Clearly, all elements of $\Xi$ are right invertible. In \cite{Zhang1992}, another set, there denoted as $I$, is used, which is obtained from $\Delta$ exactly as $\Xi$ is obtained from $\Lambda$, and it is proven that this set generates the right invertible words. However, as $\Delta$ is a rather complicated set, it can be difficult to get a handle on this set $I$. Instead, we show the following useful proposition, showing that the prefixes of pieces is enough to generate the right invertible elements.

\begin{prop}\label{Prop: Xi generates the right invertibles}
$\Xi$ generates the set of all right invertible words. That is, for any right invertible word $u \in A^\ast$, there exists $w \in \Xi$ such that $u \lra{M} w$.
\end{prop}
\begin{proof}
Given any $u \in \Delta$, we may find $\lambda \in \Lambda^\ast$ such that $u \xleftrightarrow{\ast}_{T} \lambda$, since $u$ is invertible and $\Lambda$ generates the set of invertible words by Proposition~\ref{Prop: Lambda generates the invertibles}. Hence we may obtain $u$ from $\lambda \in \Lambda^\ast$ by some finite sequence of insertions and deletions of relator words. Now, any proper prefix of $\lambda$ is by definition in $\Xi$. We claim that the property of having any prefix equivalent to a word in $\Xi^\ast$ is preserved under $\lra{M}$, and note that from this it follows that any prefix of $u$ is equivalent to a word in $\Xi^\ast$; and that, since $u \in \Delta$ was arbitrary, any prefix of an element of $\Delta$ is equivalent to a word in $\Xi^\ast$. By \cite[Lemma~4.1]{Zhang1992}, the set of prefixes of elements of $\Delta$ generates the set of all right invertible words, and hence the proposition would follow.

We will thus prove our claim. Let $\xi \in A^\ast$ be a word such that any prefix of $\xi$ is equivalent to a word in $\Xi^\ast$. We write $\xi \equiv a_1 \cdots a_n$ for $a_i \in A$. First, let $\xi^\prime$ be the result of inserting of a relator word $w_\ell$ with $(w_\ell, \varepsilon) \in T$ in $\xi$. Then $\xi^\prime \equiv a_1 \cdots a_{i-1} w_\ell a_{i} \cdots a_n$ for some $i$, with appropriate interpretations of e.g.\ $a_1 \cdots a_{i-1} \equiv \varepsilon$ if $i=1$. Then any prefix $p$ of $\xi^\prime$ satisfies either $p \equiv a_1 \cdots a_j$, or $p \equiv a_1 \cdots a_{i-1} w^\prime$ for some prefix of $w_\ell$, or $p \equiv a_1 \cdots a_{i-1} w_\ell a_i \cdots a_j$. In all three cases, it is clear that $p$ is equivalent to an element of $\Xi^\ast$, as $w^\prime \in \Xi$.

Second, assume that $\xi$ contains a relator word $w_\ell$ as a subword. Then if we write $\xi \equiv a_1 \cdots a_{i-1} w_\ell a_i \cdots a_n$ with $a_i \in A$ as before, let $\xi^\prime \equiv a_1 \cdots a_n$ denote the result of deleting $w_\ell$ from $\xi$. Any prefix $p$ of $\xi^\prime$ now either satisfies $p \equiv a_1 \cdots a_{i-1}$, or $p \equiv a_{1} \cdots a_{k}$ for $k \leq n$. In the first case, $p$ is a prefix of $\xi$, and hence equivalent to an element of $\Xi^\ast$ by assumption. In the second case, $p$ is equivalent to $a_1 \cdots a_{i-1} w_\ell a_i \cdots a_k$, which is a prefix of $\xi$, and hence equivalent to an element of $\Xi^\ast$. This concludes the proof of the claim.
\end{proof}

\begin{example}
Not every right invertible word is graphically an element of $\Xi$. In the bicyclic monoid $\pres{Mon}{b,c}{bc = 1}$, we have that $b(bc)c \xleftrightarrow{\ast}_{T} \varepsilon$ and so $bbcc$ is certainly right invertible, but this word of course cannot be written graphically as a product of elements of $\Xi = \{ b, bc\}$. 
\end{example}

We introduce some notation for convenience. We let $\fP = \Xi \setminus \Lambda \subseteq A^\ast$ be the finite set of \textit{non-empty proper prefixes} of pieces, and set $\fPe = \fP \cup \{ \varepsilon \}$ to be the finite set of \textit{proper prefixes} of pieces. We remark that it is a consequence of the proof of Proposition~\ref{Prop: Invertible contains a piece} that the first letter $a_1$ of an invertible word must be the first letter of some piece; this is clear in case $(1)$, and by the analysis given in case $(2)$, the first removal of a relator word containing $a_1$ must begin in $a_1$, for it could not contain $w_1 \lra{M} \varepsilon$. We record this in the following.

\begin{corollary}\label{Cor: First letter of RI word is prefix}
Let $u \in A^\ast$ be right invertible and non-empty. Then the first letter of $u$ is the first letter of some piece $\lambda \in \Lambda$. 
\end{corollary} 
\begin{proof}
Since $u$ is right invertible, there exists $v \in A^\ast$ such that $uv \lra{M} \varepsilon$ is invertible. But the first letter of $uv$ is the first letter of $u$, since $u$ is non-empty, and the claim follows from the above remark.
\end{proof}

The following definition will be central to the remainder of the article.

\begin{definition}
We denote by $\fR_1 := \fR_1(A)$ and call \textit{the Sch\"{u}tzenberger graph of $1$} the subgraph of $\MCG{M}{A}$ induced on $R(M)$, the set of all right invertible elements of $M$.
\end{definition}

It is clear that $\fR_1$ is the connected component of the identity in $\MCG{M}{A}$. The theory of more general Sch\"{u}tzenberger graphs is studied in great detail in work by Stephen \cite{Stephen1987}, wherein it coincides with the so-called \textit{basic graph} of $1$. For the purposes of this article, however, the above definition is all that will be necessary. As we will generally assume that we have fixed a given finite generating set $A$ for $M$, we will often suppress the $A$ in $\fR_1(A)$ and simply denote it as $\fR_1$. 

We now have sufficiently much background information on special monoids to study the structural properties of their Cayley graphs. One graph in particular, which we denote by $\fU$ and is closely associated to the group of units, will be central to the main results of this paper. We now introduce this graph. 

\section{The Sch\"{u}tzenberger Graph of the Units}\label{Sec: Constructing fU}

Throughout the entirety of this section, we will fix a finitely presented special monoid $M = \pres{Mon}{A}{R_i = 1 \: (i \in I)}$, satisfying only the condition that no $R_i$ contain a subword equal to $1$, and inherit all other notation from the previous sections. Let $U_\fB(M)$ denote the monoid defined by the monoid presentation 
\[
\pres{Mon}{\fB}{\{ b_{i,j} = b_{i,k} \mid 1 \leq i \leq \kappa; \: 1 \leq j, k \leq |\Delta_i| \} \cup \{ \mathfrak{t} = 1 \mid \mathfrak{t} \in \ell(\mathfrak{T}_0) \} }
\]
from Section~\ref{Subsec: Invertible pieces}. Then $U_\fB(M) \cong U(M)$. Let $\MCG{U_\fB(M)}{\fB}$ denote the right monoid Cayley graph of $U(M)$ with respect to the generating set $\fB$, and let $\GCG{U_\fB(M)}{\fB}$ denote the group Cayley graph of $U(M)$ with respect to the generating set $\fB$. These two graphs are very closely related to each other. Indeed, we have the following general and essentially obvious claim.

\begin{prop}\label{Prop: lud MCG is iso to GCG}
Let $G$ be a group generated as a monoid by the set $A$. As labelled graphs, $\textnormal{lud}(\MCG{G}{A}) \cong \GCG{G}{A}$. In particular, $\MCG{G}{A}$ is context-free if and only if $\GCG{G}{A}$ is.
\end{prop}

The goal of this section is to define a new graph $\fU$ from the monoid Cayley graph $\MCG{U_\fB(M)}{\fB}$, and we will do so in a way to ensure the existence of a properly discontinuous and co-compact action of $U(M)$ on $\fU$. By a straightforward application of the Svar\u{c}-Milnor lemma, we will thus be able to conclude that the graph $\fU$ is quasi-isometric to the Cayley graph of $U(M)$ (as undirected graphs). Along the way, we will also capture many of the algebraic properties of $U(M)$ in the algebraic properties of $\fU$. This graph $\fU$ will be called the \textit{Sch\"{u}tzenberger graph of the units of} $M$.

Starting with $\MCG{U_\fB(M)}{\fB}$, we first replace every edge labelled $b_{i,j}$ by a directed path of length $|\lambda_{i,j}|$ with path label $\lambda_{i,j}$, as indicated below.
\vspace{0.5cm}

\begin{center}
\begin{tikzpicture}[>=stealth',thick,scale=0.8,el/.style = {inner sep=2pt, align=left, sloped}]%
                        \node (l0)[label=below:$u$][circle, draw, fill=black!50,
                        inner sep=0pt, minimum width=8pt] at (0,0) {};
                        \node (l1)[label=below:$v$][circle, draw, fill=black!50,
                        inner sep=0pt, minimum width=8pt] at (2,0) {};
\path[->] 
    (l0)  edge node[el,below]{$b_{i,j}$}         (l1);

    					\node (dots) at (1, -2) {$\cdots$};
    					\node (lab) at (1, -3) {$(\lambda_{i,j} \equiv a_1 a_2 \cdots a_n)$};
                        \node (r0)[label=below:$u$][circle, draw, fill=black!50,
                        inner sep=0pt, minimum width=8pt] at (-3,-2) {};
                        \node (r1)[circle, draw, fill=black!50,
                        inner sep=0pt, minimum width=4pt] at (-1,-2) {};
                        \node (r2)[circle, draw, fill=black!50,
                        inner sep=0pt, minimum width=4pt] at (3,-2) {};
                        \node (r3)[label=below:$v$][circle, draw, fill=black!50,
                        inner sep=0pt, minimum width=8pt] at (5,-2) {};
\path[->] 
    (r0)  edge node[el,below]{$a_1$}         (r1)
    (r1)  edge node[el,below]{$a_2$}         (dots)
    (dots)  edge node[el,below]{$a_{n-1}$}         (r2)
    (r2)  edge node[el,below]{$a_n$}         (r3);

\end{tikzpicture}
\end{center}

We denote the resulting graph $\fU_0$, and note that its labelling alphabet is now $A$. Some subset of the vertices of $\fU_0$ will have been present already in $\MCG{U_\fB(M)}{\fB}$ -- we call this subset the \textit{locally invertible} vertices. In the above, the two enlarged vertices are locally invertible vertices. Every vertex of $\fU_0$ is locally invertible if and only if $M$ is a group. For every vertex $v$ of $\fU_0$, there exists a unique locally invertible vertex $\widetilde{v}$ such that $v$ is reachable from $\widetilde{v}$ by a path with path label in $\fPe$. This follows directly from the fact that no non-empty prefix of an element of $\Lambda$ appears as a proper prefix of another element of $\Lambda$. We then write $i(v) := \widetilde{v}$, letting $i : V(\fU_0) \to V(\fU_0)$ denote the function which, on input $v \in V(\fU_0)$, returns the \textit{locally invertible vertex associated to $v$}. Note that the prescribed path from $i(v)$ to $v$ is unique. 

Thus the set of vertices of $\fU_0$ is in bijective correspondence with the set of all triples \[(m, \xi, \lambda) \in U(M) \times \fPe \times \Lambda_\varepsilon\] with the property that $\xi$ is a proper prefix of $\lambda$. The bijection, written out explicitly, sends a vertex $v$ of $\fU_0$ to the triple $(i_m(v), \xi, \lambda)$, where $i_m(v)$ is the vertex of $\MCG{U_\fB(M)}{\fB}$ corresponding to the locally invertible vertex $i(v)$; and $\xi$ is the path label of the path from $i(v)$ to $v$; and $\lambda$ is either $\varepsilon$, if $i(v) = v$, or else is the unique piece which, when subdivided, gave rise to the path of which the path from $i(v)$ to $v$ is an initial segment. We remark that the information contained in $\xi$ and $\lambda$ cannot be condensed into just providing $\xi$, as $\xi$ may be a prefix of many different pieces.

The locally invertible vertices are precisely those of the form $(m, \varepsilon, \varepsilon)$. There is a natural partition induced by an equivalence relation $\sim_i$ on $V(\fU_0)$, which is now identified with the set of all above triples, where 
\[
(m_1, \xi, \lambda_1) \sim_i (m_2, \zeta, \lambda_2) \iff m_1 = m_2.
\]
We extend the resulting equivalence relation to the edges of $\fU_0$ in the obvious way, and denote the extended relation on $\fU_0$ by $\sim_i$. Note that all equivalence classes under $\sim_i$ are of equal and finite cardinality, and carry the structure of the same connected directed finite labelled graph. Indeed, this graph isomorphic to the subgraph of $\fU_0$ induced on the set $\{ m \} \times \fPe \times \Lambda_\varepsilon$ for any $m \in U(M)$, which is evidently finite. Furthermore, $\MCG{U_\fB(M)}{\fB}$ can be naturally identified with $\fU_0 / \sim_i$ with any loops removed.
 
\begin{example}\label{Example: abc=1, ac=1}
Let $M = \pres{Mon}{a,b,c}{abc = 1, ac = 1}$. Then we have that $\Lambda = \{ abc, ac \}, \fB = \{ b_{1,1}, b_{2,1} \}$, and $U(M) \cong 1$. The Cayley graph $\MCG{U_\fB(M)}{\fB}$ is hence just a single vertex with two loops, one labelled $b_{1,1}$, and the other labelled $b_{2,1}$. Now $\fU_0$ has $5$ vertices, with one locally invertible vertex. See Figure~\ref{Fig: sim_T is needed}. The graph $\fU_0 / \sim_i$ is a single vertex, and no edges.

\begin{figure}[h]
\begin{tikzpicture}[>=stealth',thick,scale=0.8,el/.style = {inner sep=2pt, align=left, sloped}]%
						\node (q0) [circle, draw, fill=black!50,
                        inner sep=0pt, minimum width=8pt] at (0,0) {};

                        \node (r0)[label=right:{$(1,\varepsilon, \varepsilon)$}][circle, draw, fill=black!50,
                        inner sep=0pt, minimum width=8pt] at (6,0) {};
                        \node (r1)[label=above:{$(1,a,abc)$}] [circle, draw, fill=black!50,
                        inner sep=0pt, minimum width=4pt] at (4,2) {};
						\node (r2)[label=above:{$(1,ab,abc)$}][circle, draw, fill=black!50,
                        inner sep=0pt, minimum width=4pt] at (8,2) {};                 
                        \node (r3)[label=below:{$(1,a,ac)$}] [circle, draw, fill=black!50,
                        inner sep=0pt, minimum width=4pt] at (6,-2) {};

                        %\node (r1) [circle, draw, fill=black!50,
                        %inner sep=0pt, minimum width=8pt] at (8,0) {};
						%\node (r2) [circle, draw, fill=black!50,
                        %inner sep=0pt, minimum width=4pt] at (7,-1) {};
                        %\node (r3) [circle, draw, fill=black!50,
                        %inner sep=0pt, minimum width=4pt] at (9,-1) {};
                        %\node (r4) [circle, draw, fill=black!50,
                        %inner sep=0pt, minimum width=4pt] at (8, -2) {};   

\path[->] 
    (r0)  edge node[below]{$a$}         (r1)
    (r1)  edge node[above]{$b$}         (r2)
    (r2)  edge node[below]{$c$}         (r0);

\draw[->] (r0) to [bend right] node [left] {$a$} (r3);
\draw[->] (r3) to [bend right] node [right] {$c$} (r0);
\path[->] (q0) edge [out=220,in=320,looseness=20] node[below] {$b_{2,1}$} (q0);
\path[->] (q0) edge [out=140,in=40,looseness=20] node[above] {$b_{1,1}$} (q0);

\end{tikzpicture}
\caption{Left: The Cayley graph $\MCG{U_\fB(M)}{\fB}$ of Example~\ref{Example: abc=1, ac=1}. Right: $\fU_0$ of the same example.}
\label{Fig: sim_T is needed}
\end{figure}
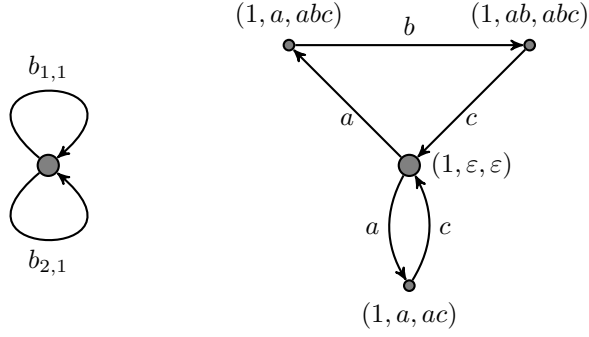
\end{example}  
  
We will now define a second equivalence relation $\sim_M$ on $\fU_0$, and show several important properties of it. This equivalence relation will capture graphically the congruence $\lra{M}$ on $A^\ast$. We define $\sim_M$ on the vertices of $\fU_0$ by
\[
(m_1, \xi, \xi^\prime) \sim_M (m_2, \zeta, \zeta^\prime) \iff m_1 \cdot \pi(\xi) = m_2 \cdot \pi(\zeta)
\] 
In Figure~\ref{Fig: sim_T is needed}, we have that the vertices labelled by $(1,a,abc)$ and $(1,a,ac)$ satisfy $(1,a,abc) \sim_M (1,a,ac)$. Now, $\sim_M$-equivalence is highly controlled, and takes place entirely within each $\sim_i$-equivalence class, as the below proposition will show. 

In order to apply rewriting techniques, we will introduce a piece of notation for the vertices of $\fU_0$. If $u = (m_1, \xi, \lambda) \in V(\fU_0)$, then let $u_0 \in \Lambda^\ast$ be any word over the pieces such that $\pi(u_0) = m_1$. We then say that a \textit{literal form} of $u$ is the triple $(u_0, \xi, \lambda)$. Note that any literal form of a vertex is an element of $A^\ast \times A^\ast \times A^\ast$, whereas the vertices of $\fU_0$ are elements of $M \times A^\ast \times A^\ast$. If $u, v \in V(\fU_0)$ have literal forms $(u_0, \xi, \xi^\prime)$ and $(v_0, \zeta, \zeta^\prime)$, respectively, then we have that 
\[
u \sim_M v \quad \iff \quad u_0 \xi \lra{M} v_0 \zeta.
\]
Since the right-hand side is independent of choice of literal form for $u$ and $v$, we are free to fix an arbitrary literal form for either and speak of \textit{the literal form} of a vertex in $\fU_0$, where the definite article is always only with respect to $\sim_M$. Hence, once vertices are given to us in a literal form, we will almost exclusively use the above biconditional to verify that $u \sim_M v$. The following proposition shows that to consider $\sim_M$-equivalence of two triples is to consider equality in $U(M)$ and equality of the second element of the triples in $M$. In particular, the third element of the triple becomes entirely redundant in the context of $\sim_M$. 

\begin{prop}\label{Prop: sim_M => sim_i}
Let $u, v \in V(\fU_0)$ with $u = (m_1, \xi, \lambda_1)$ and $v =(m_2, \zeta, \lambda_2)$. Then 
\[
u \sim_M v \iff m_1 = m_2 \quad \textnormal{and} \quad \xi \lra{M} \zeta
\]
In particular, $(u \sim_M v) \implies (u \sim_i v)$. 
\end{prop}
\begin{proof}
Let $u, v$ have literal forms $(u_0, \xi, \lambda_1)$ and $(v_0, \zeta, \lambda_2)$, respectively. We will 

$(\impliedby)$ If $m_1 = m_2$, then $u_0 \lra{M} v_0$ and consequently if also $\xi \lra{M} \zeta$, then we must have $u_0 \xi \lra{M} v_0 \zeta$. Hence $u \sim_M v$. 

$(\implies)$ We make use of the Normal Form Lemma, Lemma~\ref{Lem: Zhang's Lemma}. Assume for the contrapositive that $u_0 \not\lra{M} v_0$ or $\xi \not\lra{M} \zeta$, and assume for contradiction that $u_0 \xi \lra{M} v_0 \zeta$. Now no maximal invertible factor of $u_0 \xi$ can overlap with $\xi$, as otherwise some prefix of $\xi$ would be left invertible and hence also invertible, contradicting the maximality of elements of $\Lambda$. Hence any decomposition of $u_0 \xi$ into maximal invertible factors must contain $u_0$ as a maximal invertible factor. Analogously, any decomposition of $v_0 \zeta$ must contain $v_0$ as a factor. By Lemma~\ref{Lem: Zhang's Lemma}, we may factor $u_0 \xi$ and $v_0 \zeta$ as
\begin{align*}
u_0 \xi &\equiv u_1 w_1 \cdots \\
v_0 \zeta &\equiv v_1 w_1 \cdots
\end{align*}
where $u_1$ and $v_1$ are maximal invertible factors of $u_0 \xi$ and $v_0 \zeta$, respectively, with the property that $u_1 \lra{M} v_1$. But since $u_0$ is, by the above, the leftmost maximal invertible factor of $u_0 \xi$ we must have $u_1 \equiv u_0$ and $v_1 \equiv v_0$. This contradicts the assumption that $u_0 \not\lra{M} v_0$ or $\xi \not\lra{M} \zeta$. Consequently we cannot have that $u_0 \xi \lra{M} v_0 \zeta$, and so $u \not\sim_{T} v$, which is what was to be shown.
\end{proof}

We extend $\sim_M$ to the edges of $\fU_0$ in the obvious way, by identifying two edges if and only if their labels are graphically equal, their origins are identified under $\sim_M$, and their terminuses are identified under $\sim_M$. We set $\fU' := \fU_0 / \sim_M$. The vertices of $\fU'$ are thus equivalence classes under $\sim_M$ of pairs $(m, \xi)$ for $m \in U(M)$ and $\xi \in \fPe$, where $(m_1, \xi) \sim_M (m_2, \zeta)$ if and only if $m_1 \pi(\xi) = m_2 \pi(\zeta)$, which by Proposition~\ref{Prop: sim_M => sim_i} is equivalent to $m_1 = m_2$ and $\pi(\xi) = \pi(\zeta)$. We will denote vertices of $\fU_0$ as $[(m, \xi)]$ rather than $[(m, \xi)]_{M}$ for notational brevity. 

Again, as in the case of $\fU_0$, to enable the use of string-rewriting techniques, we will define the \textit{literal form} of a vertex $[(m, \xi)]$ as $[(u_0 ; \xi)]$, where $u_0 \in \Lambda^\ast$ is such that $\pi(u_0) = m$. The use of the definite article is justified, as it was for $\fU_0$, in that we are only interested in vertices up to $\sim_T$-equivalence. The use of a semicolon in writing the pair $[(u_0 ; \xi)]$ is to emphasise that it is distinct from $[(m, \xi)]$, but has no further significance; note that the former is the equivalence class of an element of $A^\ast \times A^\ast$, whereas the latter is the equivalence class of an element of $M \times A^\ast$.

\begin{figure}
\begin{tikzpicture}[x=0.75pt,y=0.75pt,yscale=-2,xscale=2]
%uncomment if require: \path (0,123); %set diagram left start at 0, and has height of 123

%Straight Lines [id:da8919340422281667] 
\draw    (51,56.06) -- (91,56.06) ;
\draw [shift={(71,56.06)}, rotate = 180] [color={rgb, 255:red, 0; green, 0; blue, 0 }  ][line width=0.75]    (4.37,-1.32) .. controls (2.78,-0.56) and (1.32,-0.12) .. (0,0) .. controls (1.32,0.12) and (2.78,0.56) .. (4.37,1.32)   ;
%Straight Lines [id:da26313209089074774] 
\draw    (51,96.06) -- (91,96.06) ;
\draw [shift={(71,96.06)}, rotate = 180] [color={rgb, 255:red, 0; green, 0; blue, 0 }  ][line width=0.75]    (4.37,-1.32) .. controls (2.78,-0.56) and (1.32,-0.12) .. (0,0) .. controls (1.32,0.12) and (2.78,0.56) .. (4.37,1.32)   ;
%Straight Lines [id:da7178182680144565] 
\draw    (91,96.06) -- (51,74) ;
\draw [shift={(71,85.03)}, rotate = 388.88] [color={rgb, 255:red, 0; green, 0; blue, 0 }  ][line width=0.75]    (4.37,-1.32) .. controls (2.78,-0.56) and (1.32,-0.12) .. (0,0) .. controls (1.32,0.12) and (2.78,0.56) .. (4.37,1.32)   ;
%Straight Lines [id:da16454552044469084] 
\draw    (91,56.06) -- (131,56.06) ;
\draw [shift={(111,56.06)}, rotate = 180] [color={rgb, 255:red, 0; green, 0; blue, 0 }  ][line width=0.75]    (4.37,-1.32) .. controls (2.78,-0.56) and (1.32,-0.12) .. (0,0) .. controls (1.32,0.12) and (2.78,0.56) .. (4.37,1.32)   ;
%Straight Lines [id:da40100812700517374] 
\draw    (131,56.06) -- (171,56.06) ;
\draw [shift={(151,56.06)}, rotate = 180] [color={rgb, 255:red, 0; green, 0; blue, 0 }  ][line width=0.75]    (4.37,-1.32) .. controls (2.78,-0.56) and (1.32,-0.12) .. (0,0) .. controls (1.32,0.12) and (2.78,0.56) .. (4.37,1.32)   ;
%Straight Lines [id:da3063873065849523] 
\draw    (91,96.06) -- (131,96.06) ;
\draw [shift={(111,96.06)}, rotate = 180] [color={rgb, 255:red, 0; green, 0; blue, 0 }  ][line width=0.75]    (4.37,-1.32) .. controls (2.78,-0.56) and (1.32,-0.12) .. (0,0) .. controls (1.32,0.12) and (2.78,0.56) .. (4.37,1.32)   ;
%Straight Lines [id:da6264604966139768] 
\draw    (131,96.06) -- (171,96.06) ;
\draw [shift={(151,96.06)}, rotate = 180] [color={rgb, 255:red, 0; green, 0; blue, 0 }  ][line width=0.75]    (4.37,-1.32) .. controls (2.78,-0.56) and (1.32,-0.12) .. (0,0) .. controls (1.32,0.12) and (2.78,0.56) .. (4.37,1.32)   ;
%Straight Lines [id:da3551594861112142] 
\draw    (91,56.06) -- (91,96.06) ;
\draw [shift={(91,76.06)}, rotate = 270] [color={rgb, 255:red, 0; green, 0; blue, 0 }  ][line width=0.75]    (4.37,-1.32) .. controls (2.78,-0.56) and (1.32,-0.12) .. (0,0) .. controls (1.32,0.12) and (2.78,0.56) .. (4.37,1.32)   ;
%Straight Lines [id:da7227843123184079] 
\draw    (171,96.06) -- (91,56.06) ;
\draw [shift={(131,76.06)}, rotate = 386.57] [color={rgb, 255:red, 0; green, 0; blue, 0 }  ][line width=0.75]    (4.37,-1.32) .. controls (2.78,-0.56) and (1.32,-0.12) .. (0,0) .. controls (1.32,0.12) and (2.78,0.56) .. (4.37,1.32)   ;
%Shape: Circle [id:dp9037536404790119] 
\draw  [fill={rgb, 255:red, 155; green, 155; blue, 155 }  ,fill opacity=1 ][line width=0.75]  (88.94,56.06) .. controls (88.94,54.92) and (89.86,54) .. (91,54) .. controls (92.14,54) and (93.06,54.92) .. (93.06,56.06) .. controls (93.06,57.2) and (92.14,58.12) .. (91,58.12) .. controls (89.86,58.12) and (88.94,57.2) .. (88.94,56.06) -- cycle ;
%Shape: Circle [id:dp551782348698383] 
\draw  [fill={rgb, 255:red, 155; green, 155; blue, 155 }  ,fill opacity=1 ][line width=0.75]  (88.94,96.06) .. controls (88.94,94.92) and (89.86,94) .. (91,94) .. controls (92.14,94) and (93.06,94.92) .. (93.06,96.06) .. controls (93.06,97.2) and (92.14,98.12) .. (91,98.12) .. controls (89.86,98.12) and (88.94,97.2) .. (88.94,96.06) -- cycle ;
%Straight Lines [id:da8005469622683974] 
\draw    (171,56.06) -- (211,56.06) ;
\draw [shift={(191,56.06)}, rotate = 180] [color={rgb, 255:red, 0; green, 0; blue, 0 }  ][line width=0.75]    (4.37,-1.32) .. controls (2.78,-0.56) and (1.32,-0.12) .. (0,0) .. controls (1.32,0.12) and (2.78,0.56) .. (4.37,1.32)   ;
%Straight Lines [id:da9041376890464738] 
\draw    (211,56.06) -- (251,56.06) ;
\draw [shift={(231,56.06)}, rotate = 180] [color={rgb, 255:red, 0; green, 0; blue, 0 }  ][line width=0.75]    (4.37,-1.32) .. controls (2.78,-0.56) and (1.32,-0.12) .. (0,0) .. controls (1.32,0.12) and (2.78,0.56) .. (4.37,1.32)   ;
%Straight Lines [id:da8764924760415511] 
\draw    (171,96.06) -- (211,96.06) ;
\draw [shift={(191,96.06)}, rotate = 180] [color={rgb, 255:red, 0; green, 0; blue, 0 }  ][line width=0.75]    (4.37,-1.32) .. controls (2.78,-0.56) and (1.32,-0.12) .. (0,0) .. controls (1.32,0.12) and (2.78,0.56) .. (4.37,1.32)   ;
%Straight Lines [id:da8285732059903437] 
\draw    (211,96.06) -- (251,96.06) ;
\draw [shift={(231,96.06)}, rotate = 180] [color={rgb, 255:red, 0; green, 0; blue, 0 }  ][line width=0.75]    (4.37,-1.32) .. controls (2.78,-0.56) and (1.32,-0.12) .. (0,0) .. controls (1.32,0.12) and (2.78,0.56) .. (4.37,1.32)   ;
%Straight Lines [id:da8173029298310008] 
\draw    (171,56.06) -- (171,96.06) ;
\draw [shift={(171,76.06)}, rotate = 270] [color={rgb, 255:red, 0; green, 0; blue, 0 }  ][line width=0.75]    (4.37,-1.32) .. controls (2.78,-0.56) and (1.32,-0.12) .. (0,0) .. controls (1.32,0.12) and (2.78,0.56) .. (4.37,1.32)   ;
%Straight Lines [id:da6281776008975264] 
\draw    (251,96.06) -- (171,56.06) ;
\draw [shift={(211,76.06)}, rotate = 386.57] [color={rgb, 255:red, 0; green, 0; blue, 0 }  ][line width=0.75]    (4.37,-1.32) .. controls (2.78,-0.56) and (1.32,-0.12) .. (0,0) .. controls (1.32,0.12) and (2.78,0.56) .. (4.37,1.32)   ;
%Shape: Circle [id:dp08402233031069217] 
\draw  [fill={rgb, 255:red, 155; green, 155; blue, 155 }  ,fill opacity=1 ][line width=0.75]  (168.94,56.06) .. controls (168.94,54.92) and (169.86,54) .. (171,54) .. controls (172.14,54) and (173.06,54.92) .. (173.06,56.06) .. controls (173.06,57.2) and (172.14,58.12) .. (171,58.12) .. controls (169.86,58.12) and (168.94,57.2) .. (168.94,56.06) -- cycle ;
%Shape: Circle [id:dp5555880346384849] 
\draw  [fill={rgb, 255:red, 155; green, 155; blue, 155 }  ,fill opacity=1 ][line width=0.75]  (168.94,96.06) .. controls (168.94,94.92) and (169.86,94) .. (171,94) .. controls (172.14,94) and (173.06,94.92) .. (173.06,96.06) .. controls (173.06,97.2) and (172.14,98.12) .. (171,98.12) .. controls (169.86,98.12) and (168.94,97.2) .. (168.94,96.06) -- cycle ;
%Straight Lines [id:da718881623676829] 
\draw    (251,56.06) -- (291,56.06) ;
\draw [shift={(271,56.06)}, rotate = 180] [color={rgb, 255:red, 0; green, 0; blue, 0 }  ][line width=0.75]    (4.37,-1.32) .. controls (2.78,-0.56) and (1.32,-0.12) .. (0,0) .. controls (1.32,0.12) and (2.78,0.56) .. (4.37,1.32)   ;
%Straight Lines [id:da32468621673818476] 
\draw    (251,96.06) -- (291,96.06) ;
\draw [shift={(271,96.06)}, rotate = 180] [color={rgb, 255:red, 0; green, 0; blue, 0 }  ][line width=0.75]    (4.37,-1.32) .. controls (2.78,-0.56) and (1.32,-0.12) .. (0,0) .. controls (1.32,0.12) and (2.78,0.56) .. (4.37,1.32)   ;
%Straight Lines [id:da4505338185705918] 
\draw    (251,56.06) -- (251,96.06) ;
\draw [shift={(251,76.06)}, rotate = 270] [color={rgb, 255:red, 0; green, 0; blue, 0 }  ][line width=0.75]    (4.37,-1.32) .. controls (2.78,-0.56) and (1.32,-0.12) .. (0,0) .. controls (1.32,0.12) and (2.78,0.56) .. (4.37,1.32)   ;
%Straight Lines [id:da5523365350335367] 
\draw    (291,74) -- (251,56.06) ;
\draw [shift={(271,65.03)}, rotate = 384.15] [color={rgb, 255:red, 0; green, 0; blue, 0 }  ][line width=0.75]    (4.37,-1.32) .. controls (2.78,-0.56) and (1.32,-0.12) .. (0,0) .. controls (1.32,0.12) and (2.78,0.56) .. (4.37,1.32)   ;
%Shape: Circle [id:dp8326413832041979] 
\draw  [fill={rgb, 255:red, 155; green, 155; blue, 155 }  ,fill opacity=1 ][line width=0.75]  (248.94,56.06) .. controls (248.94,54.92) and (249.86,54) .. (251,54) .. controls (252.14,54) and (253.06,54.92) .. (253.06,56.06) .. controls (253.06,57.2) and (252.14,58.12) .. (251,58.12) .. controls (249.86,58.12) and (248.94,57.2) .. (248.94,56.06) -- cycle ;
%Shape: Circle [id:dp8541138781675393] 
\draw  [fill={rgb, 255:red, 155; green, 155; blue, 155 }  ,fill opacity=1 ][line width=0.75]  (248.94,96.06) .. controls (248.94,94.92) and (249.86,94) .. (251,94) .. controls (252.14,94) and (253.06,94.92) .. (253.06,96.06) .. controls (253.06,97.2) and (252.14,98.12) .. (251,98.12) .. controls (249.86,98.12) and (248.94,97.2) .. (248.94,96.06) -- cycle ;
%Straight Lines [id:da14463474346489802] 
\draw    (131,54.14) -- (131,57.88) ;
%Straight Lines [id:da3766016512117125] 
\draw    (131,94.14) -- (131,97.88) ;
%Straight Lines [id:da7751312865895894] 
\draw    (211,54.14) -- (211,57.88) ;
%Straight Lines [id:da945673862501363] 
\draw    (211.29,94.43) -- (211.29,98.16) ;
%Straight Lines [id:da3989065077149194] 
\draw    (291,53.86) -- (291,57.59) ;
%Straight Lines [id:da2293439725164126] 
\draw    (291.29,94.14) -- (291.29,97.88) ;
%Straight Lines [id:da335146073901297] 
\draw    (51,54.43) -- (51,58.16) ;
%Straight Lines [id:da744015314175309] 
\draw    (51,94.43) -- (51,98.16) ;

% Text Node
\draw (168.33,46.64) node [anchor=north west][inner sep=0.75pt]  [font=\tiny] [align=left] {1};

\end{tikzpicture}
\caption{A portion of the graph $\fU'$ for the monoid $M = \pres{Mon}{a,b,c}{a(bc)a = 1}$. Note that $U(M) \cong \mathbb{Z}$. In the graph, horizontal movement corresponds to reading $bc$, whereas vertical and diagonal movement corresponds to reading $a$. The gray and enlarged vertices are the locally invertible vertices, and the small vertical dashes are  the non-locally invertible vertices. }
\end{figure}
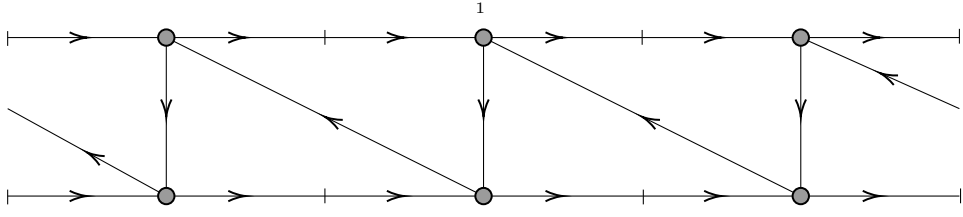

\subsection{Missing Edges}

Under stronger assumptions on the presentations involved, the graph $\fU'$ would already correctly describe an induced subgraph of the Cayley graph of $M$. This happens, for example, if the set of prefixes $\fP$ forms a code as a subset of $A^\ast$. In general, however, certain edges will be missing; this is a consequence of the fact that whereas $\Lambda \cdot \fP \cap \fP = \varnothing$ by a simple overlap argument, it might happen that $\fP \cdot \Lambda \cap \fP \neq \varnothing$. The goal of this subsection is then to describe what edges are missing from $\fU'$, and that these are evenly distributed across all $\sim_i$-classes.

We begin with a proposition which follows from an application of the Normal Form Lemma.

\begin{prop}\label{Prop: Missing edges}
Let $u_0, v_0 \in A^\ast$ be invertible, and $\xi, \zeta$ be proper prefixes of pieces. Let $a \in A$. Then $(u_0 \xi) \cdot a \lra{M} (v_0 \zeta)$ if and only if exactly one of the following occurs:
\begin{enumerate}
\item $\zeta \equiv \varepsilon$, there is some piece $\lambda \in \Lambda$ such that $\xi \cdot a \lra{M} \lambda$, and $u_0 \cdot \lambda \lra{M} v_0$.
\item $\xi \cdot a \lra{M} \zeta$ and $u_0 \lra{M} v_0$. 
\end{enumerate}
\end{prop}
\begin{proof}
$(\impliedby)$ In either case (1) or (2), it is clear that $(u_0 \xi) \cdot a \lra{M} (v_0 \zeta)$. 

$(\implies)$ We know that we have the equality
\[
u_0 \xi a \lra{M} v_0 \zeta
\]
and, seeking to use Lemma~\ref{Lem: Zhang's Lemma}, we will first detect the maximal invertible factors. The left-most maximal invertible factor in the right-hand side is $v_0$, as if this invertible factor were to extend into $\zeta$, then some prefix of $\zeta$ will be a suffix of an invertible factor. If this prefix were empty, then $\zeta \equiv \varepsilon$, and we have that both sides are invertible; thus we are in case (1) above after observing that $\xi \cdot a$ cannot possibly be congruent to a product of more than one piece. Now, if this prefix were non-empty, then by a standard overlap argument we have a contradiction. Thus, in this case, the left-most maximal invertible factor of the right-hand side must be $v_0$. As a consequence, $\zeta \not\equiv \varepsilon$. Hence $v_0 \zeta$ is not invertible, so neither is $u_0 \xi a$. Applying the same argument to $u_0 \xi a$, we see that $u_0$ must thus be the left-most maximal invertible factor of $u_0 \xi a$. By Lemma~\ref{Lem: Zhang's Lemma}, we hence have $u_0 \lra{M} v_0$ and $\xi \cdot a \lra{M} \zeta$, and we are in case (2).
\end{proof}

With this proposition in hand, we are prepared to add in the missing edges, and to define the graph $\fU$, our final object of study. 

\begin{definition}\label{Def: Adding missing edges}
We define $\fU = (V(\fU), E(\fU))$ to be the graph obtained from $\fU'$ in the following way: if there exist prefixes $\xi, \zeta \in \fPe$ and $a \in A$ such that $\xi \cdot a \lra{M} \zeta$, then for all $m \in U(M)$ we add an edge $[(m, \xi)] \xrightarrow{a} [(m , \zeta)]$ if one does not already exist. We call $\fU$ the \textit{Sch\"utzenberger graph of the units of} $M$.
\end{definition}

Note that if such an edge as in the above definition does not exist for \textit{some} $m \in U(M)$, then it does not exist for \textit{any} $m \in U(M)$, by case (2) of Proposition~\ref{Prop: Missing edges}, and is hence added in all these cases. This ensures that the edges added by Definition~\ref{Def: Adding missing edges} are added uniformly across all $\sim_i$-classes. We illustrate this in Figure~\ref{Fig: Missing edges example}. Furthermore, the vertex set of $\fU$ is the same as that of $\fU'$.

For the following proposition, we recall the definition of $\fR_1$ as the connected component of the identity element in $\MCG{M}{A}$, and that a graph homomorphism $\phi \colon G \to H$ is \textit{faithful} if $\phi(G)$ is an induced subgraph of $H$. Note that a faithful injective graph homomorphism is equivalent to a \textit{full} injective graph homomorphism, i.e. an injective graph homomorphism in which $\phi(u)$ is adjacent to $\phi(v)$ if and only if $u$ is adjacent to $v$.

\begin{prop}\label{Prop: fU is an induced subgraph of fR1}
The map
\begin{align*}
\phi : V(\fU) &\to \{ m \cdot \pi(\xi) \: | \: m \in U(M),  \xi \in \fPe  \} \subseteq V(\fR_1) \\
[(m, \xi)] &\mapsto m \cdot \pi(\xi)
\end{align*}
extends to a faithful injective labelled graph homomorphism $\phi : \fU \hookrightarrow \fR_1$. In other words, $\fU$ is isomorphic to the subgraph of $\fR_1$ induced on the set of vertices of the form $m \cdot \pi(\xi)$, where $m \in U(M)$ and $\xi \in \fP$. 
\end{prop}
\begin{proof}
We set $V_0 :=\{ m \cdot \pi(\xi) \: | \: m \in U(M),  \xi \in \fPe  \}$. It is clear that $\phi$ is injective on the set of vertices, by definition of $\fU'$ as a quotient of $\fU_0$ modulo $\sim_M$-equivalence. We define $\phi$ on $E(\fU)$ in the obvious way. 

We first show $\phi$ is injective on the set of edges. Let $u = [(m_1, \xi)]$ and $v = [(m_2, \zeta)]$ be arbitrary vertices of $\fU$. If $e : u \xrightarrow{a} v$ is an edge of $\fU$ then we will show that there is an edge in $\fR_1$ from $\phi(u)$ to $\phi(v)$ labelled by $a$, i.e.\ that $\phi(u) \cdot \pi(a) = \phi(v)$. Hence assume $e : u \xrightarrow{a} v$ is an edge of $\fU$. Let $u$ and $v$ have literal forms $[(u_0 ; \xi)]$ and $[(v_0 ; \zeta)]$, respectively. Then by definition $\phi(u) = \pi(u_0 \xi)$ and $\phi(v) = \pi(v_0 \zeta)$. Since $e$ is an edge of $\fU$ we must have that either $\xi$ and $\zeta$ are prefixes of some piece such that $\xi a \equiv \zeta$, in which case $u_0 \lra{M} v_0$, or $a$ is the final letter of a piece $\xi a$, in which case $v_0 \lra{M} u_0 \xi \cdot a$. In either case, we have 
\[
u_0 \xi \cdot a \lra{M} v_0 \zeta
\]
whence $\phi(u) \cdot \pi(a) = \phi(v)$, and $\phi$ is an injective graph homomorphism. 

To see that it is faithful, it suffices to note that if $\pi(u_0 \xi) \xrightarrow{a} \pi(v_0 \zeta)$ is an edge of $\fR_1$, then we are in exactly one of the cases (1) or (2) of Proposition~\ref{Prop: Missing edges}. In case (1), then the edge $[(m_1, \xi)] \xrightarrow{a} [(m_2, \zeta)]$ is clearly present in $\fU'$ by construction of $\fU'$ from $\fU_0$, and hence is also an edge of $\fU$. In case (2), we have that $[(m_1, \xi)] \xrightarrow{a} [(m_2, \zeta)]$ is an edge added to $\fU$ by Definition~\ref{Def: Adding missing edges}. Thus, in both cases $\phi$ is faithful, and hence $\phi$ has all the desired properties.
\end{proof}

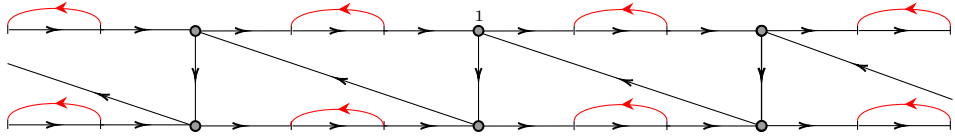
\begin{figure}[h]

\begin{tikzpicture}[x=0.75pt,y=0.75pt,yscale=-1.2,xscale=1.2]
%uncomment if require: \path (0,123); %set diagram left start at 0, and has height of 123

%Straight Lines [id:da8764924760415511] 
\draw    (171,96) -- (211,96) ;
\draw [shift={(191,96)}, rotate = 180] [color={rgb, 255:red, 0; green, 0; blue, 0 }  ][line width=0.75]    (4.37,-1.32) .. controls (2.78,-0.56) and (1.32,-0.12) .. (0,0) .. controls (1.32,0.12) and (2.78,0.56) .. (4.37,1.32)   ;
%Straight Lines [id:da8285732059903437] 
\draw    (250,96) -- (290,96) ;
\draw [shift={(270,96)}, rotate = 180] [color={rgb, 255:red, 0; green, 0; blue, 0 }  ][line width=0.75]    (4.37,-1.32) .. controls (2.78,-0.56) and (1.32,-0.12) .. (0,0) .. controls (1.32,0.12) and (2.78,0.56) .. (4.37,1.32)   ;
%Straight Lines [id:da8173029298310008] 
\draw    (171,56.06) -- (171,96.06) ;
\draw [shift={(171,76.06)}, rotate = 270] [color={rgb, 255:red, 0; green, 0; blue, 0 }  ][line width=0.75]    (4.37,-1.32) .. controls (2.78,-0.56) and (1.32,-0.12) .. (0,0) .. controls (1.32,0.12) and (2.78,0.56) .. (4.37,1.32)   ;
%Straight Lines [id:da6281776008975264] 
\draw    (290,96) -- (171,56) ;
\draw [shift={(230.5,76)}, rotate = 378.58000000000004] [color={rgb, 255:red, 0; green, 0; blue, 0 }  ][line width=0.75]    (4.37,-1.32) .. controls (2.78,-0.56) and (1.32,-0.12) .. (0,0) .. controls (1.32,0.12) and (2.78,0.56) .. (4.37,1.32)   ;
%Shape: Circle [id:dp08402233031069217] 
\draw  [fill={rgb, 255:red, 155; green, 155; blue, 155 }  ,fill opacity=1 ][line width=0.75]  (168.94,56.06) .. controls (168.94,54.92) and (169.86,54) .. (171,54) .. controls (172.14,54) and (173.06,54.92) .. (173.06,56.06) .. controls (173.06,57.2) and (172.14,58.12) .. (171,58.12) .. controls (169.86,58.12) and (168.94,57.2) .. (168.94,56.06) -- cycle ;
%Shape: Circle [id:dp5555880346384849] 
\draw  [fill={rgb, 255:red, 155; green, 155; blue, 155 }  ,fill opacity=1 ][line width=0.75]  (168.94,96.06) .. controls (168.94,94.92) and (169.86,94) .. (171,94) .. controls (172.14,94) and (173.06,94.92) .. (173.06,96.06) .. controls (173.06,97.2) and (172.14,98.12) .. (171,98.12) .. controls (169.86,98.12) and (168.94,97.2) .. (168.94,96.06) -- cycle ;
%Straight Lines [id:da4505338185705918] 
\draw    (290,56.06) -- (290,96.06) ;
\draw [shift={(290,76.06)}, rotate = 270] [color={rgb, 255:red, 0; green, 0; blue, 0 }  ][line width=0.75]    (4.37,-1.32) .. controls (2.78,-0.56) and (1.32,-0.12) .. (0,0) .. controls (1.32,0.12) and (2.78,0.56) .. (4.37,1.32)   ;
%Shape: Circle [id:dp8326413832041979] 
\draw  [fill={rgb, 255:red, 155; green, 155; blue, 155 }  ,fill opacity=1 ][line width=0.75]  (287.94,56) .. controls (287.94,54.86) and (288.86,53.94) .. (290,53.94) .. controls (291.14,53.94) and (292.06,54.86) .. (292.06,56) .. controls (292.06,57.14) and (291.14,58.06) .. (290,58.06) .. controls (288.86,58.06) and (287.94,57.14) .. (287.94,56) -- cycle ;
%Shape: Circle [id:dp8541138781675393] 
\draw  [fill={rgb, 255:red, 155; green, 155; blue, 155 }  ,fill opacity=1 ][line width=0.75]  (287.94,96.06) .. controls (287.94,94.92) and (288.86,94) .. (290,94) .. controls (291.14,94) and (292.06,94.92) .. (292.06,96.06) .. controls (292.06,97.2) and (291.14,98.12) .. (290,98.12) .. controls (288.86,98.12) and (287.94,97.2) .. (287.94,96.06) -- cycle ;
%Straight Lines [id:da945673862501363] 
\draw    (211.29,94) -- (211.29,97.73) ;
%Straight Lines [id:da49743039997295546] 
\draw    (211.8,96) -- (251.8,96) ;
\draw [shift={(231.8,96)}, rotate = 180] [color={rgb, 255:red, 0; green, 0; blue, 0 }  ][line width=0.75]    (4.37,-1.32) .. controls (2.78,-0.56) and (1.32,-0.12) .. (0,0) .. controls (1.32,0.12) and (2.78,0.56) .. (4.37,1.32)   ;
%Straight Lines [id:da20352595887176106] 
\draw    (250.29,94) -- (250.29,97.73) ;
%Straight Lines [id:da028766736379375546] 
\draw    (173.06,56.06) -- (211,56) ;
\draw [shift={(192.03,56.03)}, rotate = 539.9100000000001] [color={rgb, 255:red, 0; green, 0; blue, 0 }  ][line width=0.75]    (4.37,-1.32) .. controls (2.78,-0.56) and (1.32,-0.12) .. (0,0) .. controls (1.32,0.12) and (2.78,0.56) .. (4.37,1.32)   ;
%Straight Lines [id:da4877811910437908] 
\draw    (250,56) -- (287.94,56) ;
\draw [shift={(268.97,56)}, rotate = 180] [color={rgb, 255:red, 0; green, 0; blue, 0 }  ][line width=0.75]    (4.37,-1.32) .. controls (2.78,-0.56) and (1.32,-0.12) .. (0,0) .. controls (1.32,0.12) and (2.78,0.56) .. (4.37,1.32)   ;
%Straight Lines [id:da08250712885642564] 
\draw    (211.29,54) -- (211.29,57.73) ;
%Straight Lines [id:da8286513768155104] 
\draw    (211.8,56) -- (251.8,56) ;
\draw [shift={(231.8,56)}, rotate = 180] [color={rgb, 255:red, 0; green, 0; blue, 0 }  ][line width=0.75]    (4.37,-1.32) .. controls (2.78,-0.56) and (1.32,-0.12) .. (0,0) .. controls (1.32,0.12) and (2.78,0.56) .. (4.37,1.32)   ;
%Straight Lines [id:da2788009077236169] 
\draw    (250.29,54) -- (250.29,57.73) ;
%Straight Lines [id:da21174774171215405] 
\draw    (292.06,96.06) -- (329.89,96.02) ;
\draw [shift={(310.98,96.04)}, rotate = 539.94] [color={rgb, 255:red, 0; green, 0; blue, 0 }  ][line width=0.75]    (4.37,-1.32) .. controls (2.78,-0.56) and (1.32,-0.12) .. (0,0) .. controls (1.32,0.12) and (2.78,0.56) .. (4.37,1.32)   ;
%Straight Lines [id:da9345172949792537] 
\draw    (368.89,96.02) -- (408.89,96.02) ;
\draw [shift={(388.89,96.02)}, rotate = 180] [color={rgb, 255:red, 0; green, 0; blue, 0 }  ][line width=0.75]    (4.37,-1.32) .. controls (2.78,-0.56) and (1.32,-0.12) .. (0,0) .. controls (1.32,0.12) and (2.78,0.56) .. (4.37,1.32)   ;
%Straight Lines [id:da048705125306603536] 
\draw    (408.89,96.02) -- (291.89,57.23) ;
\draw [shift={(350.39,76.63)}, rotate = 378.34000000000003] [color={rgb, 255:red, 0; green, 0; blue, 0 }  ][line width=0.75]    (4.37,-1.32) .. controls (2.78,-0.56) and (1.32,-0.12) .. (0,0) .. controls (1.32,0.12) and (2.78,0.56) .. (4.37,1.32)   ;
%Straight Lines [id:da0999558737645947] 
\draw    (408.89,56.08) -- (408.89,96.08) ;
\draw [shift={(408.89,76.08)}, rotate = 270] [color={rgb, 255:red, 0; green, 0; blue, 0 }  ][line width=0.75]    (4.37,-1.32) .. controls (2.78,-0.56) and (1.32,-0.12) .. (0,0) .. controls (1.32,0.12) and (2.78,0.56) .. (4.37,1.32)   ;
%Shape: Circle [id:dp13791662708949715] 
\draw  [fill={rgb, 255:red, 155; green, 155; blue, 155 }  ,fill opacity=1 ][line width=0.75]  (406.83,56.02) .. controls (406.83,54.88) and (407.75,53.96) .. (408.89,53.96) .. controls (410.03,53.96) and (410.95,54.88) .. (410.95,56.02) .. controls (410.95,57.16) and (410.03,58.08) .. (408.89,58.08) .. controls (407.75,58.08) and (406.83,57.16) .. (406.83,56.02) -- cycle ;
%Shape: Circle [id:dp8714352724169381] 
\draw  [fill={rgb, 255:red, 155; green, 155; blue, 155 }  ,fill opacity=1 ][line width=0.75]  (406.83,96.08) .. controls (406.83,94.95) and (407.75,94.02) .. (408.89,94.02) .. controls (410.03,94.02) and (410.95,94.95) .. (410.95,96.08) .. controls (410.95,97.22) and (410.03,98.15) .. (408.89,98.15) .. controls (407.75,98.15) and (406.83,97.22) .. (406.83,96.08) -- cycle ;
%Straight Lines [id:da7260843840053064] 
\draw    (330.17,94.02) -- (330.17,97.76) ;
%Straight Lines [id:da15573242360719464] 
\draw    (330.69,96.02) -- (370.69,96.02) ;
\draw [shift={(350.69,96.02)}, rotate = 180] [color={rgb, 255:red, 0; green, 0; blue, 0 }  ][line width=0.75]    (4.37,-1.32) .. controls (2.78,-0.56) and (1.32,-0.12) .. (0,0) .. controls (1.32,0.12) and (2.78,0.56) .. (4.37,1.32)   ;
%Straight Lines [id:da8086269280117742] 
\draw    (369.17,94.02) -- (369.17,97.76) ;
%Straight Lines [id:da19118354850283326] 
\draw    (291.95,56.08) -- (329.89,56.02) ;
\draw [shift={(310.92,56.05)}, rotate = 539.9100000000001] [color={rgb, 255:red, 0; green, 0; blue, 0 }  ][line width=0.75]    (4.37,-1.32) .. controls (2.78,-0.56) and (1.32,-0.12) .. (0,0) .. controls (1.32,0.12) and (2.78,0.56) .. (4.37,1.32)   ;
%Straight Lines [id:da24873873608592145] 
\draw    (368.89,56.02) -- (406.83,56.02) ;
\draw [shift={(387.86,56.02)}, rotate = 180] [color={rgb, 255:red, 0; green, 0; blue, 0 }  ][line width=0.75]    (4.37,-1.32) .. controls (2.78,-0.56) and (1.32,-0.12) .. (0,0) .. controls (1.32,0.12) and (2.78,0.56) .. (4.37,1.32)   ;
%Straight Lines [id:da7824049923065537] 
\draw    (330.17,54.02) -- (330.17,57.76) ;
%Straight Lines [id:da3012045300513415] 
\draw    (330.69,56.02) -- (370.69,56.02) ;
\draw [shift={(350.69,56.02)}, rotate = 180] [color={rgb, 255:red, 0; green, 0; blue, 0 }  ][line width=0.75]    (4.37,-1.32) .. controls (2.78,-0.56) and (1.32,-0.12) .. (0,0) .. controls (1.32,0.12) and (2.78,0.56) .. (4.37,1.32)   ;
%Straight Lines [id:da21196838611604907] 
\draw    (369.17,54.02) -- (369.17,57.76) ;
%Straight Lines [id:da45860618523240704] 
\draw    (409,95.89) -- (449,95.89) ;
\draw [shift={(429,95.89)}, rotate = 180] [color={rgb, 255:red, 0; green, 0; blue, 0 }  ][line width=0.75]    (4.37,-1.32) .. controls (2.78,-0.56) and (1.32,-0.12) .. (0,0) .. controls (1.32,0.12) and (2.78,0.56) .. (4.37,1.32)   ;
%Straight Lines [id:da22161003688480663] 
\draw    (409,55.95) -- (409,95.95) ;
\draw [shift={(409,75.95)}, rotate = 270] [color={rgb, 255:red, 0; green, 0; blue, 0 }  ][line width=0.75]    (4.37,-1.32) .. controls (2.78,-0.56) and (1.32,-0.12) .. (0,0) .. controls (1.32,0.12) and (2.78,0.56) .. (4.37,1.32)   ;
%Straight Lines [id:da22127008051652886] 
\draw    (489.19,84.9) -- (409,55.89) ;
\draw [shift={(449.1,70.39)}, rotate = 379.89] [color={rgb, 255:red, 0; green, 0; blue, 0 }  ][line width=0.75]    (4.37,-1.32) .. controls (2.78,-0.56) and (1.32,-0.12) .. (0,0) .. controls (1.32,0.12) and (2.78,0.56) .. (4.37,1.32)   ;
%Shape: Circle [id:dp47032528765240333] 
\draw  [fill={rgb, 255:red, 155; green, 155; blue, 155 }  ,fill opacity=1 ][line width=0.75]  (406.94,55.95) .. controls (406.94,54.81) and (407.86,53.89) .. (409,53.89) .. controls (410.14,53.89) and (411.06,54.81) .. (411.06,55.95) .. controls (411.06,57.09) and (410.14,58.01) .. (409,58.01) .. controls (407.86,58.01) and (406.94,57.09) .. (406.94,55.95) -- cycle ;
%Shape: Circle [id:dp6424196831497302] 
\draw  [fill={rgb, 255:red, 155; green, 155; blue, 155 }  ,fill opacity=1 ][line width=0.75]  (406.94,95.95) .. controls (406.94,94.81) and (407.86,93.89) .. (409,93.89) .. controls (410.14,93.89) and (411.06,94.81) .. (411.06,95.95) .. controls (411.06,97.09) and (410.14,98.01) .. (409,98.01) .. controls (407.86,98.01) and (406.94,97.09) .. (406.94,95.95) -- cycle ;
%Straight Lines [id:da17620144132778703] 
\draw    (449.29,93.89) -- (449.29,97.62) ;
%Straight Lines [id:da6407905231431521] 
\draw    (449.8,95.89) -- (488.29,95.89) ;
\draw [shift={(469.04,95.89)}, rotate = 180] [color={rgb, 255:red, 0; green, 0; blue, 0 }  ][line width=0.75]    (4.37,-1.32) .. controls (2.78,-0.56) and (1.32,-0.12) .. (0,0) .. controls (1.32,0.12) and (2.78,0.56) .. (4.37,1.32)   ;
%Straight Lines [id:da41814695396665336] 
\draw    (488.29,93.89) -- (488.29,97.62) ;
%Straight Lines [id:da914196866116733] 
\draw    (411.06,55.95) -- (449,55.89) ;
\draw [shift={(430.03,55.92)}, rotate = 539.9100000000001] [color={rgb, 255:red, 0; green, 0; blue, 0 }  ][line width=0.75]    (4.37,-1.32) .. controls (2.78,-0.56) and (1.32,-0.12) .. (0,0) .. controls (1.32,0.12) and (2.78,0.56) .. (4.37,1.32)   ;
%Straight Lines [id:da572979008712885] 
\draw    (449.29,53.89) -- (449.29,57.62) ;
%Straight Lines [id:da31602226999091587] 
\draw    (449.8,55.89) -- (488.29,55.89) ;
\draw [shift={(469.04,55.89)}, rotate = 180] [color={rgb, 255:red, 0; green, 0; blue, 0 }  ][line width=0.75]    (4.37,-1.32) .. controls (2.78,-0.56) and (1.32,-0.12) .. (0,0) .. controls (1.32,0.12) and (2.78,0.56) .. (4.37,1.32)   ;
%Straight Lines [id:da9255461910377967] 
\draw    (488.29,53.89) -- (488.29,57.62) ;
%Straight Lines [id:da8680713785262406] 
\draw    (130.86,95.55) -- (168.94,95.55) ;
\draw [shift={(149.9,95.55)}, rotate = 180] [color={rgb, 255:red, 0; green, 0; blue, 0 }  ][line width=0.75]    (4.37,-1.32) .. controls (2.78,-0.56) and (1.32,-0.12) .. (0,0) .. controls (1.32,0.12) and (2.78,0.56) .. (4.37,1.32)   ;
%Straight Lines [id:da49099522129640905] 
\draw    (168.94,95.55) -- (92.05,69.76) ;
\draw [shift={(130.49,82.65)}, rotate = 378.53999999999996] [color={rgb, 255:red, 0; green, 0; blue, 0 }  ][line width=0.75]    (4.37,-1.32) .. controls (2.78,-0.56) and (1.32,-0.12) .. (0,0) .. controls (1.32,0.12) and (2.78,0.56) .. (4.37,1.32)   ;
%Straight Lines [id:da35089126260151704] 
\draw    (92.14,93.55) -- (92.14,97.28) ;
%Straight Lines [id:da39263013891821585] 
\draw    (92.66,95.55) -- (132.66,95.55) ;
\draw [shift={(112.66,95.55)}, rotate = 180] [color={rgb, 255:red, 0; green, 0; blue, 0 }  ][line width=0.75]    (4.37,-1.32) .. controls (2.78,-0.56) and (1.32,-0.12) .. (0,0) .. controls (1.32,0.12) and (2.78,0.56) .. (4.37,1.32)   ;
%Straight Lines [id:da4659902202480879] 
\draw    (131.14,93.55) -- (131.14,97.28) ;
%Straight Lines [id:da6549264092618012] 
\draw    (130.86,55.55) -- (168.8,55.55) ;
\draw [shift={(149.83,55.55)}, rotate = 180] [color={rgb, 255:red, 0; green, 0; blue, 0 }  ][line width=0.75]    (4.37,-1.32) .. controls (2.78,-0.56) and (1.32,-0.12) .. (0,0) .. controls (1.32,0.12) and (2.78,0.56) .. (4.37,1.32)   ;
%Straight Lines [id:da9136620108747424] 
\draw    (92.14,53.55) -- (92.14,57.28) ;
%Straight Lines [id:da4203698811215475] 
\draw    (92.66,55.55) -- (132.66,55.55) ;
\draw [shift={(112.66,55.55)}, rotate = 180] [color={rgb, 255:red, 0; green, 0; blue, 0 }  ][line width=0.75]    (4.37,-1.32) .. controls (2.78,-0.56) and (1.32,-0.12) .. (0,0) .. controls (1.32,0.12) and (2.78,0.56) .. (4.37,1.32)   ;
%Straight Lines [id:da8080835699334805] 
\draw    (131.14,53.55) -- (131.14,57.28) ;
%Curve Lines [id:da5033471047505094] 
\draw [color={rgb, 255:red, 255; green, 0; blue, 0 }  ,draw opacity=1 ]   (211.29,54) .. controls (211.57,44.86) and (250.62,44.52) .. (250.29,54) ;
\draw [shift={(231.02,47.02)}, rotate = 356.85] [fill={rgb, 255:red, 255; green, 0; blue, 0 }  ,fill opacity=1 ][line width=0.08]  [draw opacity=0] (5.36,-2.57) -- (0,0) -- (5.36,2.57) -- (3.56,0) -- cycle    ;
%Curve Lines [id:da05186335667705899] 
\draw [color={rgb, 255:red, 255; green, 0; blue, 0 }  ,draw opacity=1 ]   (330.17,54.02) .. controls (330.46,44.88) and (369.51,44.54) .. (369.17,54.02) ;
\draw [shift={(349.91,47.04)}, rotate = 356.85] [fill={rgb, 255:red, 255; green, 0; blue, 0 }  ,fill opacity=1 ][line width=0.08]  [draw opacity=0] (5.36,-2.57) -- (0,0) -- (5.36,2.57) -- (3.56,0) -- cycle    ;
%Curve Lines [id:da46449156509208356] 
\draw [color={rgb, 255:red, 255; green, 0; blue, 0 }  ,draw opacity=1 ]   (211,96) .. controls (211.29,86.86) and (250.33,86.52) .. (250,96) ;
\draw [shift={(230.73,89.02)}, rotate = 356.85] [fill={rgb, 255:red, 255; green, 0; blue, 0 }  ,fill opacity=1 ][line width=0.08]  [draw opacity=0] (5.36,-2.57) -- (0,0) -- (5.36,2.57) -- (3.56,0) -- cycle    ;
%Curve Lines [id:da49645055225645796] 
\draw [color={rgb, 255:red, 255; green, 0; blue, 0 }  ,draw opacity=1 ]   (330.17,94.02) .. controls (330.46,84.88) and (369.51,84.54) .. (369.17,94.02) ;
\draw [shift={(349.91,87.04)}, rotate = 356.85] [fill={rgb, 255:red, 255; green, 0; blue, 0 }  ,fill opacity=1 ][line width=0.08]  [draw opacity=0] (5.36,-2.57) -- (0,0) -- (5.36,2.57) -- (3.56,0) -- cycle    ;
%Curve Lines [id:da8951485516690127] 
\draw [color={rgb, 255:red, 255; green, 0; blue, 0 }  ,draw opacity=1 ]   (92.14,93.55) .. controls (92.43,84.4) and (131.48,84.07) .. (131.14,93.55) ;
\draw [shift={(111.88,86.56)}, rotate = 356.85] [fill={rgb, 255:red, 255; green, 0; blue, 0 }  ,fill opacity=1 ][line width=0.08]  [draw opacity=0] (5.36,-2.57) -- (0,0) -- (5.36,2.57) -- (3.56,0) -- cycle    ;
%Curve Lines [id:da3568473502535743] 
\draw [color={rgb, 255:red, 255; green, 0; blue, 0 }  ,draw opacity=1 ]   (92.14,53.55) .. controls (92.43,44.4) and (131.48,44.07) .. (131.14,53.55) ;
\draw [shift={(111.88,46.56)}, rotate = 356.85] [fill={rgb, 255:red, 255; green, 0; blue, 0 }  ,fill opacity=1 ][line width=0.08]  [draw opacity=0] (5.36,-2.57) -- (0,0) -- (5.36,2.57) -- (3.56,0) -- cycle    ;
%Curve Lines [id:da17115390195744062] 
\draw [color={rgb, 255:red, 255; green, 0; blue, 0 }  ,draw opacity=1 ]   (449.29,53.89) .. controls (449.57,44.75) and (488.62,44.41) .. (488.29,53.89) ;
\draw [shift={(469.02,46.91)}, rotate = 356.85] [fill={rgb, 255:red, 255; green, 0; blue, 0 }  ,fill opacity=1 ][line width=0.08]  [draw opacity=0] (5.36,-2.57) -- (0,0) -- (5.36,2.57) -- (3.56,0) -- cycle    ;
%Curve Lines [id:da6991109514432552] 
\draw [color={rgb, 255:red, 255; green, 0; blue, 0 }  ,draw opacity=1 ]   (449.29,93.89) .. controls (449.57,84.75) and (488.62,84.41) .. (488.29,93.89) ;
\draw [shift={(469.02,86.91)}, rotate = 356.85] [fill={rgb, 255:red, 255; green, 0; blue, 0 }  ,fill opacity=1 ][line width=0.08]  [draw opacity=0] (5.36,-2.57) -- (0,0) -- (5.36,2.57) -- (3.56,0) -- cycle    ;

% Text Node
\draw (287.22,46.66) node [anchor=north west][inner sep=0.75pt]  [font=\tiny] [align=left] {1};

\end{tikzpicture}
\caption{An example of the missing edges added by Definition~\ref{Def: Adding missing edges}. This particular example comes from the monoid $M = \pres{Mon}{a,b,c,d}{b(abc)b = 1, bd = 1}$. Horizontal movement corresponds to reading $abc$, whereas vertical and diagonal movement both correspond to reading $b$. The added red edges correspond to reading $d$. We note that we have omitted adding the loops corresponding to $bd = 1$ to each locally invertible vertex for ease of drawing.}
\label{Fig: Missing edges example}
\end{figure}

As a convention, we will henceforth refer to elements of $V(\fU)$ by specifying literal forms $[(u_0 ; \xi)]$, where $u_0 \in \Lambda^\ast$ and $\xi \in \fPe$. The construction of $\fU$ guarantees that this specifies precisely as much information as a pair $[(m, \xi)]$, where $m \in U(M)$. We will now turn towards studying the properties of $\fU$ as a graph with respect to the properties of the group of units $U(M)$. In particular, we will see that the context-freeness of $\fU$ is closely connected to that of $U(M)$. 

\subsection{Context-free properties of $\fU$}

We will now prove a key lemma about $\fU$. The group $U(M)$ acts properly discontinuously and transitively via left multiplication on $\MCG{U_\fB(M)}{\fB} \cong \fU_0 / \sim_i$, and this action hence extends to a properly discontinuous action of $U(M)$ on $\fU_0$ with finite quotient. As $\sim_M$ is determined entirely on the equivalence classes of $\sim_i$, and since the edges added to $\fU$  by Definition~\ref{Def: Adding missing edges} are added uniformly within each $\sim_i$-class, the above action extends to a properly discontinuous action of $U(M)$ on $\fU$. By considering the associated undirected and unlabelled graphs ud$(\fU)$ and ud$(\MCG{U_\fB(M)}{\fB})$, we hence obtain the following by the Svar\u{c}-Milnor lemma.

\begin{lemma}\label{Lem: U(M) is qi to fU}
There exists a quasi-isometry $\textnormal{ud}(\MCG{U_\fB(M)}{\fB}) \xrightarrow{\sim} \textnormal{ud}(\fU)$.
\end{lemma}

It is well worth comparing this lemma with the following due to Garretta \& Gray; this may shed some light on the r\^ole played by $\fU$ in the context of $M$.

\begin{lemma}[\cite{Garreta2019}]\label{Lem: R1 is qi to fR}
There exists a quasi-isometry $\textnormal{ud}(\MCG{\RU}{I_0}) \xrightarrow{\sim} \textnormal{ud}(\fR_1)$.
\end{lemma}

We can now state the central theorem about the graph $\fU$. 

\begin{theorem}\label{Thm: Big fat equivalence list}
Let $M$ be a finitely presented special monoid, with group of units $U(M)$. Let $\MCG{U_\fB(M)}{\fB}$ denote the right Cayley graph of $U_\fB(M) \cong U(M)$, and let $\fU$ be the Sch\"utzenberger graph of the units of $M$. Then the following are equivalent: 
\begin{enumerate}
\item $U(M)$ is virtually free.
\item $U(M)$ is context-free.
\item $\MCG{U_\fB(M)}{\fB}$ is a context-free graph.
\item $\textnormal{ud}(\MCG{U_\fB(M)}{\fB})$ is quasi-isometric to a tree.
\item $\textnormal{ud}(\fU)$ is quasi-isometric to a tree.
\item $\textnormal{ud}(\fU)$ has finite tree-width.
\item $\fU$ is a context-free graph.
\end{enumerate}
\end{theorem}
\begin{proof}
Let $M = \pres{Mon}{A}{R_i = 1 \: (i \in I)}$ be a finitely presented special monoid. We begin by noting that it follows from e.g.\ \cite{Zhang1992} that the group of units of a finitely presented special monoid is finitely generated, so we will throughout the proof assume that $U(M)$ is finitely generated.

$(1 \iff 2)$ This is the Muller-Schupp Theorem \cite{MullerSchupp1983}.

$(2 \iff 3)$ A finitely generated group $G$ is context-free if and only if $\GCG{G}{S}$ is a context-free graph for some (any) finite set $S$ which generates $G$ as a group; this is precisely the statement of Theorem~2.9 of \cite{Muller1985}. When additionally $S$ generates $G$ as a monoid, then by Proposition~\ref{Prop: lud MCG is iso to GCG}, $\GCG{G}{S}$ is context-free if and only if $\MCG{G}{S}$ is context-free. Since $\fB$ generates $U(M)$ as a monoid, we are done.

$(1 \iff 4)$ A group is virtually free if and only if its group Cayley graph with respect to (any) finite generating set is quasi-isometric to a tree; the forward implication is obvious, and the converse is well-known, e.g.\ as Theorem~7.19 of \cite{Ghys1990}. Clearly, the undirected monoid Cayley graph of a group is quasi-isometric to its group Cayley graph, and we have the equivalence. As an aside, it may be noted that we have $(3 \implies 4)$ for \textit{any} context-free graph, a statement which appears at the end of the proof of Lemma~8.4 of \cite{Chalopin2017}.

$(4 \iff 5)$ This is immediate by Lemma~\ref{Lem: U(M) is qi to fU}.

$(5 \iff 6)$ We begin with a proposition: let $\Gamma$ and $\Gamma^\prime$ be graphs such that $\Gamma$ has finite tree-width and such that the degree of $\Gamma$ and $\Gamma^\prime$ are uniformly bounded by some constant $d$. Then Proposition~5.17 of \cite{Diekert2017} says that if $\Gamma^\prime$ is quasi-isometric to $\Gamma$, then $\Gamma^\prime$ has finite tree-width too. Hence, for the forward implication, if ud$(\fU)$ is quasi-isometric to a tree, then as ud$(\fU)$ has bounded degree by finiteness of $\Lambda$, and as trees clearly have finite tree-width, we may use the aforementioned proposition to conclude that ud$(\fU)$ has finite tree-width. For the converse implication, assume that ud$(\fU)$ has finite tree-width. Then as ud$(\fU)$ has bounded degree and is quasi-isometric to ud$(\MCG{U_\fB(M)}{\fB})$, we have that ud$(\MCG{U_\fB(M)}{\fB})$ also has finite tree-width, again by the aforementioned Proposition. But ud$(\MCG{U_\fB(M)}{\fB})$ is a vertex-transitive connected graph of bounded degree. Hence by Theorem~3.10 of \cite{Kuske2005} we have that $\MCG{U_\fB(M)}{\fB})$ is context-free, and so by the equivalence $(3 \iff 5)$ we are done.

$(6 \iff 7)$ By Theorem~3.10 of \cite{Kuske2005}, if $\Gamma$ is a connected graph of bounded degree such that Aut$(\Gamma)$ has only finitely many orbits on $\Gamma$, then ud$(\Gamma)$ having finite tree-width is equivalent to $\Gamma$ being context-free. Hence it suffices to show that $\fU$ has only finitely many orbits under automorphism, as $\fU$ is connected and has bounded degree since $\Lambda$ is finite. But a co-compact (and hence co-finite) action of $U(M)$ by automorphisms on $\fU$ was constructed earlier when proving Lemma~\ref{Lem: U(M) is qi to fU}, and we are done.
\end{proof}

Structurally, $\fU$ has a rather close connection with $\mathscr{H}_1$, the Green's $\mathscr{H}$-class of the identity element. First, note that $\mathscr{H}_1 = U(M)$, and thus taking the subgraph of $\MCG{M}{A}$ induced on $\mathscr{H}_1$ will in general produce a disconnected graph. To remedy this, for a word $w \in A^\ast$, let $\Gamma_w$ denote the subgraph of $\MCG{M}{A}$ induced on the set of vertices $\{ \pi(w') \mid \exists w'' \in A^\ast, w \equiv w' w''\}$. For example, if $v \in A^\ast$ is a word with $\pi(v) = 1$, then $\Gamma_v$ is a walk from $1$ to $1$, visiting a number of $\mathscr{H}$-classes inside the $\mathscr{R}$-class of $1$. For a right invertible word $v \in A^\ast$, let $H_v$ (note that this is distinct from $\mathscr{H}_v$) denote the set of all $\mathscr{H}$-classes visited by $\Gamma_v$. Then, by a standard application of Green's Lemma via the left action of $U(M)$ on $\mathscr{R}_1$, it follows that $\fU$ is isomorphic to $\bigcup_{\lambda \in \Lambda} H_\lambda$. We note further that $\Lambda$ is a finite set, and that $\bigcup_{\lambda \in \Lambda} H_\lambda \cong \bigcup_{i=1}^{|\Lambda|} \mathscr{H}_1$, yielding another proof of Lemma~\ref{Lem: U(M) is qi to fU}.

\subsection{Ends of $\fU$} There are now two more lemmata which fall out of elementary discussion, both of which relate to the constructions in Section~\ref{Sec: Trees and Determinisations}.

\begin{lemma}\label{Lem: F = I(F) cup N(F)}
For all $w_1, w_2 \in V(\fU)$ we have
\[
\fU(w_1) \sim \fU(w_2) \implies \left( w_1 \: \textnormal{locally invertible} \iff w_2 \: \textnormal{locally invertible} \right)
\]
In particular any set $F(\fU)$ consisting of representatives of all frontier points of $\fU$ can be written as a disjoint union $I(\fU) \sqcup N(\fU)$ of the locally invertible and not locally invertible representatives, respectively.
\end{lemma}
\begin{proof}
Assume that $\fU(w_1) \sim \fU(w_2)$. Assume that $w_1$ is locally invertible and that $w_2$ is not locally invertible. Then in $\textnormal{ud}(\fU(w_2))$ there is a set of undirected non-trivial paths from $w_2$ to a locally invertible vertex passing through no other locally invertible vertices. Choose one such path. This path will be labelled either by a suffix $s$ of a piece, or by the involution of a prefix $p$ of a piece, denoted $\bar{p}$. In either case, since $w_2$ is not locally invertible, $s$ or $p$, respectively, will be proper. Since $\fU(w_1) \sim \fU(w_2)$, this same path can be found from $w_2$ in $\textnormal{ud}(\fU(w_2))$. But in the first case, $s$ will now be readable from a locally invertible vertex and hence a prefix of a piece; thus $s$ is invertible, and since it is a proper suffix of a piece we have a contradiction. In the second case, $\bar{p}$ will be readable from a locally invertible vertex, and hence $p$ is also a suffix of a piece; thus $p$ is invertible, and since it is a proper prefix of a piece we have a contradiction. Thus $w_2$ must be locally invertible. By symmetry, we have the claim. 
\end{proof}

Clearly, any $\sim_i$-class is isomorphic to the quotient of $\fU$ by the action of $U(M)$, and in each $\sim_i$-class there is only a single locally invertible vertex, and this corresponds to the orbit of the root $1$ of $\fU$. In particular, $\fU$ is $N(\fU)$-full, in the terminology of Section~\ref{Sec: Trees and Determinisations}. Furthermore, we have the following graphical interpretation of the fact that a proper prefix of a piece cannot be a suffix of a piece.

\begin{lemma}\label{Lem: fU is N(fU) overlap free}
The graph $\fU$ is $N(\fU)$-overlap-free.
\end{lemma}
\begin{proof}
Assume for contradiction that $p_1$ and $p_2$ are two paths contradicting the $N(\fU)$-overlap-free property, i.e.\ such that $p_1$ is a non-empty path starting in $1$ and $p_2$ is a path starting in a non-locally invertible vertex, but ending in a locally invertible vertex, and with $\ell(p_1) \equiv \ell(p_2) \equiv \alpha$. Then $p_2$ must begin with a subpath to the nearest locally invertible vertex, having path label $\alpha'$, which must therefore be a proper suffix of a piece. But since $p_2$ begins with such a subpath, so too does $p_1$; hence $\alpha'$ also labels a path from $1$, and therefore it must have a suffix which is a prefix of a piece, and we have a contradiction. 
\end{proof}

Let $\mathcal{N}$ denote the set of all non-locally invertible vertices. Then it is clear from the definition of the action of $U(M)$ that this action respects the structure of $\mathcal{N}$.  

\begin{corollary}\label{Cor: fU satisfies bounded folding}
The triple $(\fU, U(M), N(\fU))$ satisfies the bounded folding condition. Furthermore, for every $\phi \in U(M)$, we have that $(v \in \mathcal{N}) \iff (\phi(v) \in \mathcal{N})$.
\end{corollary}

This completes the description of $\fU$. Using this graph, which captures a great deal of graphical information regarding the group of units of $M$, we will now turn to a description of the right units. 

\section{The Sch\"{u}tzenberger Graph of $1$}\label{Sec: Constructing fR1}

The submonoid $R(M)$ of right invertible elements of a finitely presented special monoid is in many ways similarly behaved to the group of units of $M$. Zhang \cite{Zhang1992} completely described the algebraic properties of this submonoid, which can be summarised in the following proposition.

\begin{prop}[Zhang \cite{Zhang1992}]
Let $M = \pres{Mon}{A}{R_i = 1 \: (i \in I)}$ be a finitely presented special monoid. Then 
\[
R(M) \cong \operatorname{FM}(k) \ast U(M)
\]
where $\operatorname{FM}(k)$ is a free monoid of rank $k$, for some $k \in \mathbb{N}$. In particular, if $M$ is finitely presented, then $R(M)$ is finitely presented.
\end{prop}

However, whereas this isomorphism induces a quasi-isometric embedding, once suitably defined, of $R(M)$ into $M$ by Lemma~\ref{Lem: R1 is qi to fR}, it does not describe the structure of the subgraph of the Cayley graph $\MCG{M}{A}$ induced on the set of right invertible elements. The goal of this section is to describe the structure of this graph.

\subsection{Graphically representing right units}

Recall the definition of $\fR_1$ as the connected component of the identity element in the right Cayley graph of $M$. We are interested in the connections between the context-freeness of $U(M)$ and $\fR_1$. The following (incorrect!) argument along this line is very tempting. Assume that $U(M)$ is virtually free. Then the Cayley graph of $U(M)$ has decidable monadic second-order theory. It is well known that the Cayley graph of any finitely generated free monoid $F$ has decidable monadic second-order theory. Hence, by \cite[Theorem~6.3(b)]{Kuske2003}, the Cayley graph of the free product $U(M) \ast F$ has decidable monadic second order theory. Therefore, the Cayley graph of the monoid of right units $\mathscr{R}_1$ has decidable monadic second order theory. One would now wish to conclude that the graph $\fR_1$ has decidable monadic second order theory. But this isomorphism of monoids is \textit{not} an isomorphism of graphs. Indeed, there is not even any natural way to embed the Cayley graph of $U(M) \ast F$ into the Cayley graph of $M$. The best we are able to say in this direction is Lemma~\ref{Lem: R1 is qi to fR}, i.e.\ that there is a quasi-isometry between the two graphs, considered as undirected graphs. But there are many examples of graphs quasi-isometric to $\mathbb{N}$ which have undecidable monadic second-order theory, see e.g.\ \cite{Elgot1966}, despite the fact that $\mathbb{N}$ is even context-free. We therefore have no \textit{a priori} reason to expect that there should be any connection at all between the context-freeness of $U(M)$ and $\fR_1$. 

\begin{remark}
It is at this stage, and throughout the remainder of the paper, key to note the structural differences between $\fR_1$ and the Cayley graph of the right units $R(M)$; to illustrate this very concretely, we note that $\fR_1$ is not a Cayley graph of anything, unless $M$ is a group. Indeed, if $M$ is not a group, then $\fR_1$ is not even a regular graph.
\end{remark}

The reader might now be sufficiently convinced that another approach, based on $\fU$ and the tools developed in Section~\ref{Sec: Constructing fU}, is needed. We now present such an approach, starting with a rather straightforward, but central theorem. Let $\TU$ denote the graph $\T(\fU, N(\fU))$, where $N(\fU)$ as earlier denotes a set of representatives of the non-locally invertible frontier points of $\fU$. As the right invertible elements of $M$ can all be represented by elements of $\Xi^\ast$, it should not be surprising that $\TU$ has a representative for every right invertible element. Indeed, in certain restricted cases (viz. precisely when $\Xi$ is a code as a subset of $A^\ast$) we have that $\fR_1$ is isomorphic to $\TU$. In a more general setting, however, it is fully possible for the graph $\TU$ to be non-deterministic. We now show that, modulo this non-determinism, $\TU$ completely describes $\fR_1$. 

\begin{theorem}\label{Thm: R_1 is TU determinised}
Let $\eta$ be the determinising congruence on $\TU$. Then $\TU / \eta \cong \fR_1$. 
\end{theorem}
\begin{proof}
First, note that there is a map $\phi \colon \TU \to \fR_1$ which, if we let $\hat{u}_i = [(u_i ; \xi_i)]$, is given by 
\[
\phi((\hat{u}_0, \hat{u}_1, \dots, \hat{u}_k)) = \pi(u_0 \xi_0 u_1 \xi_1 \cdots u_k \xi_k).
\]
This is well-defined as the definition of $\fU$ ensures independence of the choice of representatives $u_i$ and $x_i$ for each $[(u_i ; \xi_i)]$. Furthermore, as every right invertible element can be expressed as an element of $\Xi^\ast$, it follows that $\phi$ is surjective. We will now show that $\ker(\phi) = \{ (u, v) \mid \phi(u) = \phi(v)\} = \eta$, which will yield the claim. We only show that the two coincide for the vertex sets, as for the edge sets the proof is very similar. 

\vspace{0.5cm}

$(\eta \subseteq \ker(\phi))$. Assume that $(u,v) \in \eta$ in $V(\TU)$. Then there exists walks $p_1 \colon 1 \xrightarrow{\alpha} u$ and $p_2 \colon 1 \xrightarrow{\alpha} v$, for some $\alpha \in A^\ast$. In particular there exist $u_i', v_i' \in A^\ast$ invertible and $\xi_i', \zeta_i' \in \fP$ such that 
\begin{align*}
u = ([(u_0' ; \xi_0')], \dots, [(u_k' ; \xi_k')]) \\
v = ([(v_0' ; \zeta_0')], \dots, [(v_n' ; \zeta_n')])
\end{align*}
and $\alpha \equiv u_0' \xi_0' \cdots u_k' \xi_k' \equiv v_0' \zeta_0' \cdots v_n' \zeta_n'$. Hence 
\[
\phi(u) = \pi(\alpha) = \phi(v)
\]
and so $(u, v) \in \ker(\phi)$. 

\vspace{0.5cm}

$(\ker(\phi)) \subseteq \eta)$. Assume $(u, v) \in \ker(\phi)$. Let $u = [(u_0 ; \xi_0), \dots, (u_k ; \xi_k)]$ and $v = [(v_0 ; \zeta_0), \dots, (v_n ; \zeta_n]$. Now since $u_0 \xi_0 \cdots u_k \xi_k \lra{M} v_0 \zeta_0 \cdots v_n \zeta_n$, we have by Lemma~\ref{Lem: Special monoid has ancestors} that there is some $W \in A^\ast$ s.t. 
\begin{align*}
W &\xra{M} u_0 \xi_0 \cdots u_k \xi_k \\
W &\xra{M} v_0 \zeta_0 \cdots v_n \zeta_n.
\end{align*}
We now make the claim that at every vertex $w \in V(\TU)$ and for every relator word $R_i \: (i \in I)$, there is a walk $p \colon u' \xrightarrow{R_i} u'$. Assume that $w = ([w_0 ; \mu_0], \dots, [(w_\ell ; \mu_\ell)])$, and let $R_i$ be a relator word. Then if $\mu_\ell \in \fP$ satisfies $\mu_\ell \equiv \varepsilon$, then $[(w_\ell ; \mu_\ell)]$ is locally invertible and hence within the same copy of $\fU$, one may read $R_i$, as in $\fU$ one may from any locally invertible vertex read any relator word. On the other hand, if $\mu_\ell$ is a non-trivial prefix, then onto it is attached a copy of $\fU$; thus, as the root of $\fU$ is locally invertible, one may by the same argument as above read $R_i$ into this attached copy. 

Thus, since $W$ is obtained from $u_0 \xi_0 \cdots u_k \xi_k$ by inserting a number of relator words, there is a walk from $1$ to $u$ in $\TU$ labelled by $W$. But $W$ is also obtained from $v_0 \zeta_0 \cdots v_n \zeta_n$ by inserting a number of relator words; hence there is a walk from $1$ to $v$ in $\TU$ labelled by $W$. Determinising $\TU$ will hence identify $u$ and $v$; in other words, $(u, v) \in \eta$. Hence $\eta = \ker(\phi)$.
\end{proof}

The technique of obtaining a word $W$ which rewrites in a single direction to both $u$ and $v$ in the second part of the above proof is quite elementary, but ties in with work by Stephen \cite{Stephen1987}. For brevity, and the next example alone, we define the \textit{loop graph} of $M$ as follows. Let $L$ be the flower graph of the relator words of $M$, rooted at $1$. Then the loop graph $L^\ast$ of $M$ is defined as $\T(L, V(L) \setminus \{ 1\})$. In other words, it is the graph obtained by starting with a single vertex, and then adding a loop traversing the length of all relator words from this vertex; we then iterate the procedure on all of the new vertices obtained, and define the loop graph as the direct limit of this procedure.  It should be noted that it follows from work due to Stephen \cite{Stephen1987} that the determinised form of this graph is isomorphic to $\fR_1$. However, we will not use this result, and instead illustrate the folding in $\TU$ rather elementarily in the example below. 

\begin{example}
Let $M = \pres{Mon}{a, p, q}{apa = 1, aqa = 1}$. Then $M \cong \mathbb{Z}$. Indeed, it is not difficult to realise that $p =_M q$, both being equal to the unique inverse of $a^2$. Formally, we have the following chain of deductions: 
\[
p \leftarrow p(aqa) \leftarrow (aqa)paqa \rightarrow aqqa \leftarrow aq(aqa)qa \rightarrow qaqa \rightarrow q.
\]
Applying the method given in the proof of Lemma~\ref{Lem: Special monoid has ancestors}, we see that we can pick $W \equiv aqapaaqaqa$. This word can either be factored as $(aqa)p(aaqaqa)$, where $aqa =_M aaqaqa =_M 1$, or as $(aqapaa)q(aqa)$, where $aqapaa =_M aqa =_M 1$. In the first case, rewriting both factors equal to $1$ to the empty word, we obtain $p$; in the second, we obtain $q$. We can graphically represent this as is done in Figure~\ref{Fig: p and q aqaqpapqapq}, by drawing a walk labelled by $p$ (resp. $q$) from $1$, and to every vertex along the path attaching loops equal to $1$, as prescribed by the rewriting.

\begin{figure}[h]
\begin{tikzpicture}[>=stealth',thick,scale=0.8,el/.style = {inner sep=2pt, align=left, sloped}]%
                        \node (l0) [circle, draw, fill=black!50,
                        inner sep=0pt, minimum width=8pt] at (0,0) {};
                        \node (l1) [circle, draw, fill=black!50,
                        inner sep=0pt, minimum width=4pt] at (-1,2) {};
                        \node (l2) [circle, draw, fill=black!50,
                        inner sep=0pt, minimum width=4pt] at (1,2) {};                        
                        \node (l3) [circle, draw, fill=black!50,
                        inner sep=0pt, minimum width=4pt] at (4,0) {};
                        \node (l4) [circle, draw, fill=black!50,
                        inner sep=0pt, minimum width=4pt] at (3,2) {};
                        \node (l5) [circle, draw, fill=black!50,
                        inner sep=0pt, minimum width=4pt] at (5,2) {};
                        \node (l6) [circle, draw, fill=black!50,
                        inner sep=0pt, minimum width=4pt] at (2.5,3) {};
                        \node (l7) [circle, draw, fill=black!50,
                        inner sep=0pt, minimum width=4pt] at (3.5,3) {};
\path[->] 
    (l0)  edge node[below]{$p$}         (l3)
    (l0)  edge node[below=-2pt, xshift=-4pt]{$a$}         (l1)
    (l1)  edge node[above]{$q$}         (l2)
    (l2)  edge node[below=-2pt, xshift=4pt]{$a$}         (l0)
    (l3)  edge node[below=-2pt, xshift=-4pt]{$a$}         (l4)
    (l4)  edge node[above]{$q$}         (l5)
    (l5)  edge node[below=-2pt, xshift=4pt]{$a$}         (l3)
    (l4)  edge node[below=-4pt, xshift=-4pt]{$a$}         (l6)
    (l6)  edge node[above]{$q$}         (l7)
    (l7)  edge node[below=-4pt, xshift=4pt]{$a$}         (l4);

    \node (r0) [circle, draw, fill=black!50,
    inner sep=0pt, minimum width=8pt] at (7,0) {};
    \node (r1) [circle, draw, fill=black!50,
    inner sep=0pt, minimum width=4pt] at (6,2) {};
    \node (r2) [circle, draw, fill=black!50,
    inner sep=0pt, minimum width=4pt] at (8,2) {};                        
    \node (r3) [circle, draw, fill=black!50,
    inner sep=0pt, minimum width=4pt] at (11,0) {};
    \node (r4) [circle, draw, fill=black!50,
    inner sep=0pt, minimum width=4pt] at (10,2) {};
    \node (r5) [circle, draw, fill=black!50,
    inner sep=0pt, minimum width=4pt] at (12,2) {};
    \node (r6) [circle, draw, fill=black!50,
    inner sep=0pt, minimum width=4pt] at (7.5,3) {};
    \node (r7) [circle, draw, fill=black!50,
	inner sep=0pt, minimum width=4pt] at (8.5,3) {};

\path[->] 
    (r0)  edge node[below]{$q$}         (r3)
    (r0)  edge node[below=-2pt, xshift=-4pt]{$a$}         (r1)
    (r1)  edge node[above]{$q$}         (r2)
    (r2)  edge node[below=-2pt, xshift=4pt]{$a$}         (r0)
    (r3)  edge node[below=-2pt, xshift=-4pt]{$a$}         (r4)
    (r4)  edge node[above]{$q$}         (r5)
    (r5)  edge node[below=-2pt, xshift=4pt]{$a$}         (r3)
    (r2)  edge node[below=-4pt, xshift=-4pt]{$a$}         (r6)
    (r6)  edge node[above]{$p$}         (r7)
    (r7)  edge node[below=-4pt, xshift=4pt]{$a$}         (r2);    

\end{tikzpicture}
\caption{Left: The first factorisation of $W$, giving $p$, considered as a subgraph of the loop graph of $M$. Right: The second factorisation of $W$, giving $q$. In both cases, the larger vertex corresponds to the root of the loop graph.}
\label{Fig: p and q aqaqpapqapq}
\end{figure}
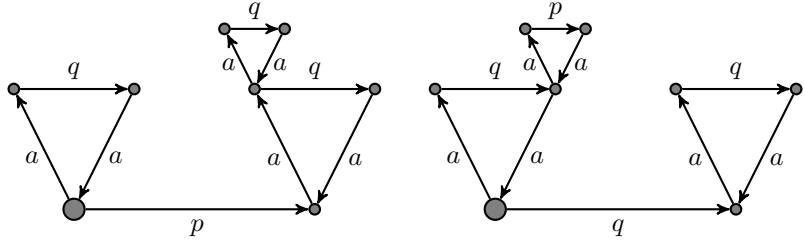

We can then assemble these two factorisations into a single graph; it is clear from the definition of $W$ that determinising this graph will identify the two vertices marked by $v_p$ and $v_q$. This is shown in Figure~\ref{Fig: p and q determinised!}. The key step in the second part of the proof of Theorem~\ref{Thm: R_1 is TU determinised} is that a graph associated in this manner to \textit{any} pair $u, v$ such that $\phi(u) = \phi(v)$ appears as a quotient graph inside $\TU$; hence determinising $\TU$ will identify all such pairs.

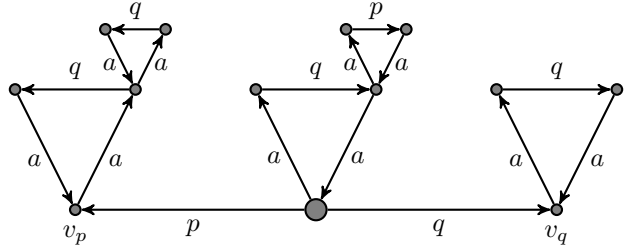
\begin{figure}[h]
\begin{tikzpicture}[>=stealth',thick,scale=0.8,el/.style = {inner sep=2pt, align=left, sloped}]%
                        \node (l0) [circle, draw, fill=black!50,
                        inner sep=0pt, minimum width=8pt] at (0,0) {};

                        \node (l3)[label=below:$v_p$] [circle, draw, fill=black!50,
                        inner sep=0pt, minimum width=4pt] at (-4,0) {};
                        \node (l4) [circle, draw, fill=black!50,
                        inner sep=0pt, minimum width=4pt] at (-3,2) {};
                        \node (l5) [circle, draw, fill=black!50,
                        inner sep=0pt, minimum width=4pt] at (-5,2) {};
                        \node (l6) [circle, draw, fill=black!50,
                        inner sep=0pt, minimum width=4pt] at (-2.5,3) {};
                        \node (l7) [circle, draw, fill=black!50,
                        inner sep=0pt, minimum width=4pt] at (-3.5,3) {};
\path[->] 
    (l0)  edge node[below]{$p$}         (l3)
    (l3)  edge node[below=-2pt, xshift=4pt]{$a$}         (l4)
    (l4)  edge node[above]{$q$}         (l5)
    (l5)  edge node[below=-2pt, xshift=-4pt]{$a$}         (l3)
    (l4)  edge node[below=-4pt, xshift=4pt]{$a$}         (l6)
    (l6)  edge node[above]{$q$}         (l7)
    (l7)  edge node[below=-4pt, xshift=-4pt]{$a$}         (l4);

    \node (r1) [circle, draw, fill=black!50,
    inner sep=0pt, minimum width=4pt] at (-1,2) {};
    \node (r2) [circle, draw, fill=black!50,
    inner sep=0pt, minimum width=4pt] at (1,2) {};                        
    \node (r3)[label=below:$v_q$][circle, draw, fill=black!50,
    inner sep=0pt, minimum width=4pt] at (4,0) {};
    \node (r4) [circle, draw, fill=black!50,
    inner sep=0pt, minimum width=4pt] at (3,2) {};
    \node (r5) [circle, draw, fill=black!50,
    inner sep=0pt, minimum width=4pt] at (5,2) {};
    \node (r6) [circle, draw, fill=black!50,
    inner sep=0pt, minimum width=4pt] at (0.5,3) {};
    \node (r7) [circle, draw, fill=black!50,
	inner sep=0pt, minimum width=4pt] at (1.5,3) {};

\path[->] 
    (l0)  edge node[below]{$q$}         (r3)
    (l0)  edge node[below=-2pt, xshift=-4pt]{$a$}         (r1)
    (r1)  edge node[above]{$q$}         (r2)
    (r2)  edge node[below=-2pt, xshift=4pt]{$a$}         (l0)
    (r3)  edge node[below=-2pt, xshift=-4pt]{$a$}         (r4)
    (r4)  edge node[above]{$q$}         (r5)
    (r5)  edge node[below=-2pt, xshift=4pt]{$a$}         (r3)
    (r2)  edge node[below=-4pt, xshift=-4pt]{$a$}         (r6)
    (r6)  edge node[above]{$p$}         (r7)
    (r7)  edge node[below=-4pt, xshift=4pt]{$a$}         (r2);    

\end{tikzpicture}
\caption{Determinising the graph in the picture will identify the two vertices labelled by $v_p$ and $v_q$, as $p =_M q$ in the monoid $M~=~\pres{Mon}{a,p,q}{apa = 1, aqa = 1}$.}
\label{Fig: p and q determinised!}
\end{figure}
\end{example}

The fact that $\TU$ determinised yields $\fR_1$, together with the fact that $\fU$ satisfies some very strong properties as shown in Section~\ref{Sec: Constructing fU}, we are now in a state where we are ready to apply Theorem~\ref{Thm: Bounded folding gives CF from T}. 

\begin{corollary}\label{Cor: fR1 CF iff U(M) VF}
Let $M = \pres{Mon}{A}{R_i = 1 \: (i \in I)}$ be a finitely presented special monoid. Then $\fR_1$ is a context-free graph if and only if $U(M)$ is a virtually free group.
\end{corollary}
\begin{proof}
$(\impliedby)$ Assume $U(M)$ is virtually free. Then by Theorem~\ref{Thm: Big fat equivalence list}, $\fU$ is a context-free graph. Clearly, $\fU$ is deterministic. Furthermore, by combining Proposition~\ref{Prop: fU is an induced subgraph of fR1} and Theorem~\ref{Thm: R_1 is TU determinised}, we have $\widetilde{\TU}(0) \cong \fR_1(0) \cong \fU$. By Corollary~\ref{Cor: fU satisfies bounded folding}, $(\fU, \textnormal{Aut}_{N(\fU)}, N(\fU))$ satisfies the bounded folding condition. Hence we can apply Theorem~\ref{Thm: Bounded folding gives CF from T}, and conclude that $\widetilde{\TU} \cong \fR_1$ is context-free.

$(\implies)$ Assume that $\fR_1$ is context-free. Since $M$ is finitely presented, so is $U(M)$ by \cite[Theorem~5]{Makanin1966}. Therefore, by the Muller-Schupp theorem, it suffices to show that $U(M)$ is context-free. By \cite[Proposition~25]{Pelecq1996} the automorphism group Aut$(\fR_1)$ of $\fR_1$ as a labelled graph is virtually free, as $\fR_1$ is deterministic, and context-free by assumption. Furthermore, by \cite[Theorem~3]{Stephen1996}, the automorphism group of $\fR_1$ is isomorphic to the Sch\"{u}tzenberger group of $\mathscr{H}_1$. As $\mathscr{H}_1$ is a group $\mathscr{H}$-class, the Sch\"{u}tzenberger group of $\mathscr{H}_1$ is hence isomorphic to $\mathscr{H}_1$ itself \cite{Schutzenberger1957, Clifford1961}. Since $\mathscr{H}_1 = U(M)$, we have that $U(M) \cong \textnormal{Aut}(\fR_1)$ is a context-free group.
\end{proof}
\begin{remark}
Note that in the forward claim of the above result, the assumption that $M$ be finitely presented is only used to make the results easier to state, as we are only generally interested in the finitely presented case. However, \cite[Proposition~25]{Pelecq1996} actually proves that the automorphism group of a deterministic context-free graph is context-free. Hence we always have the implication ($\fR_1$ context-free $\implies$ $U(M)$ context-free). 
\end{remark}

As a finitely generated group is context-free with respect to any finite generating set, and as the group of units of $M$ is independent of the generating set of $M$, we additionally have the following corollary. 

\begin{corollary}
Let $M$ be a finitely presented special monoid, and let $A$ and $B$ be two finite generating sets for $M$. Then $\fR_1(A)$ is a context-free graph if and only if $\fR_1(B)$ is context-free.
\end{corollary}

In particular, we may with impunity say that $\fR_1$ is a context-free graph or that it is not, without regard to the finite generating set chosen. We end this section with a couple of particularly illustrating example of some context-free Sch\"{u}tzenberger graphs; the first is isomorphic to $\T(\fU, N(\fU))$ without any determinisation necessary, whereas the second is not. 

\begin{example}
Let $M = \pres{Mon}{a,b,c,d,e,f}{abc = 1, def = 1}$. Then $\Lambda = \{ abc, def\}$ and $U(M) = 1$. Then $\fU$ is a graph consisting of two triangles, one of whose boundary is labelled by $abc$, and the other by $def$. The only locally invertible vertex in this graph is the shared vertex, i.e.\ the root $1$. It is not hard to see, but beneficial to verify, that $\T(\fU, N(\fU))$ is deterministic, so $\fR_1$ is congruent to this graph. This graph is shown in Figure~\ref{Fig: Schutz of red-blue triangles}.
\end{example}

\begin{figure}
\includegraphics[scale=0.7]{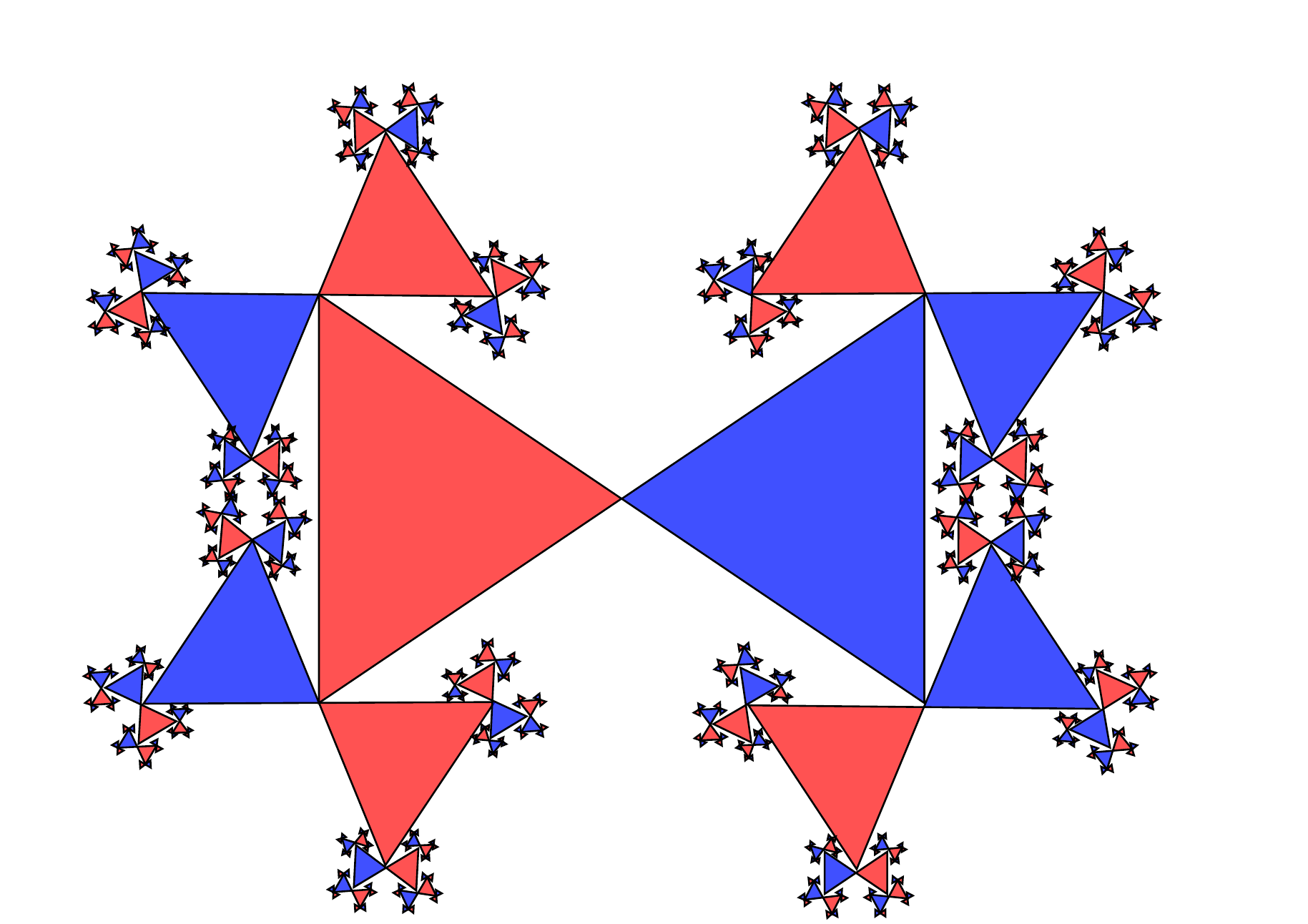}
\caption{The Sch\"{u}tzenberger graph of $1$ (with $1$ being the central vertex) for the special monoid $M$ defined by the presentation $\pres{Mon}{a,b,c,d,e,f}{abc = 1, def = 1}$. The blue triangle corresponds to reading $abc$, and the red to $def$. The graph $\fU$ is isomorphic to the graph with one red and one blue triangle attached to the root.}
\label{Fig: Schutz of red-blue triangles}
\end{figure}

\begin{example}\label{Ex: b(abc)b = 1}
Let $M = \pres{Mon}{a,b,c}{b(abc)b = 1}$. Then $\Lambda = \{ b, abc \}$, and we have $U(M) \cong \mathbb{Z}$. Then in each $\sim_i$-class of $\fU$ there are exactly two non-locally invertible vertices, corresponding to the two non-empty proper prefixes $a$ and $ab$. Attaching a copy of $\fU$ to the first within some $\sim_i$-class, we see that the resulting graph will be non-deterministic, as we can read $b$ from both inside $\fU$ and the attached copy. If we were to attach a copy of $\fU$ to the second non-locally invertible vertex within the same $\sim_i$-class, we see that the determinisation will end up identifying the two attached copies into a single one. The resulting graph is shown in Figure~\ref{Fig: Schutz of b(abc)b = 1}.
\end{example}

\begin{figure}
\begin{center}
\hspace*{-1.2cm}\includegraphics[scale=0.75]{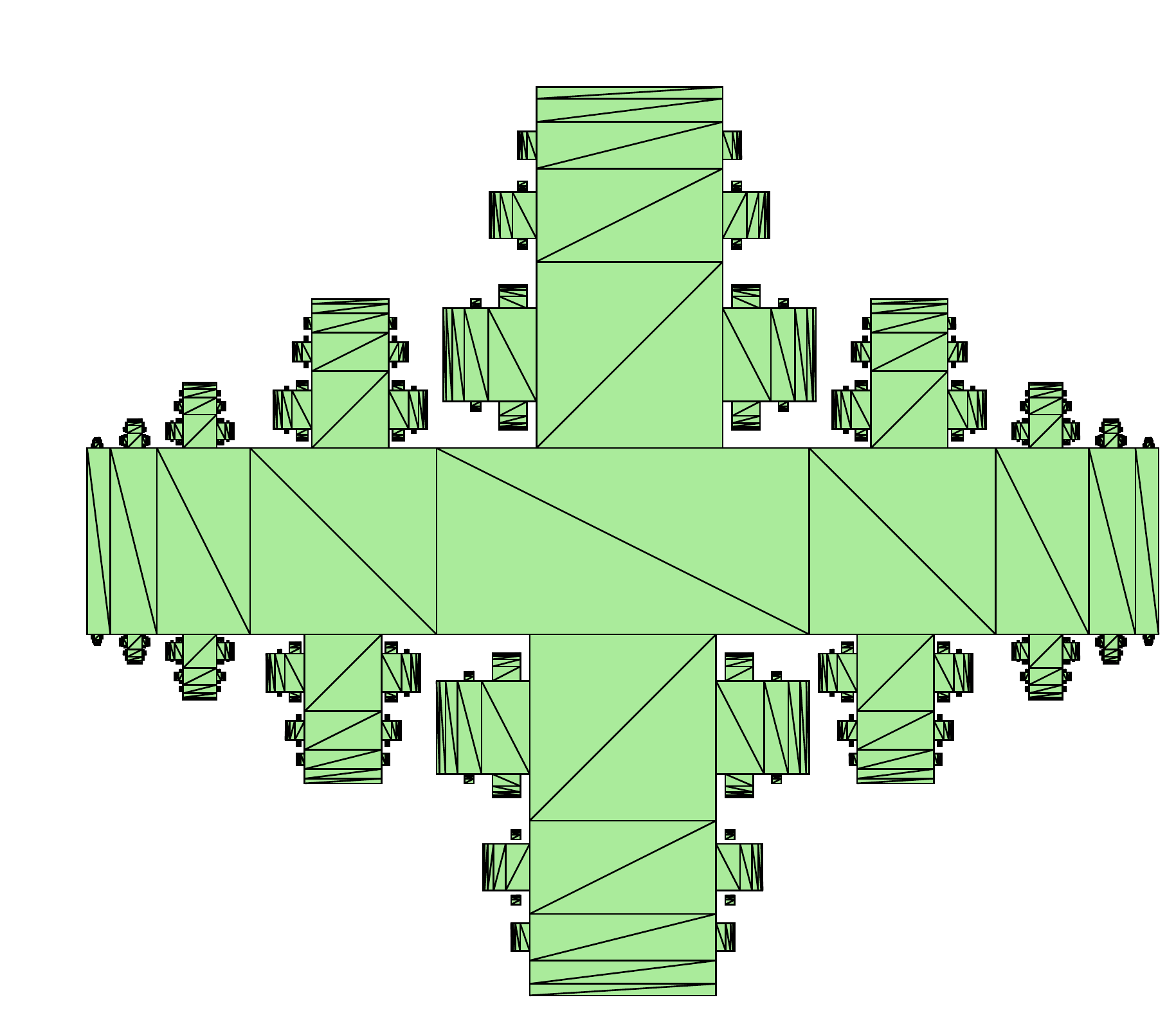}
\end{center}
\caption{The Sch\"{u}tzenberger graph of $1$ (with edge labels suppressed) for the special monoid in Example~\ref{Ex: b(abc)b = 1}. This representation is not entirely accurate, as, say, the vertical copy of $\fU$ in the upper centre continues ``behind'' the central strip infinitely far (but remains distinct from the copy in the lower centre). The reader might imagine that any copy of $\fU$ except the central one is folded in half along the line to which it is attached to the previous copy. We note that $\fU$ is embedded inside $\fR_1$ as the central horizontal strip; these are the vertices of depth $0$. }
\label{Fig: Schutz of b(abc)b = 1}
\end{figure}

\section{Building $\MCG{M}{A}$ From $\fR_1$}\label{Sec: Building MCG from fR1}

For the entirety of this section, let $M = \pres{Mon}{A}{R_i = 1 \: (i \in I)}$ denote a fixed finitely presented special monoid. This section concerns the building of the full monoid Cayley graph of $M$ from the Sch\"{u}tzenberger graph of $1$, and relies on the structural theorems of \cite{Gray2018}. Let $m \in M$ be an arbitrary element of $M$. Let $\mathcal{T} \subseteq A^\ast$ be the set of words irreducible mod $S(T)$ with no right invertible suffix; the notation $\mathcal{T}$ suggests that $\mathcal{T}$ is a form of \textit{transversal} of the $\mathscr{R}$-classes of $M$, which we will see below. By \cite[Proposition~3.7]{Gray2018}, we can then uniquely write $m = [m_0] m'$ where $m_0 \in \mathcal{T}$ and $m' \in \RU$. The following appears as \cite[Corollary~3.12]{Gray2018}.

\begin{theorem}[Gray \& Steinberg '18]\label{Thm: (GS) R_1 is iso to R_m}
Let $m \in M$. Then there is an isomorphism of labelled graphs $\fR_1 \to \fR_m$ sending $1$ to $[m_0]$. If $\Gamma_m$ is the induced subgraph of $\MCG{M}{A}$ consisting of all vertices accessible from $m$, then $\Gamma_m$ is isomorphic to $\MCG{M}{A}$ as labelled graphs via an isomorphism taking $1$ to $[m_0]$. 
\end{theorem}

This is already a great deal of structural information about $\MCG{M}{A}$ modulo $\fR_1$. However, for context-freeness, it does not give a whole lot; it is easy to construct an example of a graph whose strongly connected components are all pairwise isomorphic and context-free, and all $\Gamma_m$ (retaining the notation from the above theorem) are isomorphic, and yet the resulting graph is not context-free. Thus we need more structure. We say that an edge of a digraph \textit{enters} a strong component $C$ of the graph if its initial vertex is not in $C$ and its terminal vertex is in $C$. Dually, we say that an edge \textit{leaves} if its terminal vertex is not in $C$, and its initial vertex is in $C$. The following is \cite[Proposition~3.13]{Gray2018}.

\begin{prop}[\cite{Gray2018}]\label{Prop: (GS) Exactly one edge entering R_m}
Let $m \in \pi(\mathcal{T}) \setminus \{1\}$ (so $m = [m_0]$). Then if $m_0 \equiv x \cdot a$ with $a \in A$, we have $[x] >_{\mathscr{R}} m$, $[a] \not\in \RU$, and $[x] \xr{a} m$ is the unique edge entering $\fR_m$.
\end{prop}

This is a great deal more information. The final piece of the puzzle comes from the information of how $\MCG{M}{A}$ is built up from $\fR_1$. Let $\Gamma$ be the directed graph obtained from $\MCG{M}{A}$ by collapsing each strongly connected component (and its internal edges) to a point. The vertex set of $\Gamma$ is $M / \mathscr{R}$, and there is an edge $(m, a)$ from the $\mathscr{R}$-class $\mathscr{R}_m$ of $m$ to the $\mathscr{R}$-class $\mathscr{R}_{ma}$ of $ma$ if $m \in M, a \in A$, and $\mathscr{R}_m \neq \mathscr{R}_{ma}$. The following is \cite[Theorem~3.14]{Gray2018}.

\begin{theorem}[\cite{Gray2018}]\label{Thm: (GS) M / R is a tree}
The graph $\Gamma$ is isomorphic as a digraph to the Hasse diagram of $M / \mathscr{R}$ ordered
by $\geq_{\mathscr{R}}$. This graph is a regular rooted tree with root $\fR_1$.
\end{theorem}

The tree in the above theorem can be of infinite degree. Armed with these results, we now just need one more lemma before we can prove the main theorem. First, note that $\fR_1$ is an induced subgraph of $\MCG{M}{A}$. However, since $\fR_1$ need not be a regular graph in general, there are inside $\MCG{M}{A}$ some edges leaving $\fR_1$. In principle, if this leaving were not in any way controlled, it might be the case that no end-isomorphism inside $\fR_1$ will extend to one between the same vertices inside $\MCG{M}{A}$. On the other hand, if we do know that it is controlled, in the sense that such end-isomorphism do extend, then Proposition~\ref{Prop: (GS) Exactly one edge entering R_m} together with Theorem~\ref{Thm: (GS) M / R is a tree} and Theorem~\ref{Thm: (GS) R_1 is iso to R_m} easily combine to show that $\MCG{M}{A}$ is a context-free graph. For, if $m \in M$, then by Theorem~\ref{Thm: (GS) M / R is a tree} we have that $\MCG{M}{A}(m)$ consists only of vertices $n \in M$ with $n \leq_\mathscr{R} m$, and by Proposition~\ref{Prop: (GS) Exactly one edge entering R_m} together with Theorem~\ref{Thm: (GS) R_1 is iso to R_m}, we must hence have that $\MCG{M}{A}(m) \sim \MCG{M}{A}(m')$, where $m = [m_0] m'$ as above. Hence, since each end-isomorphism between vertices inside $\fR_1$ extends to one between the same vertices in $\MCG{M}{A}$, we therefore would have that there are only finitely many end-isomorphism classes of vertices in $\MCG{M}{A}$. 

We refer the reader to Figure~\ref{Fig: Bicyclic full example} as a bicyclic companion throughout the remainder. The reader is invited to check that the statements made in the following paragraph are all true in this example. Now, to show the claim, let $\fU^{(h)}$ denote the graph obtain from $\fU$ by the following operation: if $u = [(u_0 ; \xi)] \in V(\fU)$ and $a \in A$ are such that there is \textit{no} edge originating in $u$ labelled by $a$, then we add a new vertex $u'$ to $\fU$ and an edge $u \xrightarrow{a} u'$. Adding all such edges produces a deterministic and regular graph which we will call the \textit{hairy graph} associated to $\fU$, and denote by $\fU^{(h)}$. Each added edge will be called an $a$-\textit{hair}, where $a \in A$ is the label of the added edge. Now, by the left action of $U(M)$, it is clear that if an $a$-hair is added from $[(u_0 ; \xi)]$, then an $a$-hair must also be added to $[(v_0 ; \xi)]$ for all $v_0 \in \Lambda^\ast$. Hence our usual action extends, and $\fU^{(h)}$ is an almost-transitive graph with respect to $U(M)$. Furthermore, as these hairs do not cross between $\sim_i$-classes, it is also clear that $\fU^{(h)}$ is $N(\fU)$-overlap-free, and that $(\fU^{(h)}, N(\fU))$ is $S$-full. Thus $(\fU^{(h)}, U(M), N(\fU))$ satisfies the bounded folding condition. Note that the added vertices (the tips of the hairs) are \textit{not} included as branch points. Furthermore, the action of $U(M)$ maps hairs to hairs, and non-hairs to non-hairs. Now, by an argument entirely analogous to that of the proof of Theorem~\ref{Thm: Big fat equivalence list}, $\fU^{(h)}$ is a context-free graph. It thus follows from Theorem~\ref{Thm: Bounded folding gives CF from T} that the determinised form of $\T(\fU^{(h)}, N(\fU))$ is context-free. Let $\TUH$ denote $\T(\fU^{(h)}, N(\fU))$, and let $\TUHE$ denote its determinised form. Now for all $a \in A$ and for all vertices $v \in \TUH$ that does not have degree one and was added as a hair, there is at least one edge outgoing from $v$ with label $a$ -- in general, there is more than one. Hence, determinising $\TUH$ will produce a graph in which one can also read an edge labelled by $a$ from every vertex of the above form. But $\TUHE$ is also an induced subgraph of $\MCG{M}{A}$, being the graph $\fR_1$ with an edge added for ways in which one may leave $\fR_1$. Hence, we have captured all ways in which one can leave $\fR_1$, and shown that the resulting graph is context-free. This yields the main theorem of this paper. 

\begin{theorem}\label{Thm: MULLER-SCHUPP!}
Let $M = \pres{Mon}{A}{R_i = 1 \: (i \in I)}$ be a finitely presented special monoid. Then the Cayley graph $\MCG{M}{A}$ is context-free if and only if $U(M)$ is virtually free.
\end{theorem}
\begin{proof}
For brevity, let $\Gamma_M := \MCG{M}{A}$. By Corollary~\ref{Cor: fR1 CF iff U(M) VF}, it suffices to show that $\Gamma_M$ is context-free if and only if $\fR_1$ is context-free. Now using the same argument from automorphism groups as in the converse direction of the proof of Corollary~\ref{Cor: fR1 CF iff U(M) VF}, we know that $\fR_1$ is context-free if and only if $\TUHE$ is context-free. 

We first claim that if $m \in M$ is such that $m = [m_0]m'$, then $\Gamma_M(m) \sim \Gamma_M(m')$. By Theorem~\ref{Thm: (GS) M / R is a tree} together with Theorem~\ref{Thm: (GS) R_1 is iso to R_m}, we first know that $\Gamma_M(m) \subseteq \Gamma_{[m_0]}$, and furthermore that if $n \in \Gamma_{[m_0]}$, then we can write $n = [m_0][n_0'] n'$. Then, again applying Theorem~\ref{Thm: (GS) M / R is a tree} together with Proposition~\ref{Prop: (GS) Exactly one edge entering R_m}, we know that
\begin{align*}
d_{\Gamma_M}(1, n) &= d_{\Gamma_M}(1, [m_0]) + d_{\Gamma_M}(1, [n_0]) + d_{\Gamma_M}(1, n') \\
&= C_1 + d_{\Gamma_{[m_0]}}(1, [n_0]n').
\end{align*}
Here $C_1$ is some constant not depending on $n$. It now follows by a simple distance argument gives that since $\Gamma_{[m_0]}$ is isomorphic to $\Gamma_M$, we have that the prefix $[m_0]$ does not make a difference for the end-isomorphism type of $m$; that is, $\Gamma_M(m)~\sim~\Gamma_M(m')$. 

Now, assume that $u, v \in V(\TUHE)$ are such that $\TUHE(u) \sim \TUHE(v)$. Then as $\TUHE$ contains all edges leaving $\fR_1$ inside $\Gamma_M$, and since each such edge in $\Gamma_M$ continues with a copy of $\Gamma_M$ attached by Theorem~\ref{Thm: (GS) R_1 is iso to R_m}, we necessarily have that $\Gamma_M(u) \sim \Gamma_M(v)$. 

Thus combining the two above arguments, we have that $\Gamma_M$ and $\TUHE$ must have the same number of end-isomorphism classes of vertices. Thus one is context-free if and only if the other is, completing the proof.
\end{proof}

A full example of all the graphs involved is drawn out in Figure~\ref{Fig: Bicyclic full example}, where the bicyclic monoid is treated. As any context-free graph is quasi-isometric to a tree by \cite[Lemma~8.4]{Chalopin2017}, we have the following corollary, of independent interest.

\begin{corollary}\label{Cor: M qi to a tree iff U(M) VF}
Let $M = \pres{Mon}{A}{R_i = 1 \: (i \in I)}$ be a finitely presented special monoid. Then the right Cayley graph of $M$ is quasi-isometric to a tree if and only if $U(M)$ is virtually free.
\end{corollary}
\begin{proof}
As noted above, if $U(M)$ is virtually free, then $\MCG{M}{A}$ is context-free, and hence quasi-isometric to a tree. On the other hand, if $\MCG{M}{A}$ is quasi-isometric to a tree, then since $M$ has uniformly bounded degree, it follows that $\MCG{M}{A}$ has finite tree-width. Hence, as $\fU$ is an induced subgraph of $\fR_1$, and hence of $\MCG{M}{A}$, $\fU$ also has finite tree-width. Hence $U(M)$ is virtually free by Theorem~\ref{Thm: Big fat equivalence list}.
\end{proof}

This corollary could also be proven by noting that Lemma~\ref{Lem: R1 is qi to fR} reduces the problem of classifying when $\fR_1$ is quasi-isometric to a tree to the problem of classifying when the Cayley graph of $F \ast G$ is quasi-isometric to a tree, where $F$ is a finitely generated free monoid and $G$ is a group. This free product should be relatively straightforward to approach directly to show that this is equivalent to having $G$ be virtually free; the result would then follow by, following Gray-Steinberg, building the Cayley graph of $M$ in a tree-like way from $\fR_1$.

\begin{figure}
\begin{center}
\hspace*{-3cm}\includegraphics[scale=1]{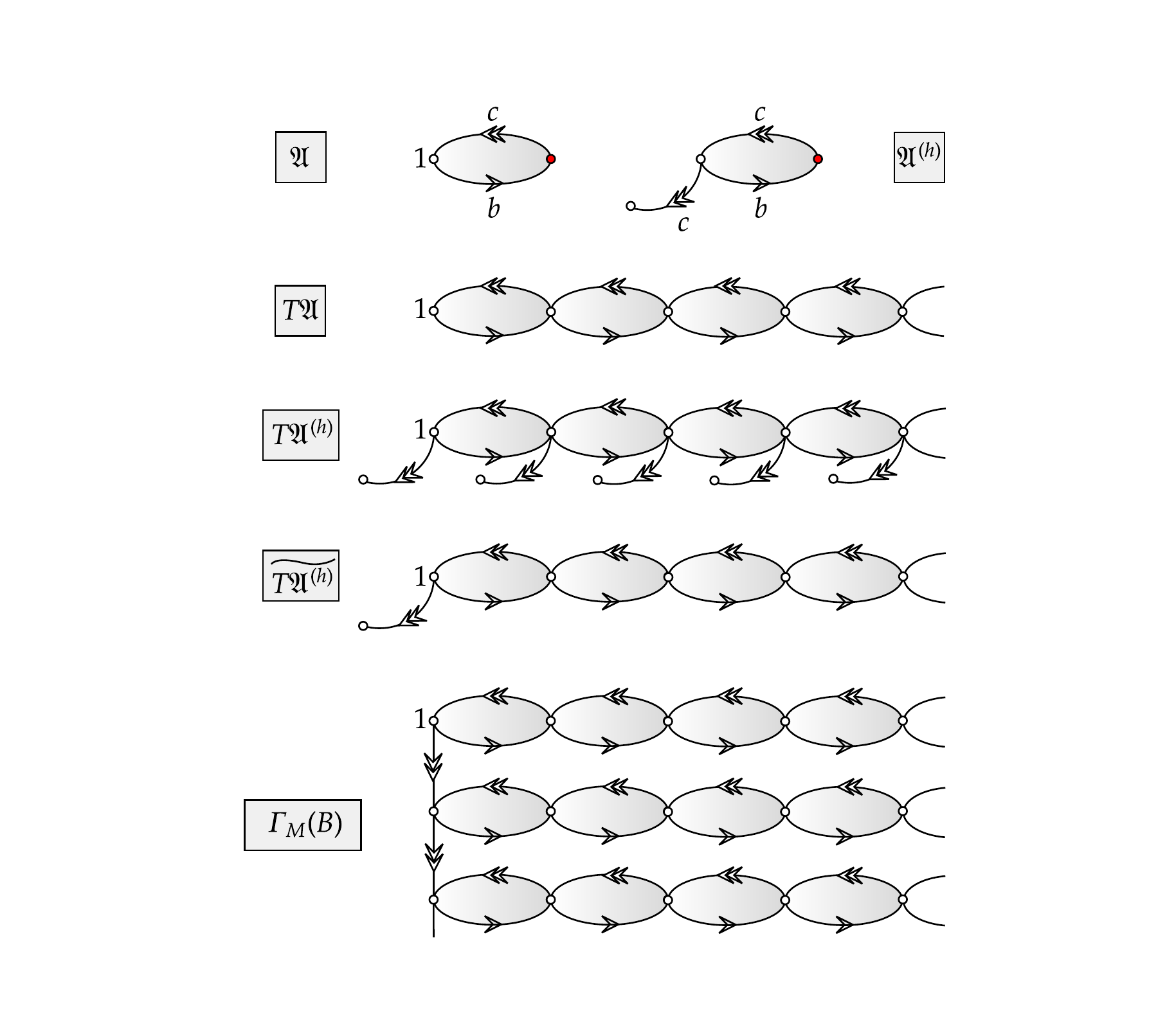}
\end{center}
\caption{All graphs involved in the proof of Theorem~\ref{Thm: MULLER-SCHUPP!} when applied to the bicyclic monoid $B = \pres{Mon}{b,c}{bc=1}$. The single red vertex in $\fU$ and $\fU^{(h)}$ denotes the single non-locally invertible vertex of either graph. Here $\Gamma_M(B)$ denotes $\MCG{B}{\{ b, c\}}$, the right Cayley graph of $B$. }
\label{Fig: Bicyclic full example}
\end{figure}

\section{Applications and Open Problems}\label{Sec: Applications and Open Problems}

We will now describe two main applications of Theorem~\ref{Thm: MULLER-SCHUPP!}, as well as some open problems and further research directions. The first is a complete characterisation of when the monadic second order theory of the Cayley graph of a special monoid is decidable; using this, we present the second application, which is to the rational subset membership problem. 

\subsection{Logic of Cayley graphs} 

Logically, we can consider a labelled graph as a formal logical structure $\Gamma$ with domain $V$ (the vertex set of the graph) and a single relation, the edge relation. A predicate of first-order logic in a graph involves vertices, the edge relation, equality, quantifiers $(\exists, \forall)$, and their boolean combinations $(\neg, \wedge, \vee, \rightarrow)$. Monadic second-order logic also allows quantification (both universal and existential) over subsets of the vertices; if such quantification is only allowed over finite subsets, then this is known as weak monadic second-order logic. 

The \textit{first-order (monadic second-order) theory} of a graph $\Gamma$ is the collection of all first-order (monadic second-order) predicates $\phi$ with no free variables such that $\Gamma \models \phi$. We say that the first-order (monadic second-order) theory of a graph is \textit{decidable} if, given any first-order (monadic second-order) predicate $\phi$, there is an algorithm which decides whether or not $\Gamma \models \phi$. For more detailed background on these notions, see e.g.\ \cite{Muller1985, Kuske2006, Schulz2010}.

Decidability of either the first-order or the monadic second-order theory of the Cayley graph of a f.g. monoid $M$ does not depend on the finite generating set chosen \cite{Kuske2006}. For this reason, we will generally omit reference to finite generating set below. There is a number of connections between decision problems for a given monoid $M$ and different theories associated to the Cayley graph of $M$. This is most apparent in the group case: the first-order theory of the Cayley graph of a group is decidable if and only if the word problem for the group is decidable \cite{Kuske2005}, and the monadic second-order theory of the Cayley graph of a group is decidable if and only if the group is virtually free \cite{Kuske2005, Muller1985}. 

However, the same need not be true for monoids in general; we only have implications in one direction. By \cite[Proposition~4]{Kuske2006}, if the first-order theory of the Cayley graph of a f.g. monoid is decidable, then the monoid has decidable word problem, but by \cite[Proposition~5]{Kuske2006} there exists a monoid with word problem decidable even in linear time, but the Cayley graph of which nonetheless has undecidable first-order theory. One of the main consequences of Theorem~\ref{Thm: MULLER-SCHUPP!} is the following, which lends credence to the adage that special monoids are, in many respects, very similarly behaved to groups.

\begin{theorem}\label{Thm: Dec MSO iff U(M) VF}
Let $M = \pres{Mon}{A}{R_i = 1 \: (i \in I)}$ be a finitely presented special monoid and $U(M)$ its group of units. Then the Cayley graph $\MCG{M}{A}$ has decidable monadic second order theory if and only if $U(M)$ is virtually free. 
\end{theorem}
\begin{proof}
$(\impliedby)$ If $U(M)$ is virtually free, then by Theorem~\ref{Thm: Big fat equivalence list} $\fU$ is context-free. Hence $\MCG{M}{A}$ is context-free by Theorem~\ref{Thm: MULLER-SCHUPP!}. By \cite[Theorem~4.4]{Muller1985}, the monadic second-order theory of $\MCG{M}{A}$ is hence decidable.

$(\implies)$ Assume that $\MCG{M}{A}$ has decidable monadic second order theory. It is shown in \cite{Kuske2003} that the class of finitely generated monoids with decidable monadic second order theory is closed under taking finitely generated submonoids. Hence $\MCG{U_\fB(M)}{\fB}$ has decidable monadic second order theory. Now by \cite[Corollary~4.1]{Kuske2005} together with Proposition~\ref{Prop: lud MCG is iso to GCG}, we have that the monoid Cayley graph of a group has decidable monadic second order theory if and only if $U(M)$ is context-free, and hence if and only if $U(M)$ is virtually free. 
\end{proof}

The question of characterising in general which monoids have decidable monadic second order theory was posed by Kuske \& Lohrey in \cite{Kuske2006}. While the fully general case remains (wide) open, the above theorem completely answers this question in the special case.

\subsection{Rational Subsets}

Our second main application lies in the rational subset membership problem. Recall that if $\pi \colon A^\ast \to M$ denotes the natural homomorphism, and $R \subseteq A^\ast$ is a regular language (in the sense of Kleene \cite{Kleene1951}), then we say that $\pi(R)$ is a \textit{rational subset} of $M$. We note that, just as Kleene's theorem tells us that a regular language is one that is accepted by a finite state automaton over $A$, so too is $\pi(R)$ accepted by a finite state automaton over $\pi(A)$. The most familiar examples of rational subsets are finitely generated submonoids. However, there are many rational subsets which are not submonoids. Any finite union of finitely generated submonoids, for example, will be rational, but will not itself be a submonoid in general.

The rational subset membership problem is the decision problem which asks: given $\pres{Mon}{A}{R}$, a regular language $R \subseteq A^\ast$, and a word $w \in A^\ast$; decide whether $\pi(w) \in \pi(R)$. For groups, this decision problem generalises the word problem, for deciding membership in the rational subset $\{ 1 \}$ is equivalent to the word problem. In general, however, there is little connection between the two problems; the two problems are independent. For elaboration on the connection with the word problem, we refer the reader to \cite{NybergBrodda2020b}.

One of the main appeals of decidability of the monadic second-order theory of the Cayley graph of a monoid is that it can be used to solve decision problems for the monoid itself, independently of any graph-theoretic considerations. This is illustrated by the following.

\begin{prop}
Let $M$ be a monoid generated by a finite set $A$. If the Cayley graph $\MCG{M}{A}$ has decidable monadic second-order theory, then $M$ has decidable rational subset membership problem.
\end{prop}
\begin{proof}
Let $L \subseteq A^\ast$ be a regular language. Let $\textbf{reach}_{L}(x, y)$ be the predicate with variables $x, y \in M$ denoting whether the vertex $y$ can be reached from $x$ by some sequence of edges $E_{\sigma_1}, \dots, E_{\sigma_n}$ such that $E_{\sigma_i}$ is an edge labelled by $\sigma_i \in A$, and $\sigma_1 \cdots \sigma_n \in L$. Clearly, for a word $w \in A^\ast$, we have that $\pi(w) \in \pi(L)$ if and only if $\MCG{M}{A} \models \textbf{reach}_{L}(1, \pi(w))$. Thus, to establish the claim, it suffices to show that $\textbf{reach}_{L}(x, y)$ is a monadic second-order predicate. In fact, it is even a weak monadic second-order predicate; this can be shown (see e.g.\ \cite{Schulz2010}) by induction on the operations of a regular expression for $L$.
\end{proof}
\begin{remark}
We do not require the full monadic second-order logic in the above proposition. The fragment, called $\text{FO}(\text{Reg})$ in \cite{Schulz2010}, consisting of first-order logic together with the reachability predicates $\textbf{reach}_{L}(x, y)$, is obviously sufficient; that it is not all of monadic second-order logic is not as obvious.
\end{remark}

Of course, an immediate corollary of the above is the following. 

\begin{corollary}\label{Cor: VF U(M) => RSMP dec}
Let $M = \pres{Mon}{A}{R_i = 1 \: (i \in I)}$ be a finitely presented special monoid with virtually free group of units. Then the rational subset membership problem for $M$ is decidable. 
\end{corollary}

Prior to the above corollary, few results on the rational subset membership problem for (non-group) monoids or semigroups were known. In fact, we will mention all such results; these were all shown by Kambites \& Render, and can be found collected in \cite{Lohrey2015}. Given a semigroup with decidable rational subset membership, then a finitely generated Rees matrix semigroup has decidable rational subset membership problem \cite{Render2010}; monoids which satisfy certain small overlap conditions (none of which are satisfied by special monoids) \cite{Kambites2009}; polycyclic and bicyclic monoids \cite{Render2009}. In fact, their result on bicyclic and polycyclic monoids is a consequence of a more general result which follows from material found in \cite{Book1982b}.

\subsection{Characterising context-free special monoids}

In this section we will combine the various results of this article together with other work by the author -- a generalisation of the Muller-Schupp theorem for groups to special monoids -- in order to demonstrate the rigidity of the class of special monoids with virtually free group of units. In particular, we will show that from many viewpoints, including that of undirected geometry and logic, this class behaves very similarly to the class of virtually free groups.

We begin with a definition; this is elaborated on in great detail in \cite{NybergBrodda2020b}. The \textit{word problem} for a group $G$ with group generating set $A$ is commonly interpreted as the set of words over $A \cup A^{-1}$ which equal the identity element of the group. This set of words can be studied using language-theoretic methods, which has seen a great deal of success; see the introduction for a brief overview. For monoids in general, this set has very little meaning; a more appropriate definition was introduced by Duncan \& Gilman in \cite{Duncan2004}. If $M$ is a monoid generated by a set $A$ with associated homomorphism $\pi \colon A^\ast \to M$, and $\#$ is a symbol not in $A$, then we define the \textit{word problem} of $M$ as the set 
\[
WP(M, A) := \{ u \# v^{\text{rev}} \mid \pi(u) = \pi(v) \}
\]
where $v^{\text{rev}}$ denotes the reverse of the word $v$. We will not discuss the technicalities involved in this definition further, but note that this definition coincides with the usual one if $M$ is a group, and that the language (cone) class it belongs to does not depend on generating set $A$ chosen. The author has shown in related work that a finitely presented special monoid has context-free word problem if and only if its group of units is virtually free \cite{NybergBrodda2020b}. This is a generalisation of the Muller-Schupp theorem from groups to special monoids -- as groups are special monoids, and the group of units of a group is the group itself, it is a proper generalisation. 

Thus, we can gather up all results of the previous sections, in particular Theorems~\ref{Thm: MULLER-SCHUPP!} and \ref{Thm: Dec MSO iff U(M) VF} along with Corollary~\ref{Cor: M qi to a tree iff U(M) VF}. This provides ample justification for claiming that the class of special monoids with virtually free units is incredibly well-behaved, especially when compared to the class of virtually free groups.

\begin{theorem}\label{Thm: Full equivalence of context-free M}
Let $M$ be a finitely presented special monoid with group of units $U(M)$. Then the following are equivalent:
\begin{enumerate}
\item $M$ has context-free word problem.
\item $U(M)$ is virtually free.
\item The right Cayley graph $\MCG{M}{A}$ is a context-free graph.
\item $\MCG{M}{A}$ is quasi-isometric to a tree (as undirected graphs).
\item The monadic second-order theory of $\MCG{M}{A}$ is decidable.
\end{enumerate}
\end{theorem}

We remark that if we take $M$ to be a finitely presented group in the statement of Theorem~\ref{Thm: Full equivalence of context-free M}, then the results are all well-known to hold, and references to all can be found -- and indeed form crucial parts of the proofs -- in the previous sections in which the above are shown for special monoids.

As a final note, all of the results stated to this point have assumed that $M$ is finitely presented. On the one hand, it is well-known that a group with context-free word problem is finitely presented. On the other hand, for general semigroups, however, this need not be true \cite[Proposition~9]{Hoffmann2012}. It is perhaps surprising that, despite their similarity to groups, special monoids can exhibit this behaviour, too.

\begin{prop}
There exists a non-finitely presentable context-free special monoid.
\end{prop}
\begin{proof}
Let $M = \pres{Mon}{a,b,c}{ab^ic = 1 \: (i \geq 1)}$. Then $M$ admits the rewriting system with rules $\{ (ab^ic \to \varepsilon) \mid i \geq 1 \}$. This system is terminating, as it is special, and furthermore it is locally confluent. By Newman's Lemma, it is complete. Evidently, no finite subset of these rules suffices to define all rules, as no rule can be applied to any other; thus $M$ is not finitely presentable with respect to $\{ a, b, c\}$, and hence is not finitely presentable with respect to any finite generating set.

As the language $ab^+c$ is regular (and hence context-free), $M$ admits a monadic context-free complete rewriting system. Hence $M$ has context-free word problem by \cite[Corollary~3.8]{Book1982}.
\end{proof}

If $M = \pres{Mon}{a,b,c}{ab^ic = 1 \: (i \geq 1)}$, then investigating the right Cayley graph of the monoid with respect to this generating set seems feasible. We conjecture that this graph is context-free. 

There are many more directions to pursue regarding the geometry of special monoids. In particular, connections with word-hyperbolicity, the fellow traveller property, automaticity, and more can be investigated; these results will appear in forthcoming work by the author. 

\section*{Acknowledgements}

The research in this article was conducted while the author was a PhD student at the University of East Anglia, Norwich, UK. The author wishes to thank his supervisor Dr Robert D. Gray for many helpful discussions and comments invaluable to the development of the present document.

\bibliography{GeometryOfSpecialMonoids} 
\bibliographystyle{amsalpha}

\end{document}